\documentclass{article}

\usepackage{microtype}
\usepackage{graphicx}
\usepackage{subfigure}
\usepackage{booktabs} 
\usepackage{colortbl}
\usepackage{bbding}
\usepackage{graphicx}
\usepackage{epstopdf}
\newcommand*{\belowrulesepcolor}[1]{%
	\noalign{%
		\kern-\belowrulesep
		\begingroup
		\color{#1}%
		\hrule height\belowrulesep
		\endgroup
	}%
}
\newcommand*{\aboverulesepcolor}[1]{%
	\noalign{%
		\begingroup
		\color{#1}%
		\hrule height\aboverulesep
		\endgroup
		\kern-\aboverulesep
	}%
}
\usepackage{hyperref}
\usepackage{tcolorbox}
\usepackage{amsmath}
\usepackage{amsthm}
\usepackage{amsmath, bm}
\usepackage{amsfonts}
\usepackage{multirow}
\usepackage{cleveref}
\usepackage{tabularx}
\usepackage{setspace}
\usepackage[flushleft]{threeparttable}
\usepackage{multicol}

\newtheorem{theorem}{Theorem}[section]
\newtheorem{corollary}{Corollary}[section]
\newtheorem{lemma}{Lemma}[section]

\newtheorem{definition}{Definition}[section]

\newtheorem{remark}{Remark}[section]

\newtheorem{assumption}{Assumption}[section]

\newcommand{\argmin}{\text{argmin}}
\newcommand{\mK}{\mathcal{K}}
\newcommand{\mKb}{\mathcal{K}-1}
\newcommand{\mE}{\mathcal{E}}
\newcommand{\M}{\mathcal{M}}
\newcommand{\D}{\operatorname{D}}
\newcommand{\retr}{\operatorname{Retr}}
\newcommand{\Exp}{\operatorname{Exp}}

\newcommand{\grad}{\operatorname{grad}}
\newcommand{\Hess}{\operatorname{Hess}}
\newcommand{\Uniform}{\operatorname{Uni}}
\newcommand{\T}{\operatorname{T}}
\newcommand{\norm}[1]{\|{#1}\|}
\newcommand{\inner}[2]{\langle{#1},{#2}\rangle}
\newcommand{\TSS}{\textbf{TSS}}
\newcommand{\TSSA}{\textbf{TSSA}}
\newcommand{\abs}[1]{\ensuremath{\left|#1\right|}}

\usepackage[accepted]{icml2024}

\icmltitlerunning{Riemannian Accelerated Zeroth-order Algorithm: \\ Improved Robustness and Lower Query Complexity}

\begin{document}

\twocolumn[
\icmltitle{Riemannian Accelerated Zeroth-order Algorithm: \\ Improved Robustness and Lower Query Complexity}




\begin{icmlauthorlist}
\icmlauthor{Chang He}{sufe}
\icmlauthor{Zhaoye Pan}{sufe}
\icmlauthor{Xiao Wang}{sufe}
\icmlauthor{Bo Jiang}{sufe}
\end{icmlauthorlist}

\icmlaffiliation{sufe}{School of Information Management and Engineering, Shanghai University of Finance and Economics, China}

\icmlcorrespondingauthor{Chang He}{ischanghe@gmail.com}

\icmlkeywords{Machine Learning, ICML}

\vskip 0.3in
]



\printAffiliationsAndNotice{}  

\begin{abstract}
Optimization problems with access to only zeroth-order information of the objective function on Riemannian manifolds arise in various applications, spanning from statistical learning to robot learning. While various zeroth-order algorithms have been proposed in Euclidean space, they are not inherently designed to handle the challenging constraints imposed by Riemannian manifolds. The proper adaptation of zeroth-order techniques to Riemannian manifolds remained unknown until the pioneering work of \cite{li2023stochastic}. However, zeroth-order algorithms are widely observed to converge slowly and be unstable in practice. To alleviate these issues, we propose a Riemannian accelerated zeroth-order algorithm with improved robustness. Regarding efficiency, our accelerated algorithm has the function query complexity of $\mathcal{O}(\epsilon^{-7/4}d)$ for finding an $\epsilon$-approximate first-order stationary point. By introducing a small perturbation, it exhibits a function query complexity of $\tilde{\mathcal{O}}(\epsilon^{-7/4}d)$ for seeking a second-order stationary point with a high probability, matching state-of-the-art result in Euclidean space. Moreover, we further establish the almost sure convergence in the asymptotic sense through the Stable Manifold Theorem. Regarding robustness, our algorithm requires larger smoothing parameters in the order of $\tilde{\mathcal{O}}(\epsilon^{7/8}d^{-1/2})$, improving the existing result by a factor of $\tilde{\mathcal{O}}(\epsilon^{3/4})$.
\end{abstract}

\section{Introduction}

Many machine learning problems frequently encounter situations where computing function gradients is costly or even infeasible. For instance, the tasks such as optimal linear combination prediction \cite{das2022estimating} and Bayesian optimization in robot learning \cite{jaquier2018geometry,jaquier2020bayesian} involve objective functions, lacking analytical forms, only observable through point-wise evaluations. Furthermore, the design space of interest is also complicated, involving constraints such as the unit sphere, probability simplex, and positive definite matrices. The limited function information and inherent constraints render these problems challenging to solve. One potent strategy for dealing with these constraints is re-expressing them through the lens of \textit{Riemannian manifolds} \cite{absil2009optimization,boumal2023introduction}. Mathematically, we can formulate the problem in consideration as follows:
\begin{equation}\label{eq.main}
\min_{x \in \M} \quad f(x),
\end{equation}
where $\M$ represents the Riemannian manifold, and $f(\cdot)$ is a \textit{nonconvex} objective function with only zeroth-order information (i.e.\ function value) available. 
For ease of discussion, we assume $f(\cdot)$ is lower bounded, i.e. $f(x) \ge f_{\operatorname{low}}$ for all $x \in \M$.
Recently, a pioneering work by \citet{li2023stochastic} introduced several Riemannian \textit{zeroth-order} algorithms to tackle problem \eqref{eq.main}, relying solely on the query of function values. 
It is well known that
the function query complexity is a key to measure the efficiency of the zeroth-order algorithms,
whereas these algorithms only exhibit inferior complexity to the one in Euclidean space.
This raises a natural question: \textit{Is it possible to develop a Riemannian accelerated zeroth-order algorithm with lower function query complexity?}

\renewcommand{\arraystretch}{1.3} 
\definecolor{LightCyan}{rgb}{0.9,1,0.9}
\begin{table*}[t] 
    \label{tab:convergence_comparison}
	\small
	\centering 
	\caption{Comparison of zeroth-order algorithms in terms of the ability to handle Riemannian manifolds, value of smoothing parameter, and function query complexity for nonconvex objective function. The symbol $\dagger$ is used to indicate that this algorithm converges to $\epsilon$-approximate first-order stationary points; otherwise, it converges to $\epsilon$-approximate second-order stationary points.}
	\vspace{0.2cm}
	\begin{tabular}{cccc} \toprule
		{Algorithms} &Riemanian Manifolds&  Smoothing parameter $\mu$ & Function query complexity \\   \midrule
		PAGD \cite{vlatakis2019efficiently}& \XSolidBrush   &  $\mathcal{O}\left(\frac{\epsilon^{3/2}}{\sqrt{d}}\right)$   &$\tilde{\mathcal{O}}\left(\frac{d}{\epsilon^2}\right)$
		\\  \midrule
  		ZO-GD \cite{bai2020escaping}& \XSolidBrush   &  $\mathcal{O}\left(\frac{\epsilon^3}{d^2}\right)$   &$\tilde{\mathcal{O}}\left(\frac{d^2}{\epsilon^{8}}\right)$
        \\   
      \midrule ZO-GD-NCF \cite{zhang2022zeroth}&\XSolidBrush  &  $\mathcal{O}\left(\frac{\epsilon^{1/2}}{d^{1/4}}\right)$    &$\tilde{\mathcal{O}}\left(\frac{d}{\epsilon^2} \right)$
    		\\  \midrule 
		ZO-PAGD \cite{zhang2022faster} &\XSolidBrush  &  $\tilde{\mathcal{O}}\left(\frac{\epsilon^{13/8}}{\sqrt{d}}\right)$   &$\tilde{\mathcal{O}}\left(\frac{d}{\epsilon^{7/4}}\right)$
        \\  \midrule ZOPGD \cite{ren2023escaping} &\XSolidBrush  &  $\tilde{\mathcal{O}}\left(\frac{\epsilon^{1/2}}{d}\right)$    &$\tilde{\mathcal{O}}\left(\frac{d}{\epsilon^2} \right)$
  		\\  \midrule 
		ZO-RGD \cite{li2023stochastic} &\Checkmark  &  $\mathcal{O}\left(\frac{\epsilon}{d^{3/2}}\right)$   &$\mathcal{O}\left(\frac{d}{\epsilon^{2}}\right)^\dagger$
        \\  \midrule
		\belowrulesepcolor{LightCyan}
		\rowcolor{LightCyan}
		RAZGD with Option I (ours) &\Checkmark & $\mathcal{O}\left(\frac{\epsilon^{5/8}}{d^{1/4}}\right)$&     \textcolor{red}{$\mathcal{O}\left(\frac{d}{\epsilon^{7/4}}\right)^\dagger$}    \\    
		\rowcolor{LightCyan}
		Perturbed RAZGD with Option I (ours) &\Checkmark & \textcolor{red}{$\tilde{\mathcal{O}}\left(\frac{\epsilon^{7/8}}{\sqrt{d}}\right)$} &     $\tilde{\mathcal{O}}\left(\frac{d}{\epsilon^{7/4}}\right)$   \\  
		\aboverulesepcolor{LightCyan}  \bottomrule
	\end{tabular} 
	\label{tas1}
	\vspace{-0.1cm}
\end{table*} 

The development of accelerated algorithms is a prominent and active topic within both machine learning and optimization communities. It traces back to the seminal breakthrough by \citet{nesterov1983method}, which paved the way for subsequent advancements in acceleration techniques. Since then, numerous fruitful results have emerged in various scenarios, such as accelerated first-order algorithms \cite{beck2009fast,lin2015universal,carmon2017convex,carmon2018accelerated,jin2018accelerated,li2022restarted} and accelerated second-order algorithms \cite{nesterov2008accelerating,bubeck2019near,jiang2021optimal}. Notably, \citet{zhang2022faster} demonstrated that zeroth-order algorithms can also benefit from acceleration and exhibit an improved complexity. Regarding the optimization problem over Riemannian manifolds, there has also been a growing interest in developing accelerated algorithms \cite{liu2017accelerated,zhang2018towards,criscitiello2022accelerated}, to name a few. Due to the space limitation, a detailed discussion is deferred to Appendix \ref{sec.related work}. Despite significant efforts in designing accelerated algorithms, none of them is applicable to problem \eqref{eq.main}.


To design an algorithm in a gradient-free manner, constructing zeroth-order estimators through function value evaluations becomes necessary. The accuracy of this approximation is tied to the \textit{smoothing parameter} (see Definition \ref{def.ZO estimators}, for example). Although the smaller value of the parameter improves the precision, it may also introduce instability in practical applications \cite{lian2016comprehensive,liu2018zeroth,liu2020primer}. Regrettably, integrating acceleration techniques into zeroth-order algorithms \cite{zhang2022faster} requires smaller smoothing parameters compared to the standard ones \cite{vlatakis2019efficiently,zhang2022zeroth}. In response to these challenges, we introduce a novel \textit{Riemannian accelerated zeroth-order algorithm}. Surprisingly, while maintaining the same function query complexity, our algorithm allows the
use of a larger smoothing parameter, compared to the Euclidean counterpart \cite{zhang2022faster}.
This, in turn, ensures the robust and stable performance of our accelerated zeroth-order algorithm. 

\paragraph{Contributions.} In this paper, we delve into a comprehensive study of Riemannian zeroth-order optimization. Our main contributions are given as follows:
\begin{itemize}
    \item By leveraging the basis of the tangent space, we extend the classical finite-difference gradient approximation to Riemannian manifolds (Definition \ref{def.ZO estimators}). Based on this estimator, we develop a Riemannian accelerated zeroth-order gradient descent (RAZGD) in Algorithm \ref{alg.AZO}, which alternates between the Riemannian zeroth-order gradient descent step (Subroutine \ref{sub.RZGD}) and the tangent space step (Subroutine \ref{sub.TSS} and \ref{sub.TSSA}). 
    
    \item Under some mild assumptions and by setting the initial point as zero in the tangent space step (Subroutine \ref{sub.TSS}), we prove that the RAZGD with Option I has the function query complexity of $\mathcal{O}(\epsilon^{-7/4}d)$ for finding a Riemannian $\epsilon$-approximate first-order stationary point, which improves the existing result by a factor of $O(\epsilon^{-1/4})$ in \cite{li2023stochastic}. For a fair comparison, we present selected zeroth-order algorithms in Table \ref{tab:convergence_comparison}.

    \item By introducing a small perturbation to the initial point in the tangent space step (Subroutine \ref{sub.TSS}), the perturbed RAZGD with Option I seeks a second-order stationary point with a high probability under $\tilde{\mathcal{O}}(\epsilon^{-7/4}d)$ query complexity guarantee, matching state-of-the-art complexity in Euclidean zeroth-order optimization \cite{zhang2022faster}. To get an almost sure convergence result, we further prove that the RAZGD with Option II converges to strict Riemannian second-order stationary points gradually. 
    
    \item Beyond the function query complexity, the perturbed RAZGD with Option I showcases resilience in choosing the smoothing parameter—an essential factor ensuring the robustness of zeroth-order algorithms. With the same function query complexity guarantee, we establish that the RAZGD only requires the smoothing parameter $\mu = \tilde{\mathcal{O}}(\epsilon^{7/8}d^{-1/2})$ for seeking $\epsilon$-approximate second-order stationary points, sharpening the existing best result $\tilde{\mathcal{O}}(\epsilon^{13/8}d^{-1/2})$ in \citet{zhang2022faster}. 

\end{itemize}

\section{Preliminaries: Optimization over manifolds}
In this section, we present the basic setup and mathematical tools for optimization over manifolds. For more details, we refer readers to see \cite{absil2009optimization,boumal2023introduction}. Throughout this paper, we use the convention $\mathcal{O}(\cdot)$ and $\Omega(\cdot)$ to denote lower and upper bounds with a universal constant, respectively. $\tilde{\mathcal{O}}(\cdot)$ ignores the polylogarithmic terms. We use $d$ to denote both the dimension of the Riemannian manifold $\M$ (i.e., $\operatorname{dim}(\M) = d$) and the dimension of the Euclidean space $\mathbb{R}^d$.

A $d$-dimensional manifold $\M$ is a topological space where each point has a neighborhood homomorphic to $d$-dimensional Euclidean space, as illustrated in Figure \ref{fig:coordinate}. A Riemannian manifold $\M$ is a real, smooth manifold equipped with a Riemannian metric. Each $x \in \M$ is associated with a $d$-dimensional real vector space $\T_x\M$, referred to as the tangent space at $x$. The Riemannian metric defines an inner product $\inner{\cdot}{\cdot}_x$ on the tangent space $\T_x\M$. The inner metric induces a corresponding norm $\|\cdot\|_x$. We denote these by $\inner{\cdot}{\cdot}$ and $\norm{\cdot}$ when there is no confusion for $x$ from the context. A vector in the tangent space is known as a tangent vector. The set of pairs $(x, s_x)$ for $x\in\M, s_x\in\T_x\M$ is called the tangent bundle $\T\M$. On the tangent space, we define $\mathbb{B}_{x,r}(s) = \{z \in\T_x\M: \norm{z-s}_x \leq r\}$, representing the closed ball of radius $r$ centered at $s \in \T_x\M$. Then we use $\Uniform(\mathbb{B}_{x,r}(s))$ to define the uniform distribution over the ball $\mathbb{B}_{x,r}(s)$.
\begin{figure}[H]
    \centering
    \includegraphics[scale=0.25]{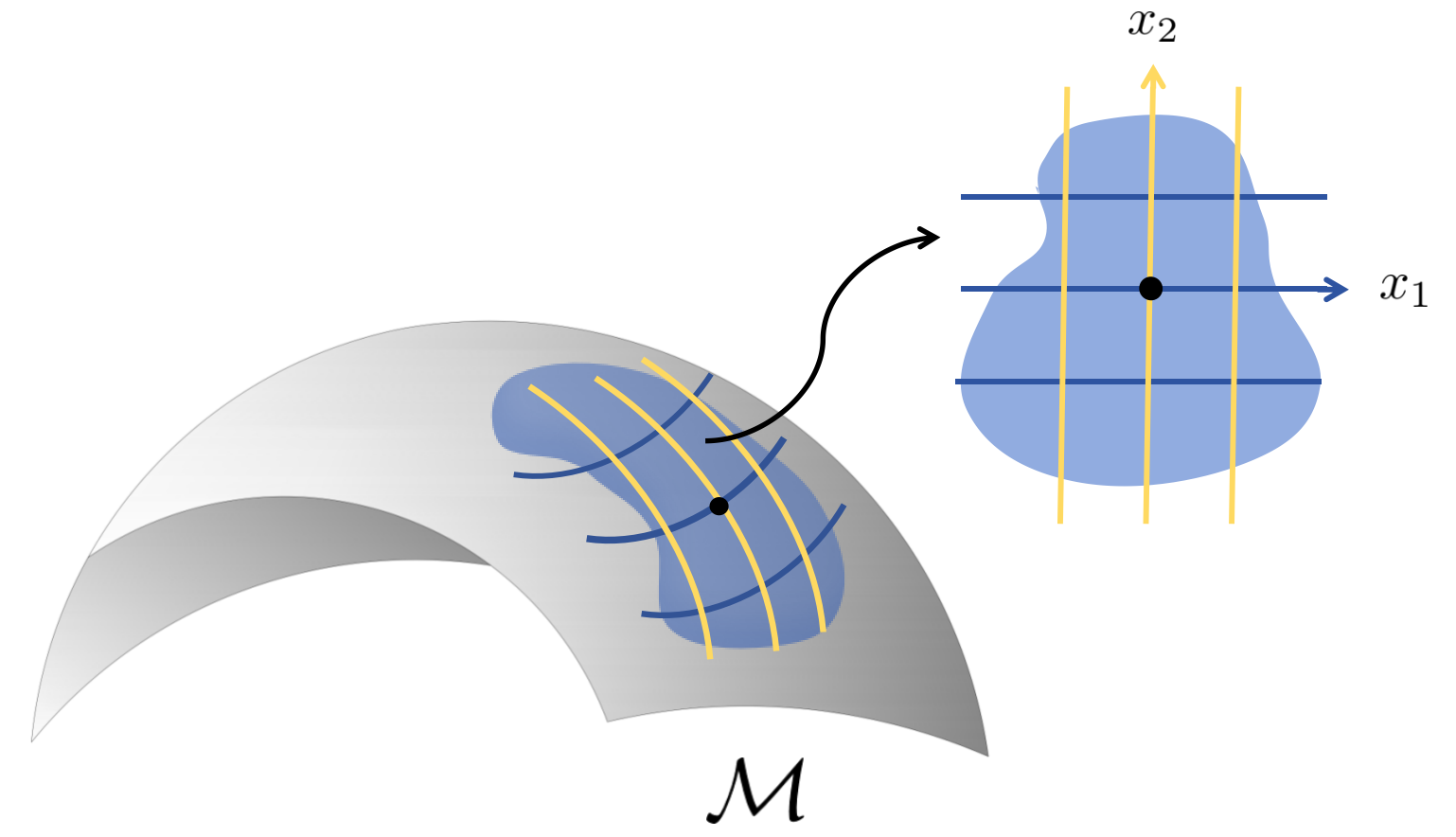}
    \caption{A 2-dimensional manifold}
    \label{fig:coordinate}
\end{figure}
Given a smooth function $f(\cdot)$, the Riemannian gradient $\grad f(x)$ of $f$ at $x \in \M$ is the unique vector in $\T_x\M$ that satisfies $\D f(x)[s] = \inner{\grad f(x)}{s}_x$ for all $s \in \T_x\M$, where $\D f(x)[s]$ is the directional derivative of $f$ at $x$ along $s$. The Riemannian metric gives rise to a well-defined notion of the derivative of vector fields, known as the Levi--Civita connection $\nabla$. The Riemannian Hessian of $f$ is the derivative of the gradient vector field: $\Hess f(x)[u] = \nabla{u}\grad f(x)$, which is a symmetric linear operator on $\T_x\M$. For the smooth curve $\gamma: [0,1] \rightarrow \M$, the velocity of the curve is defined as $\frac{d \gamma}{d t}=\gamma^{\prime}(t)$. The intrinsic acceleration $\gamma^{\prime \prime}$ of $\gamma$ is the covariant derivative of the velocity of $\gamma^\prime$: $\gamma^{\prime \prime}=\frac{\mathrm{D}}{d t} \gamma^{\prime}$ induced by the Levi–Civita connection.

We proceed to introduce the $\epsilon$-approximate stationary point on Riemannian manifolds.
\begin{definition}\label{def.approximate stationary points}
    For any $\epsilon > 0$, a point $x \in \M$ is an $\epsilon$-approximate Riemannian first-order stationary point (RFOSP) of the smooth function $f(\cdot)$ if it satisfies $\|\grad f(x)\| \le \mathcal{O}(\epsilon)$. Furthermore, if it additionally satisfies $\lambda_{\min}\left(\Hess f(x)\right) \ge \Omega(-\sqrt{\epsilon})$, then $x$ is an $\epsilon$-approximate Riemannian second-order stationary point (RSOSP), where $\lambda_{\min}(\cdot)$ denotes the the smallest eigenvalue of the symmetric operator.
\end{definition}
We also present the definition of strict Riemannian saddle points and second-order stationary points:
\begin{definition}\label{def.strict stationary points}
A point $x \in \M$ is a strict Riemannian saddle point of the smooth function $f(\cdot)$ if it satisfies $\grad f(x) = 0$ and $\lambda_{\min} \left(\Hess f(x)\right) < 0$. Otherwise, it is a strict Riemannian second-order stationary point when $\grad f(x) = 0$ and $\lambda_{\min} \left(\Hess f(x)\right) \ge 0$.
\end{definition}

To optimize over Riemannian manifolds, a key concept is the retraction (Figure \ref{fig:retraction})—a mapping enabling movement along the manifold from a point $x$ in the direction of a tangent vector $s \in \T_x\M$. This is formalized as follows:
\begin{definition}\label{def.retr}
A retraction mapping $\retr_x \colon \T_x\M \rightarrow \M$ is a smooth mapping satisfies $\retr_x(0) = x$, where $0$ is the zero vector in $\T_x\M$. Moreover, for $x \in \M$ and $s \in \T_x\M$, let
\begin{align*}
	T_{x,s} = \D\retr_x(s) \colon \T_x\M \to \T_{\retr_x(s)}\M
\end{align*}
denote the differential of $\retr_x$ at $s$ (a linear operator). The differential of $\retr_x$ at $0$, i.e. $T_{x,0}$, is the identity map.
\end{definition}
For instance, we employ $\retr_x(s) = x+s$ when $\M = \mathbb{R}^d$. On the unit sphere $\mathbb{S}^{d-1}:=\left\{x \in \mathbb{R}^d: \|x\|_2 = 1\right\}$, the retraction mapping is typically defined as $\retr_x(s) = (x+s)/\norm{x+s}_2$.
\begin{figure}[H]
    \centering
    \includegraphics[scale=0.25]{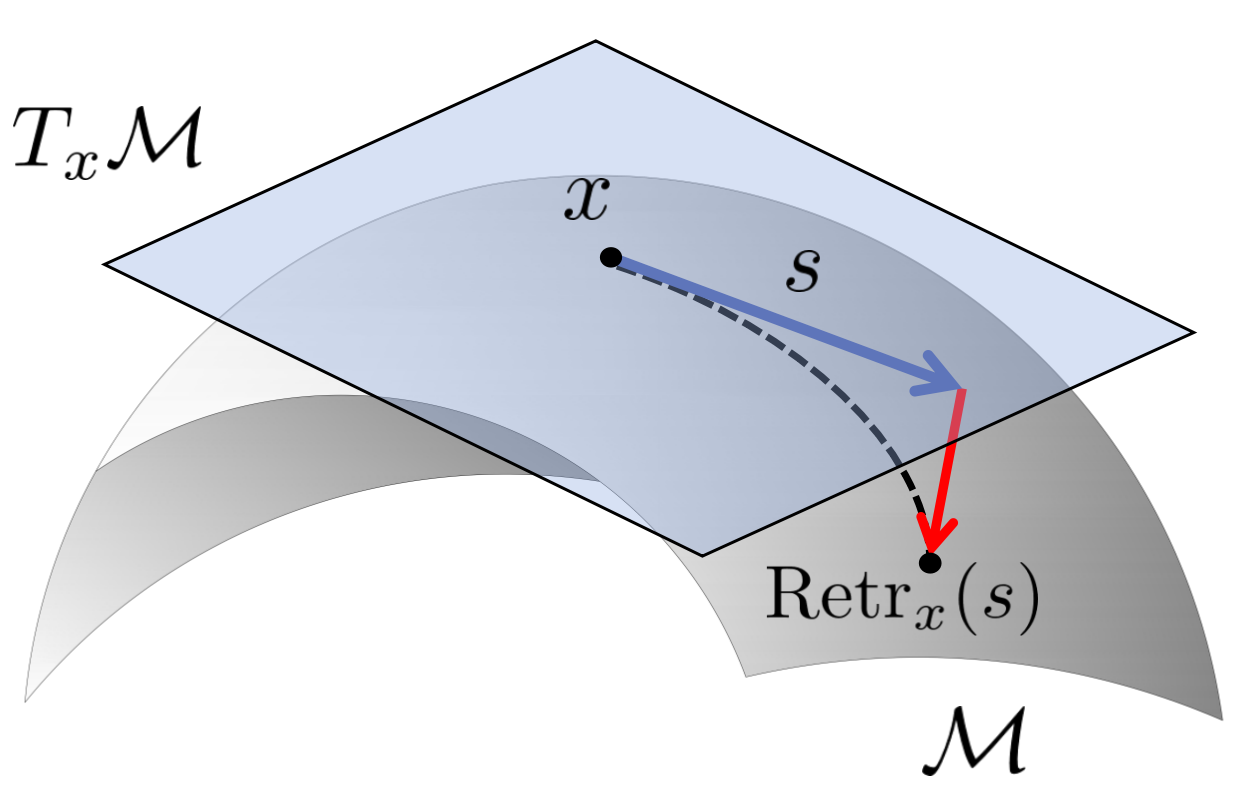}
    \caption{Retraction}
    \label{fig:retraction}
\end{figure}
Following a similar manner in \cite{criscitiello2019efficiently,criscitiello2022accelerated}, in this paper, our analysis is based on the pullback function defined as follows.
\begin{definition}\label{def.pullback}
For any $x \in \M$, the pullback function $\hat f_x(\cdot)$ is a composite function of $f$ and the retraction mapping, that is
\begin{equation*}
	\hat{f}_x = f \circ \retr_x \colon \T_x\M \rightarrow \mathbb{R}.
\end{equation*}
Specifically, as the differential of $\retr_x$ at $0$ is the identity map, it implies that 
$$
\hat f_x(0) = f(x).
$$
\end{definition}
Note that the pullback function $\hat f_x(\cdot)$ is a real function defined on the tangent space $\T_x\M$, which is locally homomorphic to Euclidean space. With a slight abuse of notation, we can define the usual gradient and Hessian of $\hat f_x(\cdot)$ as $\nabla \hat f_x(\cdot)$ and $\nabla^2 \hat f_x(\cdot)$ (mind the overloaded notation of Levi–Civita connection $\nabla$), respectively.

\section{Riemannian Accelerated Zeroth-order Gradient Descent Algorithm}\label{sec.method}
\subsection{Review of Riemannian gradient descent algorithm}
We begin with an ideal situation in which the gradient information is feasible, and consequently the simplest Riemannian gradient descent \cite{boumal2019global}
\begin{equation*}
    x_{t+1} = \retr_{x_t}(-\eta_t \grad f(x_t)), ~t = 0, 1, \ldots.
\end{equation*}
is applicable to  problem \eqref{eq.main}. For the nonconvex objective function, the basic idea behind the convergence analysis of Riemannian gradient descent revolves around a two-case discussion based on the magnitude of the gradient at the current iterate. If the norm of Riemannian gradient satisfies $\|\grad f(x_t)\| \ge \Omega(\epsilon)$, Riemannian gradient descent is shown to result in a decrease in the function value of $O(\epsilon^2)$. On the other hand, if the gradient norm is below this threshold, the current point is already an $\epsilon$-approximate Riemannian first-order stationary point. Thus, the Riemannian gradient descent algorithm requires at most $\mathcal{O}(\epsilon^{-2})$ steps to find a first-order stationary point. 
\subsection{The algorithm design}
Inspired by the Riemannian gradient descent algorithm, we employ an unconventional strategy, aiming for a more aggressive function value decrease at each update—a crucial element in designing a faster Riemannian algorithm. Given the inaccessibility of the Riemannian gradient, we carefully examine the value of the zeroth-order estimator. When the Riemannian zeroth-order estimator at the current iterate $x_t$ exceeds $\Omega\left(\sqrt{\epsilon}\right)$—deviating from the standard value of $\Omega\left(\epsilon\right)$—we choose to proceed with the Riemannian zeroth-order gradient descent step (Subroutine \ref{sub.RZGD}), resulting in the function value decrease of $O(\epsilon)$. 

For the iterate $x_t$ with a small Riemannian zeroth-order estimator, we choose the tangent space step (Subroutine \ref{sub.TSS}), which involves the accelerated zeroth-order gradient descent update in the tangent space $\T_{x_t}\M$, as depicted in Figure \ref{fig:tss}. In the tangent space step, we set a similar termination criterion as \cite{li2022restarted}, ensuring that the tangent space step either results in the function value decrease of $O(\epsilon^{3/2})$ or returns a stationary point. Therefore, after a single update from $x_t$ to $x_{t+1}$, the function value takes a decrease at least $O(\epsilon^{3/2})$, which is larger compared to the standard case. Combining all these components, we introduce the Riemannian accelerated zeroth-order gradient descent in Algorithm \ref{alg.AZO}. By always selecting Option I, it achieves lower query complexity in the non-asymptotic analysis. Moreover, with a slightly modified tangent space step (Subroutine \ref{sub.TSSA}), the RAZGD with Option II almost surely avoids strict saddle points asymptotically.

\begin{figure}[H]
    \centering
    \includegraphics[scale=0.25]{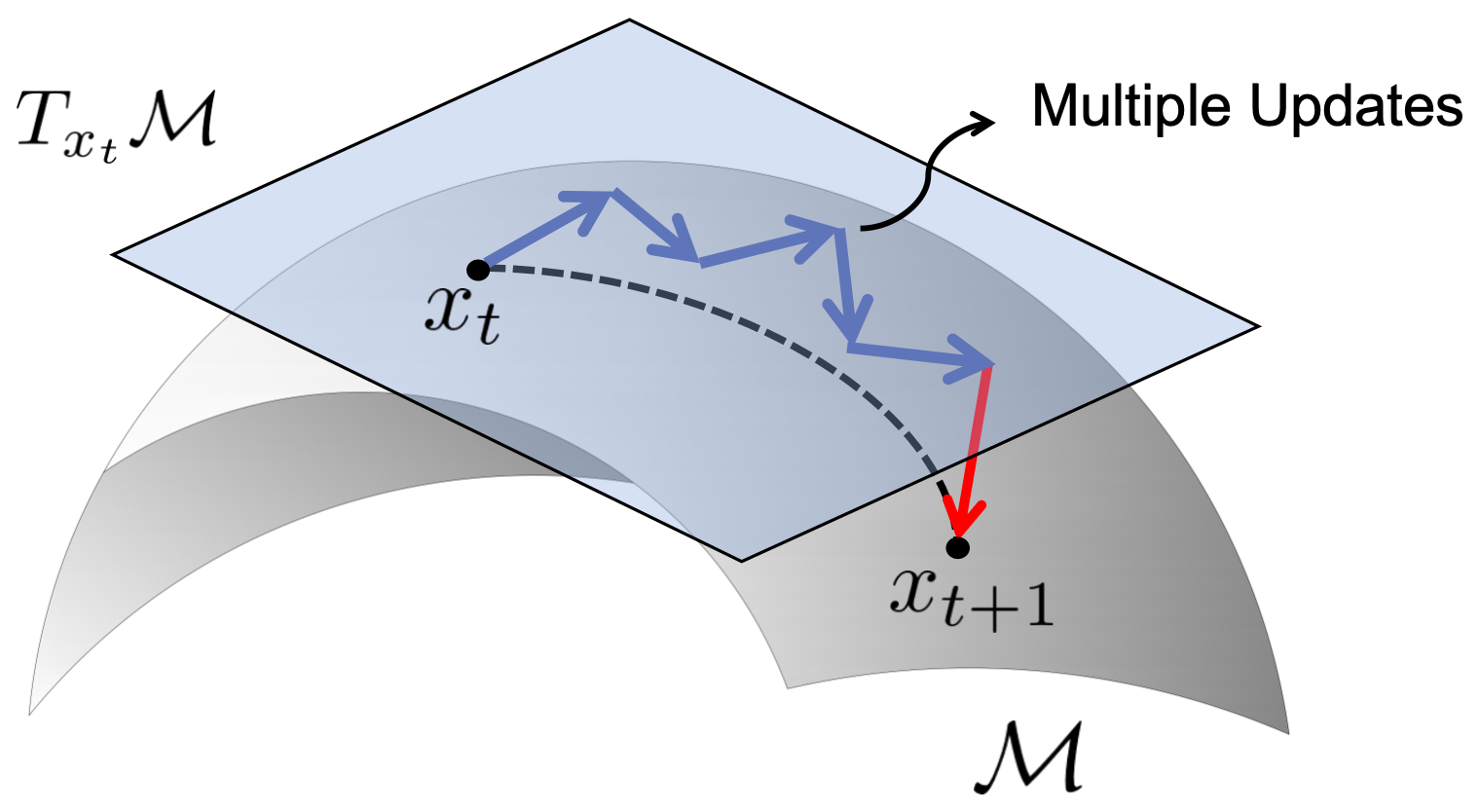}
    \caption{Tangent space step}
    \label{fig:tss}
\end{figure}

\begin{algorithm}[h]
   \caption{\textbf{R}iemannian \textbf{A}ccelerated \textbf{Z}eroth-order \textbf{G}radient \textbf{D}escent Algorithm (\textbf{RAZGD})}
   \label{alg.AZO}
\begin{onehalfspace} %
\begin{algorithmic}[1]
  \STATE \textbf{input:} parameters $\eta$, $\theta$, $B$, $K$ and $r$
      \STATE\textbf{initialize:} $x_0 \in \M$, $t = 0$
      \FOR{$t = 0, 1, \cdots, \infty$}
        \STATE Compute estimator $g_{x_t}(0;\mu)$
        \IF{$\|g_{x_t}(0; \mu)\| \ge lB$}
            \STATE $x_{t+1} = \textbf{RZGDS}(x_t,\eta,g_{x_t}(0;\mu))$
        \ELSE
            \STATE \textbf{Option I:} $x_{t+1} = \textbf{TSS}(x_t, \eta, \theta, B, K, r)$
            \STATE \textbf{Option II:} $x_{t+1} = \textbf{TSSA}(x_t, \eta, \theta, B)$
        \ENDIF
      \ENDFOR
\end{algorithmic}
\end{onehalfspace}
\end{algorithm}

\floatname{algorithm}{Subroutine}
\setcounter{algorithm}{0}
\begin{algorithm}[h]
   \caption{\textbf{R}iemannian \textbf{Z}eroth-order \textbf{G}radient \textbf{D}escent \textbf{S}tep (\textbf{RZGDS})}
   \label{sub.RZGD}
\begin{onehalfspace}
\begin{algorithmic}[1]
  \STATE \textbf{input:} $x$, $\eta$, and $g_{x}(0;\mu)$
        \IF{$\eta\|g_{x}(0; \mu)\| \le b$}
            \STATE Return $\retr_{x}(-\eta g_{x}(0;\mu))$
        \ELSE
            \STATE Compute $\alpha \in (0, 1)$ such that $
            \alpha\eta\| g_{x}(0; \mu)\| = b$
            \STATE Return $\retr_{x}(-\alpha\eta g_{x}(0;\mu))$
        \ENDIF
\end{algorithmic}
\end{onehalfspace}
\end{algorithm}

\begin{algorithm}[h]
    \caption{\textbf{T}angent \textbf{S}pace \textbf{S}tep  (\textbf{TSS})}\label{sub.TSS}
    \begin{onehalfspace}
    \begin{algorithmic}[1]
      \STATE \textbf{input:} $x$, $\eta$, $\theta$, $B$, $K$ and $r$ 
      \STATE \textbf{initialize:} $s_x^{-1} = s_x^0 = \xi \sim \Uniform(\mathbb{B}_{x,r}(0))$, $k = 0$
      \WHILE{$k < K$}
        \STATE $y_x^k = s_x^k + (1 - \theta) (s_x^k - s_x^{k-1})$
        \STATE Compute estimator $g_{x}(y_x^k;\mu)$
        \STATE $s_x^{k+1} = y_x^k - \eta g_{x}(y_x^k; \mu)$
        \STATE $k = k+1$
        \IF{$k \sum_{j=0}^{k-1} \|s_x^{j+1} - s_x^{j}\|^2 > B^2$}
            \STATE Return $\retr_x(s_x^k)$ and break
        \ENDIF
      \ENDWHILE
      \STATE $K_0 = \argmin_{\lfloor\frac{K}{2}\rfloor \le k \le K-1} \|s_x^{k+1} - s_x^k\|$
      \STATE $y_x^* = \frac{1}{K_0 + 1} \sum_{k=0}^{K_0} y_x^k$
      \STATE Return $\retr_x(y_x^*)$
  \end{algorithmic}
  \end{onehalfspace}
\end{algorithm}

\begin{algorithm}[h]
    \caption{\textbf{T}angent \textbf{S}pace \textbf{S}tep \textbf{A}symptotic (\textbf{TSSA})}\label{sub.TSSA}
    \begin{onehalfspace}
    \begin{algorithmic}[1]
      \STATE \textbf{input:} $x$, $\eta$, $\theta$ and $B$
      \STATE \textbf{initialize:} $s_x^{-1} = s_x^0 = 0$, $k = 0$, and constant $\beta < 1$, or $\beta_k=1-\frac{1}{k+2}$ 
      \WHILE{$k \sum_{j=0}^{k-1} \|s_x^{j+1} - s_x^{j}\|^2 \le B^2$}
        \STATE $y_x^k = s_x^k + (1 - \theta) (s_x^k - s_x^{k-1})$
        \STATE Compute estimator $g_{x}(y_x^k;\mu)$
        \STATE $s_x^{k+1} = y_x^k - \eta g_{x}(y_x^k; \mu)$
        \STATE $k = k+1$
        \STATE $\mu = \beta\mu$ (or $\mu=\beta_k\mu$ )
      \ENDWHILE
      \STATE $K_0 = \argmin_{\lfloor\frac{K}{2}\rfloor \le k \le K-1} \|s_x^{k+1} - s_x^k\|$
      \STATE $y_x^* = \frac{1}{K_0 + 1} \sum_{k=0}^{K_0} y_x^k$
      \STATE Return $\retr_x(y_x^*)$
  \end{algorithmic}
  \end{onehalfspace}
\end{algorithm}

As demonstrated in both tangent space steps, multiple zeroth-order updates are performed in the tangent space. To ensure the well-definiteness of the tangent space step, we define the zeroth-order estimator for every pair $(x, s_x)$ in the tangent bundle $\T\M$, where $s_x$ is the point in the tangent space $\T_x\M$. In the algorithm and its subroutines, the notation $g_x(s_x;\mu)$ represents the zeroth-order estimator for the gradient of the pullback function $\nabla \hat f_x(s_x)$ at the pair $(x, s_x) \in \T\M$, incorporating a smoothing parameter $\mu$. The formal definition is shown as follows, which generalizes the classic finite difference gradient approximation in Euclidean space \cite{scheinberg2022finite}.
\begin{definition}\label{def.ZO estimators}
Given a smoothing parameter $\mu > 0$ and a point $x \in \M$, the Riemannian coordinate-wise zeroth-order estimator at the point $s_x \in \T_x\M$ is defined as
\begin{equation}
    g_x(s_x; \mu) = \sum_{i=1}^d \frac{\hat f_x(s_x +\mu e_i) - \hat f_x(s_x -\mu e_i)}{2\mu}e_i, \notag
\end{equation}
where $\{e_1, e_2, \ldots, e_d\}$ is the basis of the tangent space $\T_x\M$.
\end{definition}

For compactness, the approximation error of the Riemannian coordinate zeroth-order estimator is deferred to Appendix \ref{sec.property estimator}.

\section{Convergence Analysis}\label{sec.analysis}
\subsection{Mild assumptions} 

We start with the following assumptions on the Riemannian manifold and objective function, which will be used throughout our analysis. Due to the space limit, we have left a discussion of these assumptions in Appendix \ref{sec.discussion assum}. Firstly, generalizing from the Euclidean case, we assume the Lipschitz continuity of the gradient and Hessian of the pullback function $\hat f_x(\cdot)$. However, it is worth noting that Lipschitz continuity holds only locally due to the nonlinear structure of Riemannian manifolds \cite{criscitiello2022accelerated}. 
\begin{assumption}\label{assum.lipschitz gradient}
There exists constants $b > 0$ and $l > 0$ and $\rho > 0$ such that for all $x \in \M$ and $s,t \in \mathbb{B}_{x,b}(0)$, the pullback function $\hat f_x(\cdot)$ satisfies
\begin{equation*}
    \norm{\nabla\hat{f}_x(s) - \nabla\hat{f}_x(t)} \leq l\norm{s-t}.
\end{equation*}
\end{assumption}
\begin{assumption}\label{assum.lipschitz hessian}
There exists constants $b > 0$ and $\rho > 0$ such that for all $x \in \M$ and and $s,t \in \mathbb{B}_{x,b}(0)$, the pullback function $\hat f_x(\cdot)$ satisfies
\begin{equation*}
    \norm{\nabla^2\hat{f}_x(s) - \nabla^2\hat{f}_x(t)} \leq \rho\norm{s-t}.
\end{equation*}  
\end{assumption}

The next assumption requires that the retraction is well-behaved.
\begin{assumption}\label{assum.retraction}
For any $x \in \M$ and $s \in \T_x\M$ satisfying $\norm{s} \le b$, the singular value of the operator $T_{x,s}$ in Definition \ref{def.retr} is bounded, that is, there exists $\sigma_{\max}, \sigma_{\min} > 0$ such that
\begin{equation*}
   \sigma_{\min} \le \sigma_{\min}(T_{x,s}) \le \sigma_{\max}(T_{x,s}) \le \sigma_{\max}.
\end{equation*}
Furthermore, there exists $\tau \geq 0$ such that the initial acceleration of the the curve $\gamma_{x, s}(t)=\operatorname{Retr}_x(t s)$ with $\|s\| = 1$ is bounded by $\tau:\left\|\gamma_{x, s}^{\prime \prime}(0)\right\| \leq \tau$.
\end{assumption}

\subsection{Non-asymptotic convergence}
In this subsection, we aim to find $\epsilon$-approximate stationary points, which is achieved by always choosing Option I in RAZGD. Our theoretical findings yield different convergence results based on the choice of the initial point $\xi$ in the tangent space step \ref{sub.TSS}. When setting $\xi$ to zero, the RAZGD converges to a first-order Riemannian stationary point. Introducing a small perturbation to $\xi$, the perturbed RAZGD seeks a second-order Riemannian stationary point with high probability. The results are presented below, and the associated proofs are deferred to Appendix \ref{sec.proof non-asymptotic}.
\begin{theorem}\label{thm.fosp}
    Suppose that Assumption \ref{assum.lipschitz gradient}, \ref{assum.lipschitz hessian} and \ref{assum.retraction} hold. Set the parameters in Algorithm \ref{alg.AZO} as follows
    \begin{equation}
       \eta = \frac{1}{4l}, ~B=\frac{1}{8}\sqrt{\frac{\epsilon}{\rho}}, ~\theta = \frac{\rho^{\frac{7}{4}}\epsilon^{\frac{1}{4}}}{l}, ~r=0, ~K = \frac{\rho^{\frac{5}{4}}}{4\epsilon^{\frac{1}{4}}}. \notag
    \end{equation}
     For any $x_0 \in \M$ and sufficiently small $\epsilon > 0$, 
     choose $\mu = \mathcal{O}\left(\frac{\epsilon^{1/4}}{d^{1/4}}\right)$ in Lines 3 of Algorithm \ref{alg.AZO}, and $\mu = \mathcal{O}\left(\frac{\epsilon^{5/8}}{d^{1/4}}\right)$ in Line 5 of Subroutine \ref{sub.TSS}. Then Algorithm \ref{alg.AZO} with Option I outputs an $\epsilon$-approximate first-order stationary point. The total number of function value evaluations is no more than
    \begin{equation}
        \mathcal{O}\left(\frac{(f(x_0) - f_{\operatorname{low}})d}{\epsilon^{\frac{7}{4}}}\right). \notag
    \end{equation}
\end{theorem}

\begin{theorem}\label{thm.sosp}
    Suppose that Assumption \ref{assum.lipschitz gradient}, \ref{assum.lipschitz hessian} and \ref{assum.retraction} hold. Set the parameters in Algorithm \ref{alg.AZO} as follows
    \begin{align*}
        &\eta = \frac{1}{4l}, ~~\quad \quad \theta = \frac{\rho^{\frac{7}{4}}\epsilon^{\frac{1}{4}}}{l}, ~~\quad\quad\chi = \mathcal{O}\left(\log \frac{d}{\delta \epsilon} \right) \ge 1, \\
        &K = \frac{\chi\rho^{\frac{5}{4}}}{4\epsilon^{\frac{1}{4}}}, \quad B=\frac{1}{8\chi^2}\sqrt{\frac{\epsilon}{\rho}}, ~\quad r=\frac{\theta B}{6 K}.
    \end{align*}
     For any $x_0 \in \M$ and sufficiently small $\epsilon > 0$, 
     choose $\mu = \mathcal{O}\left(\frac{\epsilon^{1/4}}{d^{1/4}\chi}\right) = \tilde{\mathcal{O}}\left(\frac{\epsilon^{1/4}}{d^{1/4}}\right)$ in Lines 3 of Algorithm \ref{alg.AZO}, and $\mu = \min\left\{\mathcal{O}\left(\frac{\epsilon^{5/8}}{d^{1/4}\chi^2}\right),\mathcal{O}\left(\frac{\epsilon^{7/8}}{\chi^3\sqrt{d}}\right)\right\} =\tilde{\mathcal{O}}\left(\frac{\epsilon^{7/8}}{\sqrt{d}}\right)$ in Line 5 of Subroutine \ref{sub.TSS}. Then perturbed Algorithm \ref{alg.AZO} with Option I outputs an $\epsilon$-approximate second-order stationary point with a probability of at least $1-\delta$. The total number of function value evaluations is no more than
    \begin{equation}
    O\left(\frac{(f(x_0) - f_{\operatorname{low}})d}{\epsilon^{\frac{7}{4}}}\log^6\left(\frac{d}{\delta\epsilon}\right)\right). \notag
    \end{equation}
\end{theorem}


In dealing with the unavailability of the first-order information, we carefully choose the smoothing parameter $\mu$ in the construction of its zeroth-order estimators. On the one hand, a small value of $\mu$ reduces the approximation error, yielding a sufficient decrease in the function value. On the other hand, an excessively small $\mu$ can cause practical instability. Consequently, a trade-off arises in selecting the smoothing parameter, requiring a careful balance between maintaining the decrease in the function value and ensuring practical robustness. In our proofs, we frequently use Young's inequality to guarantee this balance, leading to a better choice of the smoothing parameter compared to the corresponding choice in the Euclidean counterpart \cite{zhang2022faster}.

\begin{remark}
  For the special case of $\M = \mathbb{R}^d$ and $\retr_x(s) = x+s$, Theorem \ref{thm.sosp} reveals that the function query complexity of perturbed RAZGD with Option I matches the state-of-the-art result in Euclidean space \cite{zhang2022faster}. However, to attain the lower function query complexity, the accelerated zeroth-order algorithms in \citet{zhang2022faster} demand the smoothing parameter $\mu = \Tilde{\mathcal{O}}(\epsilon^{13/8}d^{-1/2})$. In contrast, our perturbed RAZGD relaxes this requirement to $\mu = \Tilde{\mathcal{O}}(\epsilon^{7/8}d^{-1/2})$, providing a more robust selection guarantee.      
\end{remark}

\paragraph{Why smoothing parameter $\mu$ is important?}  Compared to first-order algorithms, the notable distinction of zeroth-order algorithms lies in the necessity to construct zeroth-order estimates through function value evaluations. Among these zeroth-order estimates, the smoothing parameter plays a crucial role as an indicator. The efficiency of zeroth-order algorithms is measured by the total number of function value evaluations, while the value of the smoothing parameter determines its robustness. Generally, a smaller smoothing parameter improves the approximation quality of the zeroth-order estimator, see Lemma \ref{lemma.zero order estimator error}, for example. Nevertheless, in practical systems, an excessively small $\mu$ might induce the dominance of system noise in function differences, causing the failure to represent the function differential \cite{lian2016comprehensive,liu2018zeroth,liu2020primer,nguyen2023stochastic}. Therefore, maintaining a relatively large smoothing parameter is paramount to the robustness of zeroth-order algorithms. In the realm of randomized zeroth-order estimators, \citet{ren2023escaping} improved the value of the smoothing parameter from $\mathcal{O}(\epsilon^{3}d^{-2})$ in \cite{bai2020escaping} to a more efficient choice $\mathcal{O}(\epsilon^{1/2}d^{-1})$ for finding second-order stationary points in Euclidean space. When dealing with Riemannian manifolds, \citet{wang2021greene} demonstrated that choosing random vectors uniformly from the unit sphere enables a less restrictive smoothing parameter. The selection of the smoothing parameter can be improved from $\mathcal{O}(\epsilon (d+3)^{-3/2})$ to $\mathcal{O}(\epsilon d^{-3/2})$ for seeking Riemannian first-order stationary points.


\subsection{Asymptotic convergence} Now we turn to investigate the asymptotic convergence of Algorithm \ref{alg.AZO}, which can be proven to avoid Riemannian strict saddle points almost surely by employing Option II, i.e. tangent space step asymptotic. In contrast to Subroutine \ref{sub.TSS}, where the smoothing parameter maintains the same value during the update, in Subroutine \ref{sub.TSSA}, we initialize the smoothing parameter $\mu$ with an appropriate constant value, and then make $\mu$ gradually decay by multiplying it with the contraction parameter $\beta$. Previous results have established non-asymptotic convergence to $\epsilon$-approximate second-order stationary points, indicating that, with high probability, Algorithm \ref{alg.AZO} with Option I will output a point satisfying the specified threshold. However, there is a gap between high probability and probability 1 for convergencing to second-order stationary points. We close this gap by providing the following result asserting that the set of initial points that can be iterated to Riemannian saddle points has measure (the volume induced by Riemannian metric) zero.

\begin{theorem}\label{thm.asymptotic}
Suppose that Assumption \ref{assum.lipschitz gradient}, \ref{assum.lipschitz hessian} and \ref{assum.retraction} hold. For any $x_0 \in \M$ and sufficiently small $\epsilon>0$, set
\[
\frac{\theta}{(2-\theta)\lambda_*}<\eta\le\frac{1}{4 l}, \ \ \theta = \frac{\rho^{\frac{7}{4}}\epsilon^{\frac{1}{4}}}{l} \le\min\left\{\frac{-2\lambda_*}{4l-\lambda_*},1\right\}
\]
where $\lambda_*$ is the negative eigenvalue of Hessian at saddle points with the greatest magnitude. Choose the smoothing parameter $\mu$ with a reasonable constant magnitude in both Line 4 of Algorithm \ref{alg.AZO} and during the initialization of Subroutine \ref{sub.TSSA}. Choose constant $\beta<1$ or a sequence $\beta_k=\left(1-\frac{1}{k+2}\right)$, and the rest parameters follow the choices of Theorem \ref{thm.fosp}. Then Algorithm \ref{alg.AZO} with Option II avoids strict Riemannian saddle points almost surely. Furthermore, this implies that the Algorithm \ref{alg.AZO} with Option II asymptotically converges to a strict Riemannian second-order stationary point.
\end{theorem}

\begin{remark}
     In the tangent space step $\TSSA$, the smoothing parameter $\mu$ decreases in exponential rate, which makes the zeroth-order method almost identical to a first-order method. Despite this setting being convenient for theoretical analysis, the rapid decaying of smoothing parameters makes the algorithm less attractive from the zeroth-order optimization perspective. To reduce the rate of decaying of the smoothing parameter, we propose an alternating approach with a time-varying contracting factor, i.e., multiplying by a factor of $\left(1-\frac{1}{k+2}\right)$. It is obvious that the rate of decaying of the smoothing parameter in $\TSSA$ stage is $\mu_{k+1}=\frac{1}{3(k+2)}\mu$, which is much slower than the exponential decaying given by $\mu_{k+1}=\beta^k\mu$. Fortunately, the asymptotic avoidance of saddle points holds with $\beta_k=1-\frac{1}{k+2}$.
\end{remark}
\section{Numerical Experiments}
In this section, we conduct experiments to demonstrate the robustness and efficiency of RAZGD. Specifically, we implement the tangent space step (Subroutine \ref{sub.TSS}) due to its non-asymptotic complexity guarantee. All experiments are performed on a computer with a 24-core Intel Core i9-13900HX processor.

\subsection{Improved robustness}
To verify the robust performance, we consider the following quartic function \cite{lucchi2021second,zhang2022faster} on Euclidean space $\mathbb{R}^d$:
\begin{equation*}
f\left(x_1, x_2, \ldots, x_d, y\right)=\frac{1}{4} \sum_{i=1}^d x_i^4-y \sum_{i=1}^d x_i+\frac{d}{2} y^2
\end{equation*}
which has a strict saddle point at $x_0 = (0,\ldots,0)^{\top}$ and two global minima at $(1,\ldots,1)^{\top}$ and $(-1,\ldots,-1)^{\top}$.

In this experiment, we test Algorithm \ref{alg.AZO} with perturbation in the tangent space step (Perturbed-RAZGD) along with two Euclidean accelerated zeroth-order algorithms, ZO-Perturbed-AGD, and ZO-Perturbed-AGD-ANCF in \cite{zhang2022faster}. We choose the retraction as $\retr_x(s) = x+s$. The basic parameters for all three algorithms follow respective theorems. The smoothing parameter $\mu$ is set to $0.01$ for each algorithm, and notably, we run an additional choice of $\mu=0.3$ for our algorithm. The initial point is set as the saddle point $x_0$. Due to the inherent randomness in these algorithms, each algorithm is executed $10$ times and we report the averaged function value versus the averaged number of function queries in Figure \ref{fig:mu}. Figure \ref{fig:mu} demonstrates that the variance of our algorithm (indicated by the width of the shadow) is smaller than the other two with the same smoothing parameter $\mu = 0.01$. Furthermore, our algorithm still convergences even with a larger smoothing parameter $\mu = 0.3$. The two aspects visually showcase the robustness of our accelerated zeroth-order algorithm.
\begin{figure*}[h]
\centering
\subfigure[d=20]{\includegraphics[height=4cm,width=5cm]{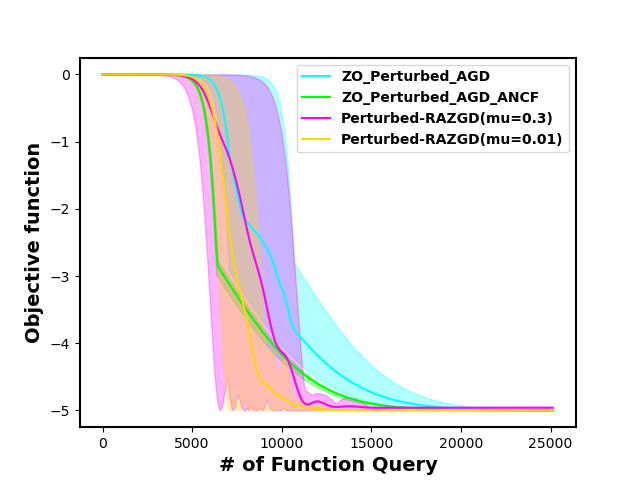}}
\subfigure[d=100]{\includegraphics[height=4cm,width=5cm]{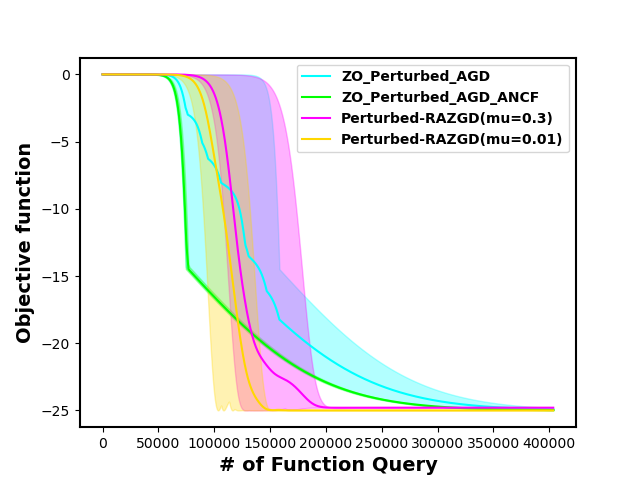}}
\subfigure[d=200]{\includegraphics[height=4cm,width=5cm]{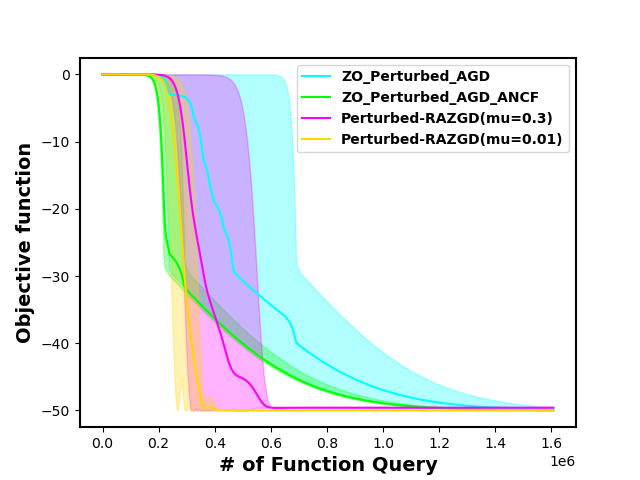}}
\caption{Performance of different zeroth-order accelerated algorithms to minimize the quartic function with growing dimensions. Confidence intervals show mini-max intervals over ten runs.}
\label{fig:mu}
\end{figure*}


\subsection{Lower function queries}\label{subsec.numerc}
In this part, we assess the acceleration effectiveness of the non-perturbed RAZGD with Option I by comparing it with Riemannian zeroth-order gradient descent (RZGD) and Euclidean projected zeroth-order gradient descent (PZGD). The corresponding pseudocodes are left in Appendix \ref{sec.RZOGD}. 

We first consider the simplex constrained least-square problem \cite{li2023simplex}
\begin{align*}
    &\min \quad \|Ax-b\|_2^2 \\
    &\text{s.t.} \quad \ x \in \Delta^{d-1},
\end{align*}
where $\Delta^{d-1}=\{x \in \mathbb{R}^d: \sum_{i=1}^d x_i=1 \text { and } x \geq 0\}$, $A \in \mathbb{R}^{m \times d}$ and $b \in \mathbb{R}^m$. Since the positive orthant is a Riemannian manifold with the Shahshahani metric, the simplex has a natural submanifold structure \cite{Shahshahani}. For any point $x$ in the interior of $\Delta^{d-1}$, the tangent space is the hyperplane passing through $0$ and parallel to $x \in \Delta^{d-1}$, i.e. $\T_x\Delta^{d-1} = \{s \in \mathbb{R}^d: \sum_{j=1}^d s_j = 0\}$. We use the exponential map on the Shahshahani manifold as the retraction \cite{feng2022accelerated}. A detailed discussion of Riemannian geometry of the simplex is deferred to Appendix \ref{GometryofSimplex}. In the experiment, the feature matrix $A$ is drawn from a standard Gaussian distribution, and the label vector $b$ is generated using the expression $A\zeta+\mu$. Here, $\zeta$ and $\mu$ are randomly sampled from a Gaussian distribution, with the additional constraint that the sum of all elements equals 1 for $\zeta$. For PZGD, we apply the projection in \cite{chen2011projection}. The results are reported in Figure \ref{fig:linear}, showing that RAZGD requires lower function queries. 
\begin{figure}[H]
\centering
\subfigure[m=200,d=20]{\includegraphics[height=3cm,width=4cm]{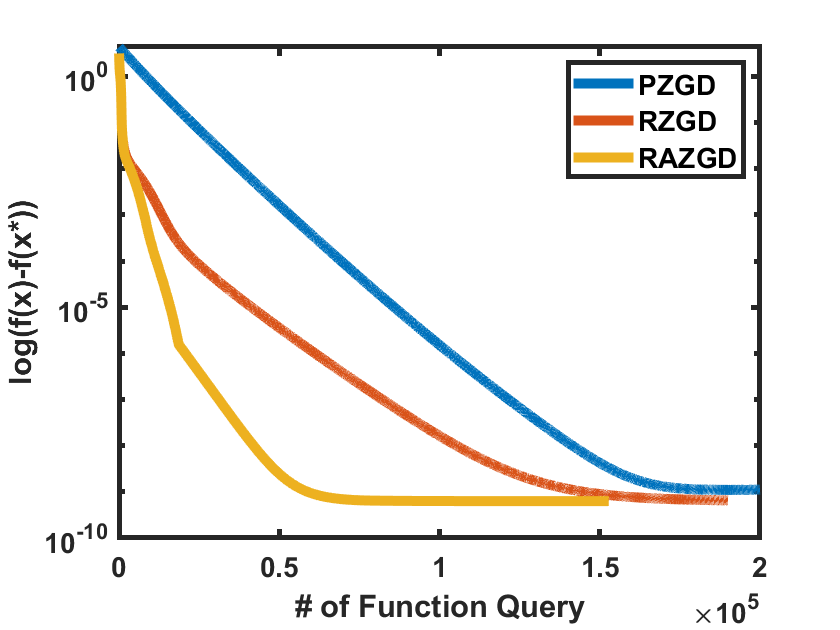}}
\subfigure[m=300,d=30]{\includegraphics[height=3cm,width=4cm]{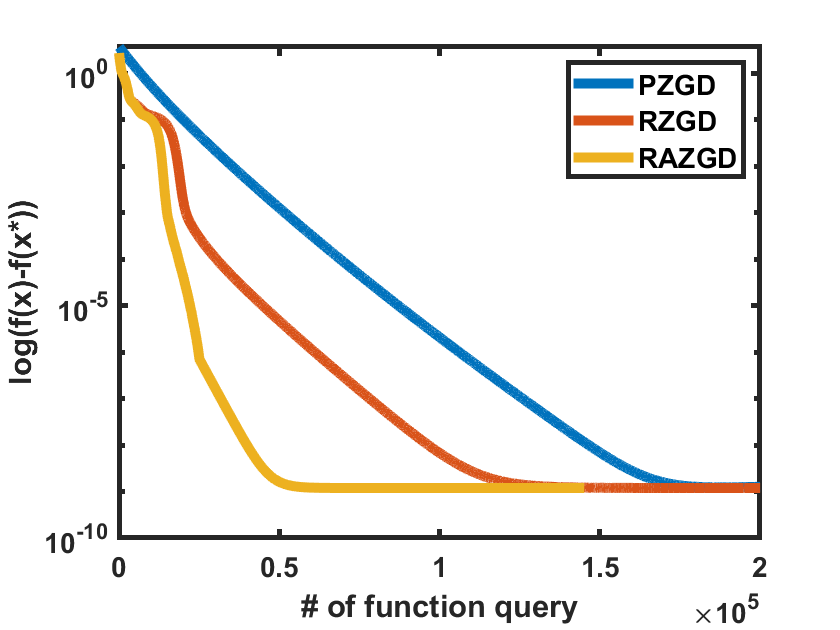}}
\caption{Performance on linear least-squares over the unit simplex with different problem sizes.}
\label{fig:linear}
\end{figure}

\begin{figure}[H]
\centering
\subfigure[Category:3 Dimension:10]{\includegraphics[height=3cm,width=4cm]{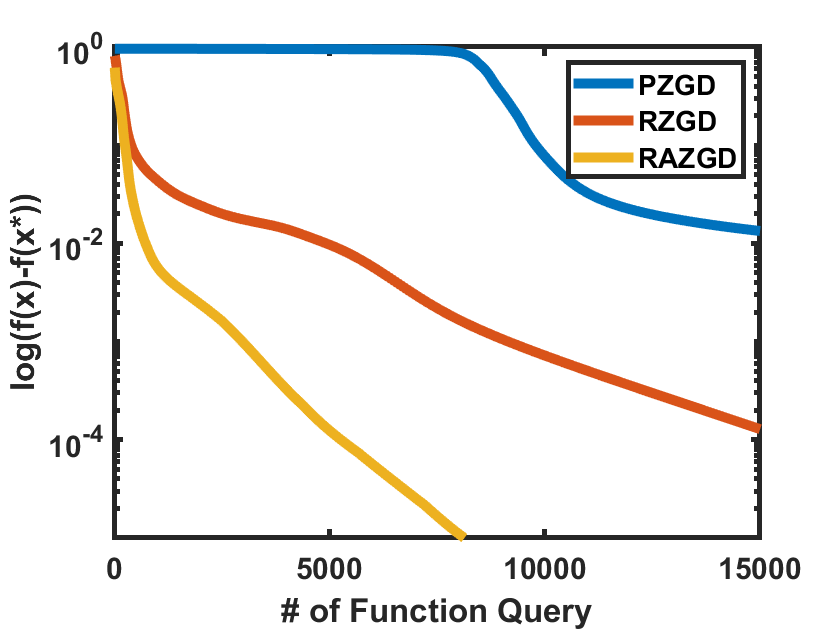}}
\subfigure[Category:2 Dimension:20]{\includegraphics[height=3cm,width=4cm]{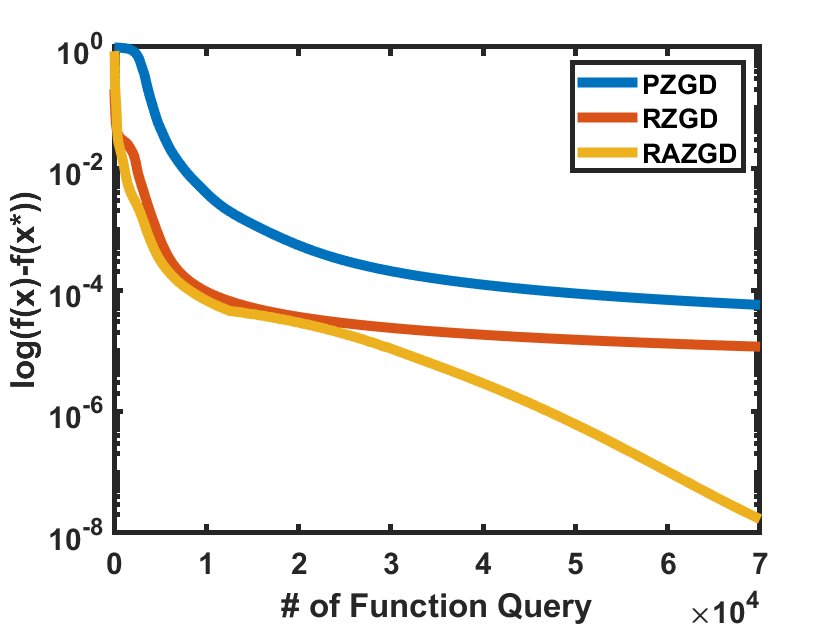}}
\vspace{-0.1cm}
\caption{Performance on empirical hypervolume under the sphere manifold with different categories and dimensions.}
\label{fig:SCOR}
\end{figure}
To further demonstrate the efficiency, we consider a real-world application: the optimal linear combination of continuous predictors in the context of a binary classification problem \cite{das2022estimating}. For multi-category responses, the optimal predictor combination can be obtained
by maximization of the empirical hypervolume under the manifold, with the following form
\begin{align*}
    &\max \quad f(x) \\
    &\text{s.t.} \quad \ \sum_{i=1}^d x_i^2=1, x \in \mathbb{R}^d,
\end{align*}
where the objective function $f(\cdot)$ takes no analytic form. We test algorithms on both two disease categories and three disease categories, and results are shown in Figure \ref{fig:SCOR}. For the unit sphere $\mathbb{S}^{d-1}$, the tangent space is defined as $\T_x\mathbb{S}^{d-1}:=\{s \in \mathbb{R}^d: \sum_{j=1}^d x_j s_j = 0\}$. In the experiment, the process of biomarker data generation is consistent with \cite{das2022estimating}. The retraction is chosen as $\retr_x(s) = (x+s)/\norm{x+s}_2$. For PZGD, we use $x/\|x\|_2$ as the projection to the unit sphere. It is worth mentioning that in practical scenarios, the lower function queries lead to less running time. Thus, RAZGD reaches the target accuracy within 30 seconds in both cases, while PZGD needs more than 300 seconds to achieve the same accuracy, indicating the effective performance of our accelerated algorithm.


\section{Conclusions}
In the paper, we introduce a Riemannian accelerated zeroth-order gradient descent based on the deterministic coordinate-wise zeroth-order estimator. Our accelerated algorithm attains the best-known function query complexity for achieving both $\epsilon$-approximate first-order and second-order stationary points respectively. Notably, it 
allows a larger
smoothing parameter
and thus demonstrates better robustness. Furthermore, we also establish the asymptotic convergence behavior with probability $1$. Experimental results are presented, verifying the superior performance in terms of both function query complexity and robustness.
\newpage
\section*{Acknowledgement}
We thank Huikang Liu (Shanghai Jiao Tong University) for several helpful discussions. This research is partially supported by the National Natural Science Foundation of China (Grants 72394360,72394364, 72394365, 72171141).

\section*{Potential Broader Impact}
This paper does not have any potential societal or ethical consequences.


\bibliography{main}

\begin{thebibliography}{57}
\providecommand{\natexlab}[1]{#1}
\providecommand{\url}[1]{\texttt{#1}}
\expandafter\ifx\csname urlstyle\endcsname\relax
  \providecommand{\doi}[1]{doi: #1}\else
  \providecommand{\doi}{doi: \begingroup \urlstyle{rm}\Url}\fi

\bibitem[Absil et~al.(2009)Absil, Mahony, and Sepulchre]{absil2009optimization}
Absil, P.-A., Mahony, R., and Sepulchre, R.
\newblock Optimization algorithms on matrix manifolds.
\newblock In \emph{Optimization Algorithms on Matrix Manifolds}. Princeton University Press, 2009.

\bibitem[Agarwal et~al.(2021)Agarwal, Boumal, Bullins, and Cartis]{agarwal2021adaptive}
Agarwal, N., Boumal, N., Bullins, B., and Cartis, C.
\newblock Adaptive regularization with cubics on manifolds.
\newblock \emph{Mathematical Programming}, 188:\penalty0 85--134, 2021.

\bibitem[Alimisis et~al.(2021)Alimisis, Orvieto, Becigneul, and Lucchi]{alimisis2021momentum}
Alimisis, F., Orvieto, A., Becigneul, G., and Lucchi, A.
\newblock Momentum improves optimization on riemannian manifolds.
\newblock In \emph{International conference on artificial intelligence and statistics}, pp.\  1351--1359. PMLR, 2021.

\bibitem[Bai et~al.(2020)Bai, Agarwal, and Aggarwal]{bai2020escaping}
Bai, Q., Agarwal, M., and Aggarwal, V.
\newblock Escaping saddle points for zeroth-order non-convex optimization using estimated gradient descent.
\newblock In \emph{2020 54th Annual Conference on Information Sciences and Systems (CISS)}, pp.\  1--6. IEEE, 2020.

\bibitem[Balasubramanian \& Ghadimi(2022)Balasubramanian and Ghadimi]{balasubramanian2022zeroth}
Balasubramanian, K. and Ghadimi, S.
\newblock Zeroth-order nonconvex stochastic optimization: Handling constraints, high dimensionality, and saddle points.
\newblock \emph{Foundations of Computational Mathematics}, pp.\  1--42, 2022.

\bibitem[Beck \& Teboulle(2009)Beck and Teboulle]{beck2009fast}
Beck, A. and Teboulle, M.
\newblock A fast iterative shrinkage-thresholding algorithm for linear inverse problems.
\newblock \emph{SIAM journal on imaging sciences}, 2\penalty0 (1):\penalty0 183--202, 2009.

\bibitem[Bishop \& O’Neill(1969)Bishop and O’Neill]{bishop1969manifolds}
Bishop, R.~L. and O’Neill, B.
\newblock Manifolds of negative curvature.
\newblock \emph{Transactions of the American Mathematical Society}, 145:\penalty0 1--49, 1969.

\bibitem[Boumal(2023)]{boumal2023introduction}
Boumal, N.
\newblock \emph{An introduction to optimization on smooth manifolds}.
\newblock Cambridge University Press, 2023.

\bibitem[Boumal et~al.(2019)Boumal, Absil, and Cartis]{boumal2019global}
Boumal, N., Absil, P.-A., and Cartis, C.
\newblock Global rates of convergence for nonconvex optimization on manifolds.
\newblock \emph{IMA Journal of Numerical Analysis}, 39\penalty0 (1):\penalty0 1--33, 2019.

\bibitem[Bridson \& Haefliger(2013)Bridson and Haefliger]{bridson2013metric}
Bridson, M.~R. and Haefliger, A.
\newblock \emph{Metric spaces of non-positive curvature}, volume 319.
\newblock Springer Science \& Business Media, 2013.

\bibitem[Bubeck et~al.(2019)Bubeck, Jiang, Lee, Li, and Sidford]{bubeck2019near}
Bubeck, S., Jiang, Q., Lee, Y.~T., Li, Y., and Sidford, A.
\newblock Near-optimal method for highly smooth convex optimization.
\newblock In \emph{Conference on Learning Theory}, pp.\  492--507. PMLR, 2019.

\bibitem[Carmon et~al.(2017)Carmon, Duchi, Hinder, and Sidford]{carmon2017convex}
Carmon, Y., Duchi, J.~C., Hinder, O., and Sidford, A.
\newblock “convex until proven guilty”: Dimension-free acceleration of gradient descent on non-convex functions.
\newblock In \emph{International conference on machine learning}, pp.\  654--663. PMLR, 2017.

\bibitem[Carmon et~al.(2018)Carmon, Duchi, Hinder, and Sidford]{carmon2018accelerated}
Carmon, Y., Duchi, J.~C., Hinder, O., and Sidford, A.
\newblock Accelerated methods for nonconvex optimization.
\newblock \emph{SIAM Journal on Optimization}, 28\penalty0 (2):\penalty0 1751--1772, 2018.

\bibitem[Chen \& Ye(2011)Chen and Ye]{chen2011projection}
Chen, Y. and Ye, X.
\newblock Projection onto a simplex.
\newblock \emph{arXiv preprint arXiv:1101.6081}, 2011.

\bibitem[Criscitiello \& Boumal(2019)Criscitiello and Boumal]{criscitiello2019efficiently}
Criscitiello, C. and Boumal, N.
\newblock Efficiently escaping saddle points on manifolds.
\newblock \emph{Advances in Neural Information Processing Systems}, 32, 2019.

\bibitem[Criscitiello \& Boumal(2022)Criscitiello and Boumal]{criscitiello2022accelerated}
Criscitiello, C. and Boumal, N.
\newblock An accelerated first-order method for non-convex optimization on manifolds.
\newblock \emph{Foundations of Computational Mathematics}, pp.\  1--77, 2022.

\bibitem[Das et~al.(2022)Das, De, Maiti, Kamal, Hutcheson, Fuller, Chakraborty, and Peterson]{das2022estimating}
Das, P., De, D., Maiti, R., Kamal, M., Hutcheson, K.~A., Fuller, C.~D., Chakraborty, B., and Peterson, C.~B.
\newblock Estimating the optimal linear combination of predictors using spherically constrained optimization.
\newblock \emph{BMC bioinformatics}, 23\penalty0 (3):\penalty0 1--20, 2022.

\bibitem[Fan et~al.(2021)Fan, Gao, Wu, Jia, and Harandi]{fan2021learning}
Fan, X., Gao, Z., Wu, Y., Jia, Y., and Harandi, M.
\newblock Learning a gradient-free {R}iemannian optimizer on tangent spaces.
\newblock In \emph{Proceedings of the AAAI Conference on Artificial Intelligence}, volume~35, pp.\  7377--7384, 2021.

\bibitem[Feng et~al.(2022)Feng, Panageas, and Wang]{feng2022accelerated}
Feng, Y., Panageas, I., and Wang, X.
\newblock Accelerated multiplicative weights update avoids saddle points almost always.
\newblock \emph{arXiv preprint arXiv:2204.11407}, 2022.

\bibitem[Flokas et~al.(2019)Flokas, Vlatakis-Gkaragkounis, and Piliouras]{GP2019}
Flokas, L., Vlatakis-Gkaragkounis, E.~V., and Piliouras, G.
\newblock Efficiently avoiding saddle points with zero order methods: No gradients required.
\newblock In \emph{NeurIPS}, 2019.

\bibitem[Horn \& Johnson(2012)Horn and Johnson]{horn2012matrix}
Horn, R.~A. and Johnson, C.~R.
\newblock \emph{Matrix analysis}.
\newblock Cambridge university press, 2012.

\bibitem[Jaquier et~al.(2018)Jaquier, Rozo, Caldwell, and Calinon]{jaquier2018geometry}
Jaquier, N., Rozo, L.~D., Caldwell, D.~G., and Calinon, S.
\newblock Geometry-aware tracking of manipulability ellipsoids.
\newblock In \emph{Robotics: Science and Systems}, number CONF, 2018.

\bibitem[Jaquier et~al.(2020)Jaquier, Rozo, Calinon, and B{\"u}rger]{jaquier2020bayesian}
Jaquier, N., Rozo, L., Calinon, S., and B{\"u}rger, M.
\newblock Bayesian optimization meets riemannian manifolds in robot learning.
\newblock In \emph{Conference on Robot Learning}, pp.\  233--246. PMLR, 2020.

\bibitem[Jiang et~al.(2021)Jiang, Wang, and Zhang]{jiang2021optimal}
Jiang, B., Wang, H., and Zhang, S.
\newblock An optimal high-order tensor method for convex optimization.
\newblock \emph{Mathematics of Operations Research}, 46\penalty0 (4):\penalty0 1390--1412, 2021.

\bibitem[Jin et~al.(2018)Jin, Netrapalli, and Jordan]{jin2018accelerated}
Jin, C., Netrapalli, P., and Jordan, M.~I.
\newblock Accelerated gradient descent escapes saddle points faster than gradient descent.
\newblock In \emph{Conference On Learning Theory}, pp.\  1042--1085. PMLR, 2018.

\bibitem[Jin \& Sra(2022)Jin and Sra]{jin2022understanding}
Jin, J. and Sra, S.
\newblock Understanding riemannian acceleration via a proximal extragradient framework.
\newblock In \emph{Conference on Learning Theory}, pp.\  2924--2962. PMLR, 2022.

\bibitem[Kim \& Yang(2022)Kim and Yang]{kim2022accelerated}
Kim, J. and Yang, I.
\newblock Accelerated gradient methods for geodesically convex optimization: Tractable algorithms and convergence analysis.
\newblock In \emph{International Conference on Machine Learning}, pp.\  11255--11282. PMLR, 2022.

\bibitem[Li \& Lin(2022)Li and Lin]{li2022restarted}
Li, H. and Lin, Z.
\newblock Restarted nonconvex accelerated gradient descent: No more polylogarithmic factor in the $o(\epsilon^{-7/4})$ complexity.
\newblock In \emph{International Conference on Machine Learning}, pp.\  12901--12916. PMLR, 2022.

\bibitem[Li et~al.(2023{\natexlab{a}})Li, Balasubramanian, and Ma]{li2023stochastic}
Li, J., Balasubramanian, K., and Ma, S.
\newblock Stochastic zeroth-order {R}iemannian derivative estimation and optimization.
\newblock \emph{Mathematics of Operations Research}, 48\penalty0 (2):\penalty0 1183--1211, 2023{\natexlab{a}}.

\bibitem[Li et~al.(2023{\natexlab{b}})Li, Balasubramanian, and Ma]{li2023zeroth}
Li, J., Balasubramanian, K., and Ma, S.
\newblock Zeroth-order {R}iemannian averaging stochastic approximation algorithms.
\newblock \emph{arXiv preprint arXiv:2309.14506}, 2023{\natexlab{b}}.

\bibitem[Li et~al.(2023{\natexlab{c}})Li, McKenzie, and Yin]{li2023simplex}
Li, Q., McKenzie, D., and Yin, W.
\newblock From the simplex to the sphere: faster constrained optimization using the hadamard parametrization.
\newblock \emph{Information and Inference: A Journal of the IMA}, 12\penalty0 (3):\penalty0 iaad017, 2023{\natexlab{c}}.

\bibitem[Lian et~al.(2016)Lian, Zhang, Hsieh, Huang, and Liu]{lian2016comprehensive}
Lian, X., Zhang, H., Hsieh, C.-J., Huang, Y., and Liu, J.
\newblock A comprehensive linear speedup analysis for asynchronous stochastic parallel optimization from zeroth-order to first-order.
\newblock \emph{Advances in Neural Information Processing Systems}, 29, 2016.

\bibitem[Lin et~al.(2015)Lin, Mairal, and Harchaoui]{lin2015universal}
Lin, H., Mairal, J., and Harchaoui, Z.
\newblock A universal catalyst for first-order optimization.
\newblock \emph{Advances in neural information processing systems}, 28, 2015.

\bibitem[Lin et~al.(2020)Lin, Saparbayeva, Zhang, and Dunson]{lin2020accelerated}
Lin, L., Saparbayeva, B., Zhang, M.~M., and Dunson, D.~B.
\newblock Accelerated algorithms for convex and non-convex optimization on manifolds.
\newblock \emph{arXiv preprint arXiv:2010.08908}, 2020.

\bibitem[Liu et~al.(2018)Liu, Kailkhura, Chen, Ting, Chang, and Amini]{liu2018zeroth}
Liu, S., Kailkhura, B., Chen, P.-Y., Ting, P., Chang, S., and Amini, L.
\newblock Zeroth-order stochastic variance reduction for nonconvex optimization.
\newblock \emph{Advances in Neural Information Processing Systems}, 31, 2018.

\bibitem[Liu et~al.(2020)Liu, Chen, Kailkhura, Zhang, Hero~III, and Varshney]{liu2020primer}
Liu, S., Chen, P.-Y., Kailkhura, B., Zhang, G., Hero~III, A.~O., and Varshney, P.~K.
\newblock A primer on zeroth-order optimization in signal processing and machine learning: Principals, recent advances, and applications.
\newblock \emph{IEEE Signal Processing Magazine}, 37\penalty0 (5):\penalty0 43--54, 2020.

\bibitem[Liu et~al.(2017)Liu, Shang, Cheng, Cheng, and Jiao]{liu2017accelerated}
Liu, Y., Shang, F., Cheng, J., Cheng, H., and Jiao, L.
\newblock Accelerated first-order methods for geodesically convex optimization on {R}iemannian manifolds.
\newblock \emph{Advances in Neural Information Processing Systems}, 30, 2017.

\bibitem[Lucchi et~al.(2021)Lucchi, Orvieto, and Solomou]{lucchi2021second}
Lucchi, A., Orvieto, A., and Solomou, A.
\newblock On the second-order convergence properties of random search methods.
\newblock \emph{Advances in Neural Information Processing Systems}, 34:\penalty0 25633--25645, 2021.

\bibitem[Maass et~al.(2022)Maass, Manzie, Nesic, Manton, and Shames]{maass2022tracking}
Maass, A.~I., Manzie, C., Nesic, D., Manton, J.~H., and Shames, I.
\newblock Tracking and regret bounds for online zeroth-order euclidean and {R}iemannian optimization.
\newblock \emph{SIAM Journal on Optimization}, 32\penalty0 (2):\penalty0 445--469, 2022.

\bibitem[Mertikopoulos \& Sandholm(2018)Mertikopoulos and Sandholm]{MertiRiemannGame}
Mertikopoulos, P. and Sandholm, W.~H.
\newblock Riemannian game dynamics.
\newblock \emph{Journal of Economic Theory}, 2018.

\bibitem[Nesterov(2008)]{nesterov2008accelerating}
Nesterov, Y.
\newblock Accelerating the cubic regularization of newton’s method on convex problems.
\newblock \emph{Mathematical Programming}, 112\penalty0 (1):\penalty0 159--181, 2008.

\bibitem[Nesterov \& Spokoiny(2017)Nesterov and Spokoiny]{nesterov2017random}
Nesterov, Y. and Spokoiny, V.
\newblock Random gradient-free minimization of convex functions.
\newblock \emph{Foundations of Computational Mathematics}, 17:\penalty0 527--566, 2017.

\bibitem[Nesterov(1983)]{nesterov1983method}
Nesterov, Y.~E.
\newblock A method of solving a convex programming problem with convergence rate o$\backslash$bigl(k\^{}2$\backslash$bigr).
\newblock In \emph{Doklady Akademii Nauk}, volume 269, pp.\  543--547. Russian Academy of Sciences, 1983.

\bibitem[Nguyen \& Balasubramanian(2023)Nguyen and Balasubramanian]{nguyen2023stochastic}
Nguyen, A. and Balasubramanian, K.
\newblock Stochastic zeroth-order functional constrained optimization: Oracle complexity and applications.
\newblock \emph{INFORMS Journal on Optimization}, 5\penalty0 (3):\penalty0 256--272, 2023.

\bibitem[Ostrowski(1961)]{ostrowski1961some}
Ostrowski, A.~M.
\newblock On some metrical properties of operator matrices and matrices partitioned into blocks.
\newblock \emph{Journal of Mathematical Analysis and Applications}, 2\penalty0 (2):\penalty0 161--209, 1961.

\bibitem[Panageas et~al.(2019)Panageas, Piliouras, and Wang]{PPW19}
Panageas, I., Piliouras, G., and Wang, X.
\newblock First-order methods almost always avoid saddle points: The case of vanishing step-sizes.
\newblock In \emph{Advances in Neural Information Processing Systems 32: Annual Conference on Neural Information Processing Systems 2019, NeurIPS 2019, 8-14 December 2019, Vancouver, BC, Canada}, pp.\  6471--6480, 2019.

\bibitem[Ren et~al.(2023)Ren, Tang, and Li]{ren2023escaping}
Ren, Z., Tang, Y., and Li, N.
\newblock Escaping saddle points in zeroth-order optimization: the power of two-point estimators.
\newblock In \emph{International Conference on Machine Learning}, pp.\  28914--28975. PMLR, 2023.

\bibitem[Scheinberg(2022)]{scheinberg2022finite}
Scheinberg, K.
\newblock Finite difference gradient approximation: To randomize or not?
\newblock \emph{INFORMS Journal on Computing}, 34\penalty0 (5):\penalty0 2384--2388, 2022.

\bibitem[Shahshahani(1979)]{Shahshahani}
Shahshahani, S.
\newblock \emph{A New Mathematical Framework for the Study of linage and Selection}, volume~17.
\newblock American Mathematical Society, 1979.

\bibitem[Shub(1987)]{shub}
Shub, M.
\newblock \emph{Global stability of dynamical systems}.
\newblock Springer Science \& Business Media, 1987.

\bibitem[Vlatakis-Gkaragkounis et~al.(2019)Vlatakis-Gkaragkounis, Flokas, and Piliouras]{vlatakis2019efficiently}
Vlatakis-Gkaragkounis, E.-V., Flokas, L., and Piliouras, G.
\newblock Efficiently avoiding saddle points with zero order methods: No gradients required.
\newblock \emph{Advances in neural information processing systems}, 32, 2019.

\bibitem[Wang(2023)]{wang2023sharp}
Wang, T.
\newblock On sharp stochastic zeroth-order hessian estimators over {R}iemannian manifolds.
\newblock \emph{Information and Inference: A Journal of the IMA}, 12\penalty0 (2):\penalty0 787--813, 2023.

\bibitem[Wang et~al.(2021)Wang, Huang, and Li]{wang2021greene}
Wang, T., Huang, Y., and Li, D.
\newblock From the greene--wu convolution to gradient estimation over {R}iemannian manifolds.
\newblock \emph{arXiv preprint arXiv:2108.07406}, 2021.

\bibitem[Zhang \& Gu(2022)Zhang and Gu]{zhang2022faster}
Zhang, H. and Gu, B.
\newblock Faster gradient-free methods for escaping saddle points.
\newblock In \emph{The Eleventh International Conference on Learning Representations}, 2022.

\bibitem[Zhang \& Sra(2016)Zhang and Sra]{zhang2016first}
Zhang, H. and Sra, S.
\newblock First-order methods for geodesically convex optimization.
\newblock In \emph{Conference on Learning Theory}, pp.\  1617--1638. PMLR, 2016.

\bibitem[Zhang \& Sra(2018)Zhang and Sra]{zhang2018towards}
Zhang, H. and Sra, S.
\newblock Towards riemannian accelerated gradient methods.
\newblock \emph{arXiv preprint arXiv:1806.02812}, 2018.

\bibitem[Zhang et~al.(2022)Zhang, Xiong, and Gu]{zhang2022zeroth}
Zhang, H., Xiong, H., and Gu, B.
\newblock Zeroth-order negative curvature finding: Escaping saddle points without gradients.
\newblock \emph{Advances in Neural Information Processing Systems}, 35:\penalty0 38332--38344, 2022.

\end{thebibliography}
\bibliographystyle{icml2024}

\newpage
\appendix
\onecolumn

\section{Further Related work}\label{sec.related work}
\paragraph{Zeroth-order optimization on Riemannian manifolds.} Riemannian zeroth-order algorithms typically involve two key steps: constructing Riemannian zeroth-order estimators and applying them with standard optimization algorithms, such as Riemannian gradient descent. Through noisy evaluations of the objective function, \citet{li2023stochastic} studied the randomized zeroth-order estimators for the Riemannian gradient and Hessian, extending the Gaussian smoothing technique \cite{nesterov2017random, balasubramanian2022zeroth} onto the Riemannian manifold. Subsequently, \citet{wang2021greene} proposed an alternative zeroth-order gradient estimator based on the Greene–Wu convolution over Riemannian manifolds, demonstrating superior approximation quality compared to the approach by \citet{li2023stochastic}. When it comes to the Riemannian Hessian, \citet{wang2023sharp} introduced a novel Riemannian zeroth-order estimator that relies solely on constant function evaluations. Turning to Riemannian zeroth-order algorithms, \citet{fan2021learning} first proposed a Riemannian meta-optimization method that learns a gradient-free optimizer without theoretical guarantees. \citet{li2023stochastic} studied several zeroth-order algorithms for stochastic Riemannian optimization, presenting the first complexity results. Subsequently, they improved sample complexities by introducing zeroth-order Riemannian averaging stochastic approximation algorithms in \cite{li2023zeroth}. Moreover, \citet{maass2022tracking} studied the exploration of zeroth-order algorithms in the context of Riemannian online learning.

\paragraph{Acceleration on Riemannian manifolds.}  The main challenge in Riemannian optimization arises from the nonlinear structure of Riemannian manifolds, and two powerful techniques have been developed. The first involves leveraging trigonometric comparison inequalities \cite{zhang2016first,alimisis2021momentum}, while the second utilizes the tangent space step \cite{criscitiello2019efficiently,criscitiello2022accelerated}. In cases where the objective function is geodesically convex \cite{bishop1969manifolds,bridson2013metric}, a recent line of work \cite{liu2017accelerated,zhang2018towards,lin2020accelerated,alimisis2021momentum,jin2022understanding,kim2022accelerated} focused on generalizing Nesterov's accelerated update to Riemannian optimization, mirroring the well-known convergence result of accelerated gradient descent on Euclidean convex optimization. Moreover, outside the geodesic convexity, \citet{criscitiello2022accelerated} established the extension of Euclidean nonconvex acceleration techniques \cite{jin2018accelerated,carmon2018accelerated} to Riemannian manifolds, improving the convergence rate compared to Riemannian gradient descent \cite{boumal2019global,criscitiello2019efficiently}.

\section{Discussion of Assumptions}\label{sec.discussion assum}
In the paper, the Lipschitz-type continuity of the pullback function follows from previous works \cite{boumal2019global,agarwal2021adaptive,criscitiello2019efficiently,criscitiello2022accelerated}, and a detailed comparison of the parallel transport based Lipschitz continuity, such as
\begin{equation*}
    \left\|\grad f(x) - \Gamma_y^x \grad f(y)\right\| \le O(d_{\M}(x,y))
\end{equation*}
is provided in the textbook \cite{boumal2023introduction}, where $\Gamma_y^x: \T_y\M \rightarrow \T_x\M$ denotes parallel transport from $y$ to $x$ along any minimizing geodesic, and $d_{\M}(x,y)$ is the Riemannian distance. Since our interest is developing Riemannian zeroth-order algorithms, the Hessian Lipschitz continuity in Assumption \ref{assum.lipschitz hessian} is stronger compared to those in \cite{criscitiello2019efficiently} and \cite{criscitiello2022accelerated}. Whereas in the special case where $\M = \mathbb{R}^d$ and $\retr_x(s) = x+s$, both Assumptions \ref{assum.lipschitz gradient} and \ref{assum.lipschitz hessian} reduce to the standard Lipschitz continuity in Euclidean space.

For the well-behaved retraction mapping (Assumption \ref{assum.retraction}), when the sectional curvature and the covariant derivative of the Riemann curvature endomorphism are both bounded, exponential mapping ensures it holds (Theorem 2.7 in \cite{criscitiello2022accelerated}). For more details, readers can refer to \cite{agarwal2021adaptive,criscitiello2022accelerated, boumal2023introduction}.

\section{Auxiliary Lemmas}
We first list the concept of the adjoint of a linear operator, which is essential in bridging the differential and Hessian of a function on a manifold and their counterparts obtained by interplay with retraction map.
\begin{definition}
Let $E$ and $E'$ be two Euclidean spaces, with inner products $\langle,\rangle_a$ and $\langle,\rangle_b$ respectively. Let $A:E\rightarrow E'$ be a linear operator. The adjoint of $A$ is a linear operator $A^*:E'\rightarrow E$ defined by this property:
\[
\forall u\in E,v\in E',\ \ \ \langle A(u),v\rangle_b=\langle u,A^*(v)\rangle_a.
\]
In particular, if $A$ maps $E$ to $E$ equipped with an inner product $\langle,\rangle$ and 
\[
\forall u,v\in E,\ \ \ \langle A(u),v\rangle=\langle u,A(v)\rangle,
\]
this is, if $A=A^*$, we say $A$ is self-adjoint.
\end{definition}

Several useful lemmas and inequalities are presented below.
\begin{lemma}[Lemma 2.5 in \cite{criscitiello2022accelerated}]\label{lemma.pullback gradient hessian}
    For $f: \mathcal{M} \rightarrow \mathbb{R}$ twice continuously differentiable, $x \in \mathcal{M}$ and $s \in \mathrm{T}_x \mathcal{M}$, with $T_{x, s}^*$ denoting the adjoint of $T_{x, s}$,
    $$
    \nabla \hat{f}_x(s)=T_{x, s}^* \operatorname{grad} f\left(\operatorname{Retr}_x(s)\right), \quad \nabla^2 \hat{f}_x(s)=T_{x, s}^* \operatorname{Hess} f\left(\operatorname{Retr}_x(s)\right) T_{x, s}+W_s,
    $$
    where $T_{x,s}$ si the differential of $\retr_x$ at $s$ (a linear operator):
    \begin{equation*}
        T_{x,s} = \operatorname{D}\retr_x(s): \T_x\M \rightarrow \T_{\retr_x(s)}\M,
    \end{equation*}
    and $W_s$ is a self-adjoint linear operator on $\T_x \M$ defined through polarization by
    $$
    \left\langle W_s[\dot{s}], \dot{s}\right\rangle=\left\langle\operatorname{grad} f\left(\operatorname{Retr}_x(s)\right), \gamma_{x, s}^{\prime \prime}(0)\right\rangle,
    $$
    with $\gamma_{x, s}^{\prime \prime}(0) \in \mathrm{T}_{\operatorname{Retr}_x(s)} \mathcal{M}$ the intrinsic acceleration on $\mathcal{M}$ of $\gamma(\tau)=\operatorname{Retr}_x(s+\tau \dot{s})$ at $\tau=0$.
\end{lemma}

\begin{lemma}[Mechanism in \cite{li2022restarted}]\label{fact}
    For the tangent space step (Subroutine \ref{sub.TSS}), denote $\mK$ to be the iteration number when the ``if condition" on Line 7 triggers, i.e.
    \begin{equation*}
        \mK = \min_{k} \left\{k: k \sum_{j=0}^{k-1} \|s_x^{j+1} - s_x^{j}\|^2 > B^2 \right\}.  
    \end{equation*}
    Then for each $k = 0, 1, \cdots, \mK-1$, it holds that
    \begin{align*}
        &\|s_x^k - s_x^0\| \le B, \\
        &\|y_x^k - s_x^0\| \le 2B.
    \end{align*}
    When the "if condition" does not trigger, for all $k = 0, 1, \cdots, K$, it holds that
    \begin{align*}
        &\|s_x^k - s_x^0\| \le B, \\
        &\|y_x^k - s_x^0\| \le 2B.
    \end{align*}
\end{lemma}

\begin{lemma}[Young's inequality]
If $a \geq 0$ and $b \geq 0$ are nonnegative real numbers and if $p>1$ and $q>1$ are real numbers such that $\frac{1}{p}+\frac{1}{q}=1$, then
$$
a b \leq \frac{a^p}{p}+\frac{b^q}{q} .
$$

Equality holds if and only if $a^p=b^q$. Specifically, for any $\epsilon > 0$, it holds that
$$
a b \leq \frac{a^2}{2 \varepsilon}+\frac{\varepsilon b^2}{2}.
$$
\end{lemma}

\begin{lemma}[Minkowski's inequality]
Given $x_1, \ldots, x_n \in \mathbb{R}$ and $y_1, \ldots, y_n \in \mathbb{R}$, for any $p > 0$, it holds that 
$$
\left(\sum_{k=1}^n\left|x_k+y_k\right|^p\right)^{\frac{1}{p}} \leq\left(\sum_{k=1}^n\left|x_k\right|^p\right)^{\frac{1}{p}}+\left(\sum_{k=1}^n\left|y_k\right|^p\right)^{\frac{1}{p}}.
$$
    
\end{lemma}
\section{Approximation Error of the Estimator}\label{sec.property estimator}
\begin{lemma}\label{lemma.zero order estimator error}
    Suppose that Assumption \ref{assum.lipschitz gradient} and \ref{assum.lipschitz hessian} holds For any smoothing parameter $\mu \in (0, b)$ and $(x, s_x) \in \T\M$ satisfying $\|s_x\| \in \mathbb{B}_{x, b-\mu}(0)$, the Riemannian coordinate-wise zeroth-order estimator in Definition \ref{def.ZO estimators} satisfies
    \begin{equation*}
        \left\|g_x(s_x; \mu) - \nabla \hat f_x(s_x)\right\| \le \min\left\{\frac{l\mu\sqrt{d}}{2}, \frac{\rho\mu^2\sqrt{d}}{6}\right\}.
    \end{equation*}
\end{lemma}
\begin{proof}
    First note that
    \begin{equation}\label{eq.estimator gradient lip}
        \begin{aligned}
    &\left| \hat f_x(s_x+\mu e_i) - \hat f_x(s_x-\mu e_i) - 2\mu \inner{\nabla \hat f_x(s_x)}{e_i}\right| \\
    = &\left|\left(\hat f_x(s_x+\mu e_i) - \hat f_x(s_x) - \mu \inner{\nabla \hat f_x(s_x)}{e_i}\right) - \left(\hat f_x(s_x-\mu e_i) - \hat f_x(s_x) + \mu \inner{\nabla \hat f_x(s_x)}{e_i}\right)\right| \\
    \le &\left|\hat f_x(s_x+\mu e_i) - \hat f_x(s_x) - \mu \inner{\nabla \hat f_x(s_x)}{e_i}\right| + \left|\hat f_x(s_x-\mu e_i) - \hat f_x(s_x) + \mu \inner{\nabla \hat f_x(s_x)}{e_i}\right| \\
    \le &l\mu^2,  
        \end{aligned}
    \end{equation}
    where the last inequality holds due to the Lipschitz continuity of $\nabla \hat f_x(\cdot)$ in Assumption \ref{assum.lipschitz gradient}. Consequently, we have
    \begin{align*}
       &\left\|g_x(s_x; \mu) - \hat \nabla f_x(s)\right\| \\
       =   &\left\|\sum_{i=1}^d \frac{\hat f_x(s_x+\mu e_i) - f_x(s_x-\mu e_i)}{2\mu}e_i - \sum_{i=1}^d \inner{\nabla \hat f_x(s_x)}{e_i} e_i\right\| \\
       = &\frac{1}{2\mu}\left\|\sum_{i=1}^d \left(\hat f_x(s_x+\mu e_i) - \hat f_x(s_x-\mu e_i) - 2\mu \inner{\nabla \hat f_x(s_x)}{e_i}\right)e_i\right\| \\
       = &\frac{1}{2\mu}\sqrt{\sum_{i=1}^d\left(\hat f_x(s_x+\mu e_i) - \hat f_x(s_x-\mu e_i) - 2\mu \inner{\nabla \hat f_x(s_x)}{e_i}\right)^2} \\
       \le &\frac{1}{2\mu} \sqrt{dl^2\mu^4} \\
       = &\frac{l\mu\sqrt{d}}{2}.
    \end{align*}
    Note that it also holds that
    \begin{align*}
       &\hat f_x(s_x+\mu e_i) - \hat f_x(s_x-\mu e_i) - 2\mu \inner{\nabla \hat f_x(s_x)}{e_i} \\
       = &\left(\hat f_x(s_x+\mu e_i) - \hat f_x(s_x) - \mu \inner{\nabla \hat f_x(s_x)}{e_i} - \frac{\mu^2}{2}\inner{\nabla^2 \hat f_x(s_x) e_i}{e_i}\right) \\
    - &\left(\hat f_x(s_x-\mu e_i) - \hat f_x(s_x) + \mu \inner{\nabla \hat f_x(s_x)}{e_i} - \frac{\mu^2}{2}\inner{\nabla^2 \hat f_x(s_x) e_i}{e_i}\right),
    \end{align*}
    and thus the same argument in \eqref{eq.estimator gradient lip} gives
    \begin{equation*}
        \left| \hat f_x(s_x+\mu e_i) - \hat f_x(s_x-\mu e_i) - 2\mu \inner{\nabla \hat f_x(s_x)}{e_i}\right| \le \frac{\rho\mu^3}{3}.
    \end{equation*}
    Similarly, we establish
    \begin{equation*}
        \left\|g_x(s_x; \mu) - \hat \nabla f_x(s)\right\| \le \frac{1}{2\mu}\sqrt{d\left(\frac{\rho\mu^3}{3}\right)^2} = \frac{\rho\mu^2\sqrt{d}}{6}.
    \end{equation*}
    Therefore, we can conclude
    \begin{equation*}
        \left\|g_x(s_x; \mu) - \nabla \hat f_x(s_x)\right\| \le \min\left\{\frac{l\mu\sqrt{d}}{2}, \frac{\rho\mu^2\sqrt{d}}{6}\right\}.
    \end{equation*}
\end{proof}
For simplicity, we use $\mathbf{E}(\mu) = \min\left\{\frac{l\mu\sqrt{d}}{2}, \frac{\rho\mu^2\sqrt{d}}{6}\right\}$ to represent the upper bound of approximation error of the Riemannian coordinate-wise zeroth-order estimator. This notation is widely used throughout the non-asymptotic convergence analysis. 

\section{Proofs of Non-asymptotic Convergence Analysis}\label{sec.proof non-asymptotic}
In the following non-asymptotic analysis, the magnitudes of parameters in RAZGD (Algorithm \ref{alg.AZO}) are set as:
\begin{equation}\label{eq.parameter setting}
    \eta = \frac{1}{4l}, ~B = \tilde{\mathcal{O}}\left(\epsilon^{\frac{1}{2}}\right), ~\theta = \mathcal{O}\left(\epsilon^{\frac{1}{4}}\right), ~ r = \mathcal{O}(\epsilon), ~ K =  \tilde{\mathcal{O}}\left(\epsilon^{-\frac{1}{4}}\right).
\end{equation}

\subsection{Riemannian zeroth-order gradient descent step}
For the iterate $x_t$ with a relatively large zeroth-order estimator, i.e. $\|g_{x_t}(0;\mu)\| \ge lB$, we show that the Riemannian zeroth-order gradient descent step (Subroutine \ref{sub.RZGD}) results in the function value decrease of $\mathcal{O}(B^2)$.

\begin{lemma}\label{lemma.function value decrease large gradient}
        Suppose that Assumption \ref{assum.lipschitz gradient} and \ref{assum.lipschitz hessian} hold. Under the parameter setting \eqref{eq.parameter setting}, choose a reasonably small $\mu$ such that $\mathbf{E}(\mu) \le \frac{lB}{2}$ in Algorithm \ref{alg.AZO}. Then, for the iterate $x_t$ satisfying $\|g_{x_t}(0; \mu)\| \ge lB$, we have:
    \begin{equation*}
        f(x_{t+1}) \le f(x_t) - \min\left\{\frac{lB^2}{16}, lb^2\right\}.
    \end{equation*}
\end{lemma}
\begin{proof}
    First, consider the scenario where $\|g_{x_t}(0;\mu)\| \le \frac{b}{\eta}$; thus, $\eta g_{x_t}(0;\mu) \in \mathbb{B}_{x_t,b}(0)$, ensuring that local Lipschitz continuity holds. Based on Assumption \ref{assum.lipschitz gradient}, we have
    \begin{subequations}
        \begin{align}
            &f(x_{t+1}) \notag \\
            =  &f(\retr_{x_t}(-\eta g_{x_t}(0;\mu))) \notag \\
            =   &\hat f_{x_t}(-\eta g_{x_t}(0;\mu)) \notag \\
            \le &\hat f_{x_t}(0) -\eta \inner{\nabla \hat f_{x_t}(0)}{g_{x_t}(0;\mu)} + \frac{l\eta^2}{2}\norm{g_{x_t}(0;\mu)}^2 \notag \\
            =   &f(x_t) - \frac{\eta}{2}\left(\norm{\nabla \hat f_{x_t}(0)}^2 + \norm{g_{x_t}(0;\mu)}^2 - \norm{\nabla \hat f_{x_t}(0) - g_{x_t}(0;\mu)}^2\right) + \frac{l\eta^2}{2}\norm{g_{x_t}(0;\mu)}^2 \notag \\
            \le &f(x_t) - \frac{\eta}{2}\left(1 - l\eta\right)\norm{g_{x_t}(0;\mu)}^2 + \frac{\eta}{2} \mathbf{E}(\mu)^2. \notag
        \end{align}
    \end{subequations}
    Substituting $\eta = \frac{1}{4l}$, $\mathbf{E}(\mu) \le \frac{lB}{2}$ and $\|g_{x_t}(0; \mu)\| \ge lB$ gives that
    \begin{equation*}
        f(x_{t+1}) \le f(x_t) - \frac{lB^2}{16}.
    \end{equation*}
    For the extremely large estimator $\|g_{x_t}(0;\mu)\| \ge \frac{b}{\eta}$, similarly, it holds that
    \begin{subequations}
        \begin{align}
            &f(x_{t+1}) \notag \\
            =   &\hat f_{x_t}(-\alpha\eta g_{x_t}(0;\mu)) \notag \\
            \le &\hat f_{x_t}(0) -\alpha\eta \inner{\nabla \hat f_{x_t}(0)}{g_{x_t}(0;\mu)} + \frac{l\alpha^2\eta^2}{2}\norm{g_{x_t}(0;\mu)}^2 \notag \\
            =   &f(x_t) - \frac{\alpha\eta}{2}\left(\norm{\nabla \hat f_{x_t}(0)}^2 + \norm{g_{x_t}(0;\mu)}^2 - \norm{\nabla \hat f_{x_t}(0) - g_{x_t}(0;\mu)}^2\right) + \frac{l\alpha^2\eta^2}{2}\norm{g_{x_t}(0;\mu)}^2 \notag \\
            \le &f(x_t) - \frac{4lb^2}{2\alpha} + \frac{\alpha\eta}{2} \mathbf{E}(\mu)^2 + \frac{lb^2}{2} \label{subeq.alpha 1} \\
            \le  &f(x_t) - lb^2 \label{subeq.alpha 2},
        \end{align}
    \end{subequations}
    where we use $\alpha\|g_{x_t}(0;\mu)\| = \frac{b}{\eta}$ and $\eta = \frac{1}{4l}$ in \eqref{subeq.alpha 1}, and \eqref{subeq.alpha 2} holds because $\alpha < 1$ and $\mathbf{E}(\mu)^2 \le \frac{l^2B^2}{4} \le \frac{lb^2}{2}$. Therefore, we conclude
    \begin{equation*}
        f(x_{t+1}) \le f(x_t) - \min\left\{\frac{lB^2}{16}, lb^2\right\}.
    \end{equation*}
\end{proof}

\subsection{Tangent space step: function value decrease}
In this subsection, we establish that the tangent space step results in the function value decrease for the case when the "if condition" (Line 8 of Subroutine \ref{sub.TSS}) triggers. According to Lemma \ref{fact}, we know that when the "if condition" triggers, for each $k = 0, 1, \ldots, \mK-1$, $s_{x_t}^k \in \mathbb{B}_{x_t,b}(0)$ and $y_{x_t}^k \in \mathbb{B}_{x_t,b}(0)$. The following lemma states that $s_{x_t}^{\mK}$ stays within the ball $\mathbb{B}_{x_t,b}(0)$ as well, and thus, the local Lipschitz continuity holds for all iterates in the tangent space step.

\begin{lemma}\label{lemma.bound all tss iterate}
        Suppose that Assumption \ref{assum.lipschitz gradient} and \ref{assum.lipschitz hessian} hold. Under the parameter setting \eqref{eq.parameter setting}, choose a reasonably small $\mu$ such that $\mathbf{E}(\mu) \le \frac{lB}{2}$ holds in Algorithm \ref{alg.AZO}. Then, for the tangent space step at iterate $x_t$, when ``if condition" triggers, we have:
    \begin{equation*}
        \|\nabla \hat f_{x_t}(y_{x_t}^{\mK-1})\| \le 4lB, ~\text{and}~\|s_{x_t}^\mK - s_{x_t}^0\| \le 4B.
    \end{equation*}
\end{lemma}
\begin{proof}
    By the mechanism of Algorithm \ref{alg.AZO}, we know that the zeroth-order estimator $g_{x_t}(0;\mu)$ satisfies $\|g_{x_t}(0;\mu)\| \le lB$. Recall $s_{x_t}^0 = \xi_t \sim \Uniform\left(\mathbb{B}_{x_t, r}(0)\right)$, we have
    \begin{subequations}
        \begin{align}
            &\|\nabla \hat f_{x_t}(s_{x_t}^0)\| \notag \\
            \le &\|\nabla \hat f_{x_t}(s_{x_t}^0) - \nabla \hat f_{x_t}(0)\| + \|\nabla \hat f_{x_t}(0) - g_{x_t}(0;\mu)\| + \|g_{x_t}(0;\mu)\| \notag \\
            \le &l \cdot \|\xi_t\| + \mathbf{E}(\mu) + lB \notag \\
            \le &2lB, \notag
        \end{align}
    \end{subequations}
    where the last inequality uses $\|\xi_t\| = r = \mathcal{O}(\epsilon) \le \frac{B}{2}$. Therefore, we could upper bound $\|\nabla \hat f_{x_t}(y_{x_t}^{\mK-1})\|$ as
    \begin{equation*}
        \|\nabla \hat f_{x_t}(y_{x_t}^{\mK-1})\| \le \|\nabla \hat f_{x_t}(y_{x_t}^{\mK-1}) - \nabla \hat f_{x_t}(s_{x_t}^{0})\| + \|\nabla \hat f_{x_t}(s_{x_t}^0)\| \le l\|y_{x_t}^{\mK-1} - s_{x_t}^{0}\| + 2lB \le 4lB.
    \end{equation*}
    Since $s_{x_t}^\mK = y_{x_t}^{\mK-1} - \eta g_{x_t}(y_{x_t}^{\mK-1};\mu)$, it follows that
    \begin{subequations}
        \begin{align}
            &\|s_{x_t}^\mK - s_{x_t}^0\| \notag \\
            \le &\|s_{x_t}^\mK - y_{x_t}^{\mK - 1}\| + \|y_{x_t}^{\mK - 1} - s_{x_t}^{0}\| \notag \\
            \le &\eta \|g_{x_t}(y_{x_t}^{\mK-1};\mu)\| + 2B \notag \\
            \le &\eta \|g_{x_t}(y_{x_t}^{\mK-1};\mu) - \nabla \hat f_{x_t}(y_{x_t}^{\mK-1})\| + \eta \|\nabla \hat f_{x_t}(y_{x_t}^{\mK-1})\| + 2B \notag \\
            \le &\eta \mathbf{E}(\mu) + \eta \cdot 4lB + 2B \notag \\
            \le &4B, \notag
        \end{align}
    \end{subequations}
    where the last inequality holds due to that $\eta = \frac{1}{4l}$ and $\mathbf{E}(\mu) \le \frac{lB}{2}$.
\end{proof}

To establish the function value decrease in the tangent space step, we mimic the proof strategy in \cite{li2022restarted}. First note that $\nabla^2 \hat f_{x_t}(s_t^0)$ is self-adjoint, there exists a basis of eigenvectors $\{u_j\}_{j=1}^d$ satisfying
\begin{equation*}
     \nabla^2 \hat f_{x_t}(s_t^0) u_j = \lambda_j u_j,
\end{equation*}
where $\lambda_1, \cdots, \lambda_d$ are associated eigenvalues. Based on the basis $\{u_j\}_{j=1}^d$ and local coordinate $\{e_j\}_{j=1}^d$ of tangent space $\T_{x_t}\M$, we introduce the following notations for any given $s_{x_t}, \nabla \hat f_{x_t}(\cdot) \in \T_{x_t}\M$:
\begin{align*}
    &\tilde s_{x_t,j} = \inner{s_{x_t}}{u_j},~s_{x_t,j} = \inner{s_{x_t}}{e_j}, ~j = 1, \cdots, d, \\
    &\tilde \nabla_j \hat f_{x_t}(\cdot) = \inner{\nabla \hat f_{x_t}(\cdot)}{u_j},~\nabla_j \hat f_{x_t}(\cdot) = \inner{\nabla \hat f_{x_t}(\cdot)}{e_j}, ~j = 1, \cdots, d.
\end{align*}
Therefore, it holds that
\begin{align*}
    &s_{x_t} = \sum_{j=1}^d \tilde s_{x_t,j} u_j = \sum_{j=1}^d s_{x_t,j} e_j, \\
    &\|s_{x_t}\|^2 = \sum_{j=1}^d |\tilde s_{x_t,j}|^2 = \sum_{j=1}^d |s_{x_t,j}|^2,\\
    &\hat f_{x_t}(\cdot) = \sum_{j=1}^d \tilde \nabla_j \hat f_{x_t}(\cdot) u_j = \sum_{j=1}^d \nabla_j \hat f_{x_t}(\cdot) e_j,\\
    &\|\hat f_{x_t}(\cdot)\|^2 = \sum_{j=1}^d |\tilde \nabla_j \hat f_{x_t}(\cdot)|^2 = \sum_{j=1}^d |\nabla_j \hat f_{x_t}(\cdot)|^2.
\end{align*}
\begin{lemma}\label{lemma.tss function value with h}
    Suppose that Assumption \ref{assum.lipschitz gradient} and \ref{assum.lipschitz hessian} hold. Under the parameter setting \eqref{eq.parameter setting}, choose a reasonably small $\mu$ such that $\mathbf{E}(\mu) \le \frac{lB}{2}$ in Algorithm \ref{alg.AZO}. Then, for the tangent space step at iterate $x_t$, when the ``if condition" triggers, we have
    \begin{equation*}
        \hat f_{x_t}(s_{x_t}^\mK) \le \hat f_{x_t}(s_{x_t}^0) + \frac{32\rho B^3}{3} + \sum_{j=1}^d h_j(\tilde s_{x_t,j}^\mK),
    \end{equation*}
    where
    \begin{equation*}
        h_j(z) = \inner{\tilde \nabla_j \hat f_{x_t}(s_{x_t}^0)}{z - \tilde s_{x_t,j}^0} + \frac{\lambda_j}{2}(z - \tilde s_{x_t,j}^0)^2, \ j = 1, \ldots, d.
    \end{equation*}
    are one-dimensional quadratic functions.
\end{lemma}
\begin{proof}
    From the Hessian Lipschitz continuity (Assumption \ref{assum.lipschitz hessian}), we have
    \begin{subequations}
    \begin{align}
        &\hat f_{x_t}(s_{x_t}^\mK) \notag \\
        \le &\hat f_{x_t}(s_{x_t}^0) + \inner{\nabla \hat f_{x_t}(s_{x_t}^0)}{s_{x_t}^\mK - s_{x_t}^0} + \frac{1}{2}\inner{\nabla^2 \hat f_{x_t}(s_{x_t}^0)(s_{x_t}^\mK - s_{x_t}^0)}{s_{x_t}^\mK - s_{x_t}^0} + \frac{\rho}{6}\|s_{x_t}^\mK - s_{x_t}^0\|^3 \notag \\
        \le &\hat f_{x_t}(s_{x_t}^0) + \inner{\nabla \hat f_{x_t}(s_{x_t}^0)}{s_{x_t}^\mK - s_{x_t}^0} + \frac{1}{2}\inner{\nabla^2 \hat f_{x_t}(s_{x_t}^0)(s_{x_t}^\mK - s_{x_t}^0)}{s_{x_t}^\mK - s_{x_t}^0} + \frac{32\rho B^3}{3}, \label{subeq.9a} \\
        = &\hat f_{x_t}(s_{x_t}^0)  + \frac{32\rho B^3}{3} + \sum_{j=1}^d \inner{\tilde \nabla_j \hat f_{x_t}(s_{x_t}^0)}{\tilde s_{x_t,j}^\mK - \tilde s_{x_t,j}^0} + \frac{\lambda_j}{2}(\tilde s_{x_t,j}^\mK - \tilde s_{x_t,j}^0)^2 \label{subeq.9b} \\
        = &\hat f_{x_t}(s_{x_t}^0)  + \frac{32\rho B^3}{3} + \sum_{j=1}^d h_j(\tilde s_{x_t,j}^\mK), \notag
    \end{align}
\end{subequations}
where \eqref{subeq.9a} comes from Lemma \ref{lemma.bound all tss iterate}, and \eqref{subeq.9b} holds because $\{u_j\}_{j=1}^d$ forms a standard basis of the tangent space $\T_{x_t}\M$.
\end{proof}
The above lemma indicates that it is sufficient to analyze the behavior of one-dimensional quadratic functions $h_j(\tilde s_{x_t,j}^\mK)$, $j = 1, \ldots, d$. Recall the $k$-th update in the tangent space step at iterate $x_t$:
\begin{align*}
    &y_{x_t}^k = s_{x_t}^k + (1 - \theta) (s_{x_t}^k - s_{x_t}^{k-1}), \\
    &s_{x_t}^{k+1} = y_{x_t}^k - \eta g_{x_t}(y_{x_t}^k; \mu).
\end{align*}
It equivalents as
\begin{equation}\label{eq.agd for h}
    \begin{aligned}
    &\tilde y_{x_t, j}^k = \tilde s_{x_t, j}^k + (1 - \theta) (\tilde s_{x_t, j}^k - \tilde s_{x_t, j}^{k-1}), \\
    &\tilde s_{x_t, j}^{k+1} = \tilde y_{x_t, j}^k - \eta \nabla h_j (\tilde y_{x_t, j}^k) - \eta \mE_{x_t, j}^k,
\end{aligned}
\end{equation}
where 
\begin{align*}
    \mE_{x_t, j}^k = \inner{g_{x_t}(y_{x_t}^k; \mu)}{u_j} - \nabla h_j (\tilde y_{x_t, j}^k) := \tilde g_{x_t, j}(y_{x_t}^k; \mu) - \nabla h_j (\tilde y_{x_t, j}^k).
\end{align*}
for all $j = 1, \cdots, d$. As \citet{li2022restarted} pointed out, the $k$-th update in the tangent space step can be viewed as applying inexact accelerated gradient descent to $h_j(\cdot)$ with the error $\mE_{x_t, j}^k$. The following lemma describes the error that could be controlled.
\begin{lemma}\label{lemma.upper bound of mE}
        Suppose that Assumption \ref{assum.lipschitz gradient} and \ref{assum.lipschitz hessian} hold. Under the parameter setting \eqref{eq.parameter setting}, choose a reasonably small $\mu$ such that $\mathbf{E}(\mu) \le \frac{lB}{2}$ in Algorithm \ref{alg.AZO}. Then, for the tangent space step at iterate $x_t$, the error $\mE_{x_t, j}^k$ in update \eqref{eq.agd for h} satisfies
    \begin{equation*}
        \sum_{j=1}^d |\mE_{x_t, j}^0|^2 \le \mathbf{E}(\mu)^2, \ \text{and} \ \sum_{j=1}^d |\mE_{x_t, j}^k|^2 \le 2\mathbf{E}(\mu)^2 + 8\rho^2 B^4, ~\forall k \ge 1.
    \end{equation*}
\end{lemma}
\begin{proof}
    For any $j = 1, \cdots, d$, since $y_{x_t}^0 = s_{x_t}^0$, it holds that
    \begin{equation*}
       \mE_{x_t, j}^0 = \tilde g_{x_t, j}(y_{x_t}^0; \mu) - \nabla h_j (\tilde y_{x_t, j}^0) = \tilde g_{x_t, j}(y_{x_t}^0; \mu) - \tilde \nabla_j \hat f_{x_t}(\tilde y_{x_t}^0). 
    \end{equation*}
    Summing over $j$ gives
    \begin{align*}
        \sum_{j=1}^d |\mE_{x_t, j}^0|^2
        = &\sum_{j=1}^d | \tilde g_{x_t, j}(y_{x_t}^0; \mu) - \tilde \nabla_j \hat f_{x_t}(y_{x_t}^0)|^2 \\
        = &\|g_{x_t}(y_{x_t}^0; \mu) - \nabla \hat f_{x_t}(y_{x_t}^0)\|^2 \\
        \le &\mathbf{E}(\mu)^2.
    \end{align*}
    For any $k \ge 1$, by the definition of $\mE_{x_t, j}^k$, it follows
    \begin{subequations}
        \begin{align*}
            \sum_{j=1}^d |\mE_{x_t, j}^k|^2
            = &\sum_{j=1}^d |\tilde g_{x_t, j}(y_{x_t}^k; \mu) - \nabla h_j (\tilde y_{x_t, j}^k)|^2 \notag \\
            \le & 2\sum_{j=1}^d|\tilde g_{x_t, j}(y_{x_t}^k; \mu) - \tilde \nabla_j \hat f_{x_t}(y_{x_t}^k)|^2 + 2\sum_{j=1}^d|\tilde \nabla_j \hat f_{x_t}(y_{x_t}^k) -\nabla h_j (\tilde y_{x_t, j}^k)|^2. \notag
        \end{align*}
    \end{subequations}
    For the first term, we have
    \begin{equation*}
        \sum_{j=1}^d|\tilde g_{x_t, j}(y_{x_t}^k; \mu) - \tilde \nabla_j \hat f_{x_t}(y_{x_t}^k)|^2 = \|g_{x_t}(y_{x_t}^k; \mu) - \nabla \hat f_{x_t}(y_{x_t}^k)\|^2 \le \mathbf{E}(\mu)^2.
    \end{equation*}
    For the second term, since $\nabla^2 \hat f_{x_t}(s_t^0) u_j = \lambda_j u_j$, $\forall j$, we have
    \begin{subequations}
        \begin{align}
            &\sum_{j=1}^d|\tilde \nabla_j \hat f_{x_t}(y_{x_t}^k) -\nabla h_j (\tilde y_{x_t, j}^k)|^2 \notag \\
            =  &\sum_{j=1}^d|\tilde \nabla_j \hat f_{x_t}(y_{x_t}^k) - \tilde \nabla_j \hat f_{x^t}(s_{x_t}^0) - \lambda_j(\tilde y_{x_t}^k - \tilde s_{x_t}^0)|^2 \notag \\
            =  &\sum_{j=1}^d|\inner{\nabla \hat f_{x_t}(y_{x_t}^k) - \nabla \hat f_{x_t}(s_{x_t}^0) - \lambda_j(y_{x_t, j}^k - s_{x_t}^0)}{u_j}|^2 \notag \\
            =  &\sum_{j=1}^d|\inner{\nabla \hat f_{x_t}(y_{x_t}^k) - \nabla \hat f_{x_t}(s_{x_t}^0) - \nabla^2 \hat f_{x_t}(s_t^0)(y_{x_t}^k - s_{x_t}^0)}{u_j}|^2 \notag \\
            = &\|\nabla \hat f_{x_t}(y_{x_t}^k) - \nabla \hat f_{x_t}(s_{x_t}^0) - \nabla^2 \hat f_{x_t}(s_t^0)(y_{x_t}^k - s_{x_t}^0)\|^2 \notag\\
            \le &\frac{\rho^2}{4}\|y_{x_t}^k - s_{x_t}^0\|^4 \label{subeq.13a}\\
            \le &4\rho^2 B^4, \label{subeq.13b}
        \end{align}
    \end{subequations}
    where \eqref{subeq.13a} is due to the Lipschitz continuity of $\nabla^2 \hat f_{x_t}(\cdot)$, and \eqref{subeq.13b} follows from the Fact \eqref{fact}. Combining all the above inequalities completes the proof.
\end{proof}

Now we proceed to analyze the value of $\sum_{j=1}^d h_j(\tilde s_{x_t,j}^\mK)$. We split it into the following two cases:
\begin{equation*}
    \mathcal{S}_1:=\left\{j: \lambda_j \ge -\frac{\theta}{\eta}\right\} \ \text{and} \ \mathcal{S}_2:=\left\{j: \lambda_j < -\frac{\theta}{\eta}\right\}.
\end{equation*}

\begin{lemma}\label{lemma.tss s1}
    Suppose that Assumption \ref{assum.lipschitz gradient} and \ref{assum.lipschitz hessian} hold. Under the parameter setting \eqref{eq.parameter setting}, choose a reasonably small $\mu$ such that $\mathbf{E}(\mu) \le \frac{lB}{2}$ in Algorithm \ref{alg.AZO}. Then, for the tangent space step at iterate $x_t$, when the ``if condition" triggers, we have:
    \begin{equation*}
        \sum_{j \in \mathcal{S}_1} h_j(\tilde s_{x_t,j}^\mK) \le - \frac{3\theta}{8\eta} \sum_{j \in \mathcal{S}_1} \sum_{k=0}^{\mKb}|\tilde s_{x_t,j}^{k+1} - \tilde s_{x_t,j}^k|^2 + \frac{4\eta\mK}{\theta}\mathbf{E}(\mu)^2 + \frac{16\eta\mK\rho^2B^4}{\theta}.
    \end{equation*}
\end{lemma}
The following proof follows the proof of Lemma 3.2 in \cite{li2022restarted}. We only list the sketch for simplicity.
\begin{proof}
    For any $k = 0, 1, \cdots, \mK-1$ and $j \in \mathcal{S}_1$, as $h_j(\cdot)$ is a one-dimensional quadratic function
        \begin{subequations}
        \begin{align}
            &h_j(\tilde s_{x_t,j}^{k+1}) - h_j(\tilde s_{x_t,j}^{k})\notag \\
            =   &\inner{\nabla h_j(\tilde s_{x_t,j}^{k})}{\tilde s_{x_t,j}^{k+1} - \tilde s_{x_t,j}^k} + \frac{\lambda_j}{2}|\tilde s_{x_t,j}^{k+1} - \tilde s_{x_t,j}^k|^2 \notag \\
            =   &\inner{\nabla h_j(\tilde s_{x_t,j}^{k}) - \nabla h_j(\tilde y_{x_t,j}^{k})}{\tilde s_{x_t,j}^{k+1} - \tilde s_{x_t,j}^k} + \inner{\nabla h_j(\tilde y_{x_t,j}^{k})}{\tilde s_{x_t,j}^{k+1} - \tilde s_{x_t,j}^k} + \frac{\lambda_j}{2}|\tilde s_{x_t,j}^{k+1} - \tilde s_{x_t,j}^k|^2 \notag \\
            =   &\lambda_j\inner{\tilde s_{x_t,j}^{k} - \tilde y_{x_t,j}^{k}}{\tilde s_{x_t,j}^{k+1} - \tilde s_{x_t,j}^k} -\frac{1}{\eta} \inner{\tilde s_{x_t,j}^{k+1} - \tilde y_{x_t,j}^k + \eta \mathcal{E}_{x_t,j}^k}{\tilde s_{x_t,j}^{k+1} - \tilde s_{x_t,j}^k} + \frac{\lambda_j}{2}|\tilde s_{x_t,j}^{k+1} - \tilde s_{x_t,j}^k|^2 \notag \\
            =   & \lambda_j\inner{\tilde s_{x_t,j}^{k} - \tilde y_{x_t,j}^{k}}{\tilde s_{x_t,j}^{k+1} - \tilde s_{x_t,j}^k} -\frac{1}{\eta} \inner{\tilde s_{x_t,j}^{k+1} - \tilde y_{x_t,j}^k}{\tilde s_{x_t,j}^{k+1} - \tilde s_{x_t,j}^k} - \inner{\mathcal{E}_{x_t,j}^k}{\tilde s_{x_t,j}^{k+1} - \tilde s_{x_t,j}^k} + \frac{\lambda_j}{2}|\tilde s_{x_t,j}^{k+1} - \tilde s_{x_t,j}^k|^2 \notag\\
            =   &\frac{\lambda_j}{2}(|\tilde s_{x_t,j}^{k+1} - \tilde y_{x_t,j}^{k}|^2 - |\tilde s_{x_t,j}^{k} - \tilde y_{x_t,j}^{k}|^2 - |\tilde s_{x_t,j}^{k+1} - \tilde s_{x_t,j}^k|^2) - \inner{\mathcal{E}_{x_t,j}^k}{\tilde s_{x_t,j}^{k+1} - \tilde s_{x_t,j}^k} \notag \\
            &+ \frac{\lambda_j}{2}|\tilde s_{x_t,j}^{k+1} - \tilde s_{x_t,j}^k|^2 + \frac{1}{2\eta}(|\tilde s_{x_t,j}^{k} - \tilde y_{x_t,j}^k|^2 - |\tilde s_{x_t,j}^{k+1} - \tilde y_{x_t,j}^k|^2 - |\tilde s_{x_t,j}^{k+1} - \tilde s_{x_t,j}^k|^2) \notag \\
            \le & \frac{\lambda_j}{2}(|\tilde s_{x_t,j}^{k+1} - \tilde y_{x_t,j}^{k}|^2 - |\tilde s_{x_t,j}^{k} - \tilde y_{x_t,j}^{k}|^2) + \frac{2\eta}{\theta}|\mE_{x_t,j}^k|^2 + \frac{\theta}{8\eta}|\tilde s_{x_t,j}^{k+1} - \tilde s_{x_t,j}^k|^2 \notag \\
            &+ \frac{1}{2\eta}(|\tilde s_{x_t,j}^{k} - \tilde y_{x_t,j}^k|^2 - |\tilde s_{x_t,j}^{k+1} - \tilde y_{x_t,j}^k|^2 - |\tilde s_{x_t,j}^{k+1} - \tilde s_{x_t,j}^k|^2), \notag
        \end{align}
    \end{subequations}
    where the last inequality is due to $- \inner{\mathcal{E}_{x_t,j}^k}{\tilde s_{x_t,j}^{k+1} - \tilde s_{x_t,j}^k} \le |\mathcal{E}_{x_t,j}^k| \cdot |\tilde s_{x_t,j}^{k+1} - \tilde s_{x_t,j}^k| \le \frac{2\eta}{\theta}|\mE_{x_t,j}^k|^2 + \frac{\theta}{8\eta}|\tilde s_{x_t,j}^{k+1} - \tilde s_{x_t,j}^k|^2$ by applying Young's inequality. Since $l \ge \lambda_j \ge -\frac{\theta}{\eta}$ holds, it follows
    \begin{align*}
        &(-\frac{1}{2\eta} + \frac{\lambda_j}{2})|\tilde s_{x_t,j}^{k+1} - \tilde y_{x_t,j}^k|^2 \le (-2l + \frac{l}{2})|\tilde s_{x_t,j}^{k+1} - \tilde y_{x_t,j}^k|^2 \le 0,\\
        &-\frac{\lambda_j}{2}|\tilde s_{x_t,j}^{k} - \tilde y_{x_t,j}^k|^2 \le \frac{\theta}{2\eta}|\tilde s_{x_t,j}^{k} - \tilde y_{x_t,j}^k|^2.
    \end{align*}
    Combing the above inequalities implies that
    \begin{subequations}
        \begin{align}
            &h_j(\tilde s_{x_t,j}^{k+1}) - h_j(\tilde s_{x_t,j}^{k}) \notag \\
            \le &\frac{1}{2\eta}(|\tilde s_{x_t,j}^{k} - \tilde y_{x_t,j}^k|^2 - |\tilde s_{x_t,j}^{k+1} - \tilde s_{x_t,j}^k|^2) + \frac{2\eta}{\theta}|\mE_{x_t,j}^k|^2 + \frac{\theta}{8\eta}|\tilde s_{x_t,j}^{k+1} - \tilde s_{x_t,j}^k|^2 + \frac{\theta}{2\eta}|\tilde s_{x_t,j}^{k} - \tilde y_{x_t,j}^k|^2 \notag \\
            =   &\frac{(1-\theta)^2}{2\eta}|\tilde s_{x_t,j}^{k} - \tilde s_{x_t,j}^{k-1}|^2 - (\frac{1}{2\eta} - \frac{\theta}{8\eta})|\tilde s_{x_t,j}^{k+1} - \tilde s_{x_t,j}^k|^2 + \frac{2\eta}{\theta}|\mE_{x_t,j}^k|^2 + \frac{\theta(1-\theta)^2}{2\eta}|\tilde s_{x_t,j}^{k} - \tilde s_{x_t,j}^{k-1}|^2 \notag \\
            =   &\frac{(1+\theta)(1-\theta)^2}{2\eta}|\tilde s_{x_t,j}^{k} - \tilde s_{x_t,j}^{k-1}|^2 - (\frac{1}{2\eta} - \frac{\theta}{8\eta})|\tilde s_{x_t,j}^{k+1} - \tilde s_{x_t,j}^k|^2 + \frac{2\eta}{\theta}|\mE_{x_t,j}^k|^2. \notag
        \end{align}
    \end{subequations}
    To proceed, we define the potential function
    \begin{equation*}
        l_{x_t,j}^{k} = h_j(\tilde s_{x_t,j}^{k}) + \frac{(1+\theta)(1-\theta)^2}{2\eta}|\tilde s_{x_t,j}^{k} - \tilde s_{x_t,j}^{k-1}|^2,
    \end{equation*}
    and it gives 
    \begin{subequations}
        \begin{align}
            &l_{x_t,j}^{k+1} - l_{x_t,j}^{k} \notag\\
            \le &- (\frac{1}{2\eta} - \frac{\theta}{8\eta} - \frac{(1+\theta)(1-\theta)^2}{2\eta})|\tilde s_{x_t,j}^{k+1} - \tilde s_{x_t,j}^k|^2 + \frac{2\eta}{\theta}|\mE_{x_t,j}^k|^2 \notag\\
            \le &- \frac{3\theta}{8\eta}|\tilde s_{x_t,j}^{k+1} - \tilde s_{x_t,j}^k|^2 + \frac{2\eta}{\theta}|\mE_{x_t,j}^k|^2. \notag
        \end{align}
    \end{subequations}
    Summing over $k = 0, 1, \cdots, \mK-1$ and $j \in \mathcal{S}_1$, and using $s_{x_t}^0 = s_{x_t}^{-1}$, we conclude
    \begin{subequations}
        \begin{align}
            &\sum_{j \in \mathcal{S}_1} h_j(\tilde s_{x_t,j}^\mK) \notag\\
            \le &\sum_{j \in \mathcal{S}_1} l_{x_t,j}^\mK \notag\\
            =   &\sum_{j \in \mathcal{S}_1} \sum_{k=0}^{\mKb} (l_{x_t,j}^{k+1} - l_{x_t,j}^k) + l_{x_t,j}^0 \notag\\
            =   &\sum_{j \in \mathcal{S}_1} \sum_{k=0}^{\mKb} (l_{x_t,j}^{k+1} - l_{x_t,j}^k) \notag \\
            \le &- \frac{3\theta}{8\eta} \sum_{j \in \mathcal{S}_1} \sum_{k=0}^{\mKb}|\tilde s_{x_t,j}^{k+1} - \tilde s_{x_t,j}^k|^2 + \frac{2\eta}{\theta}\sum_{j \in \mathcal{S}_1} \sum_{k=0}^{\mKb}|\mE_{x_t,j}^k|^2 \notag\\
            \le  &- \frac{3\theta}{8\eta} \sum_{j \in \mathcal{S}_1} \sum_{k=0}^{\mKb}|\tilde s_{x_t,j}^{k+1} - \tilde s_{x_t,j}^k|^2 + \frac{4\eta\mK}{\theta}\mathbf{E}(\mu)^2 + \frac{16\eta\mK\rho^2B^4}{\theta}, \notag
        \end{align}
    \end{subequations}
    where the last inequality is due to Lemma \eqref{lemma.upper bound of mE}.
\end{proof}
\begin{lemma}\label{lemma.tss s2}
    Suppose that Assumption \ref{assum.lipschitz gradient} and \ref{assum.lipschitz hessian} hold. Under the parameter setting \eqref{eq.parameter setting}, choose a reasonably small $\mu$ such that $\mathbf{E}(\mu) \le \frac{lB}{2}$ in Algorithm \ref{alg.AZO}. Then, for the tangent space step at iterate $x_t$, when the ``if condition" triggers, we have:
    \begin{equation*}
        \sum_{j \in \mathcal{S}_2} h_j(\tilde s_{x_t,j}^\mK) \le -\frac{\theta}{2\eta} \sum_{j \in \mathcal{S}_2} \sum_{k=0}^{\mKb}|\tilde s_{x_t,j}^{k+1} - \tilde s_{x_t,j}^{k}|^2 + \eta\mK \mathbf{E}(\mu)^2 + \frac{\mK}{2B^{\frac{3}{2}}}\mathbf{E}(\mu)^{2} + \frac{\mK B^{\frac{7}{2}}}{2} + \frac{\eta\mK \mathbf{E}(\mu)^2}{\theta} + \frac{4\eta \mK\rho^2 B^4}{\theta}.
    \end{equation*}
\end{lemma}
\begin{proof}
    Let $v_{x_t,j} = \tilde s_{x_t,j}^0 -\frac{1}{\lambda_j}\tilde \nabla_j \hat f_{x_t}(s_{x_t}^0)$, the one-dimensional quadratic function $h_j(\cdot)$ can be rewritten as
    \begin{equation*}
        h_j(z) = \frac{\lambda_j}{2}(z - v_{x_t,j})^2 - \frac{1}{2\lambda_j}|\tilde \nabla_j \hat f_{x_t}(s_{x_t}^0)|^2.
    \end{equation*}
    Consequently, for any $k = 0, 1, \cdots, \mK - 1$ and $j \in \mathcal{S}_2$, we have
    \begin{equation}
        \begin{aligned}
            &h_j(\tilde s_{x_t,j}^{k+1}) - h_j(\tilde s_{x_t,j}^{k})  \\
            =   &\frac{\lambda_j}{2}|\tilde s_j^{k+1} - v_{x_t,j}|^2 - \frac{\lambda_j}{2}|\tilde s_{x_t,j}^{k} - v_{x_t,j}|^2 \\
            =   &\frac{\lambda_j}{2}|\tilde s_{x_t,j}^{k+1} - \tilde s_{x_t,j}^{k}|^2 + \lambda_j\inner{\tilde s_{x_t,j}^{k+1} - \tilde s_{x_t,j}^{k}}{\tilde s_{x_t,j}^{k} - v_{x_t,j}}  \\
            \le &-\frac{\theta}{2\eta}|\tilde s_{x_t,j}^{k+1} - \tilde s_{x_t,j}^{k}|^2 + \lambda_j\inner{\tilde s_{x_t,j}^{k+1} - \tilde s_{x_t,j}^{k}}{\tilde s_{x_t,j}^{k} - v_{x_t,j}}. 
        \end{aligned}
    \end{equation}
    Recall the update in \eqref{eq.agd for h}, it holds that
    \begin{subequations}
        \begin{align}
            &\tilde s_{x_t,j}^{k+1} - \tilde s_{x_t,j}^{k} \notag \\
            = &\tilde y_{x_t,j}^k - \eta \nabla h_j(\tilde y_{x_t,j}^k) - \eta \mE_{x_t,j}^k - \tilde s_{x_t,j}^{k} \notag \\
            = &(1-\theta)(\tilde s_{x_t,j}^{k} - \tilde s_{x_t,j}^{k-1}) - \eta \nabla h_j(\tilde y_{x_t,j}^k) - \eta \mE_{x_t,j}^k \notag \\
            = &(1-\theta)(\tilde s_{x_t,j}^{k} - \tilde s_{x_t,j}^{k-1}) - \eta \lambda_j (\tilde y_{x_t,j}^k - v_j) - \eta \mE_{x_t,j}^k \notag \\
            = &(1-\theta)(\tilde s_{x_t,j}^{k} - \tilde s_{x_t,j}^{k-1}) - \eta \lambda_j (\tilde s_{x_t,j}^k - v_{x_t,j} + (1-\theta)(\tilde s_{x_t,j}^{k} - \tilde s_{x_t,j}^{k-1})) - \eta \mE_{x_t,j}^k. \notag
        \end{align}
    \end{subequations}
    Substituting the above equality into the term $\inner{\tilde s_{x_t,j}^{k+1} - \tilde s_{x_t,j}^{k}}{\tilde s_{x_t,j}^{k} - v_{x_t,j}}$ gives 
    \begin{equation}\label{eq.20}
        \begin{aligned}
            &\inner{\tilde s_{x_t,j}^{k+1} - \tilde s_{x_t,j}^{k}}{\tilde s_{x_t,j}^{k} - v_{x_t,j}}\\
            =  &(1-\theta)\inner{\tilde s_{x_t,j}^{k} - \tilde s_{x_t,j}^{k-1}}{\tilde s_{x_t,j}^{k} - v_{x_t,j}} - \eta\lambda_j|\tilde s_{x_t,j}^{k} - v_{x_t,j}|^2\\ 
            &\underbrace{- \eta\lambda_j(1-\theta)\inner{\tilde s_{x_t,j}^{k} - \tilde s_{x_t,j}^{k-1}}{\tilde s_{x_t,j}^{k} - v_{x_t,j}} - \eta\inner{\mE_{x_t,j}^k}{\tilde s_{x_t,j}^{k} - v_{x_t,j}}}_{\clubsuit}.\\
        \end{aligned}
    \end{equation}
    We first provide a lower bound for the term $\clubsuit$ in \eqref{eq.20}:
    \begin{subequations}
        \begin{align}
            &- \eta\lambda_j(1-\theta)\inner{\tilde s_{x_t,j}^{k} - \tilde s_{x_t,j}^{k-1}}{\tilde s_{x_t,j}^{k} - v_{x_t,j}} - \eta\inner{\mE_{x_t,j}^k}{\tilde s_{x_t,j}^{k} - v_{x_t,j}} \notag \\
            = &\frac{\eta\lambda_j(1-\theta)}{2}(|\tilde s_{x_t,j}^{k} - \tilde s_{x_t,j}^{k-1}|^2 + |v_{x_t,j} - \tilde s_{x_t,j}^{k}|^2 -  |\tilde s_{x_t,j}^{k-1} - v_{x_t,j}|^2)- \eta\inner{\mE_{x_t,j}^k}{\tilde s_{x_t,j}^{k} - v_{x_t,j}} \notag \\
            \ge &\frac{\eta\lambda_j(1-\theta)}{2}(|\tilde s_{x_t,j}^{k} - \tilde s_{x_t,j}^{k-1}|^2 + |\tilde s_{x_t,j}^{k} - v_{x_t,j}|^2) + \frac{\eta}{2\lambda_j(1+\theta)}|\mE_{x_t,j}^k|^2 + \frac{\eta\lambda_j(1+\theta)}{2}|\tilde s_{x_t,j}^{k} - v_{x_t,j}|^2 \label{subeq.21a} \\
            = &\frac{\eta\lambda_j(1-\theta)}{2}|\tilde s_{x_t,j}^{k} - \tilde s_{x_t,j}^{k-1}|^2 + \eta\lambda_j|\tilde s_{x_t,j}^{k} - v_{x_t,j}|^2 + \frac{\eta}{2\lambda_j(1+\theta)}|\mE_{x_t,j}^k|^2 \notag,
        \end{align}
    \end{subequations}
    where \eqref{subeq.21a} holds because $\lambda_j < 0$ and $- \inner{\mE_{x_t,j}^k}{\tilde s_{x_t,j}^{k} - v_{x_t,j}} \ge  \frac{1}{2\lambda_j(1+\theta)}|\mE_{x_t,j}^k|^2 + \frac{\lambda_j(1+\theta)}{2}|\tilde s_{x_t,j}^{k} - v_{x_t,j}|^2$ due to the Young's inequality. Plugging the above inequality back into \eqref{eq.20} gives
    \begin{equation}
        \begin{aligned}
            &\inner{\tilde s_{x_t,j}^{k+1} - \tilde s_{x_t,j}^{k}}{\tilde s_{x_t,j}^{k} - v_{x_t,j}}\\ 
            \ge &(1-\theta)\inner{\tilde s_{x_t,j}^{k} - \tilde s_{x_t,j}^{k-1}}{\tilde s_{x_t,j}^{k} - v_{x_t,j}} + \frac{\eta\lambda_j(1-\theta)}{2}|\tilde s_{x_t,j}^{k} - \tilde s_{x_t,j}^{k-1}|^2 + \frac{\eta}{2\lambda_j(1+\theta)}|\mE_{x_t,j}^k|^2 \\
            = &(1-\theta)\inner{\tilde s_{x_t,j}^{k} - \tilde s_{x_t,j}^{k-1}}{\tilde s_{x_t,j}^{k} - \tilde s_{x_t,j}^{k-1}} + (1-\theta)\inner{\tilde s_{x_t,j}^{k} - \tilde s_{x_t,j}^{k-1}}{\tilde s_{x_t,j}^{k-1} - v_{x_t,j}} + \frac{\eta\lambda_j(1-\theta)}{2}|\tilde s_{x_t,j}^{k} - \tilde s_{x_t,j}^{k-1}|^2 + \frac{\eta}{2\lambda_j(1+\theta)}|\mE_{x_t,j}^k|^2 \\
            = &(1+\frac{\eta\lambda_j}{2})(1-\theta)|\tilde s_{x_t,j}^{k} - \tilde s_{x_t,j}^{k-1}|^2 + (1-\theta)\inner{\tilde s_{x_t,j}^{k} - \tilde s_{x_t,j}^{k-1}}{\tilde s_{x_t,j}^{k-1} - v_{x_t,j}} + \frac{\eta}{2\lambda_j(1+\theta)}|\mE_{x_t,j}^k|^2 \\
            \ge &(1-\theta)\inner{\tilde s_{x_t,j}^{k} - \tilde s_{x_t,j}^{k-1}}{\tilde s_{x_t,j}^{k-1} - v_{x_t,j}} + \frac{\eta}{2\lambda_j}|\mE_{x_t,j}^k|^2, \notag
        \end{aligned}
    \end{equation}
    where the last inequality comes from $(1+\frac{\eta\lambda_j}{2})(1-\theta) \ge (1-\frac{\eta l}{2})(1-\theta) = (1-\frac{1}{8})(1-\theta) > 0$. Hence, it implies that
    \begin{subequations}
        \begin{align}
            &\inner{\tilde s_{x_t,j}^{k+1} - \tilde s_{x_t,j}^{k}}{\tilde s_{x_t,j}^{k} - v_{x_t,j}} \notag \\
            \ge &(1-\theta)\inner{\tilde s_{x_t,j}^{k} - \tilde s_{x_t,j}^{k-1}}{\tilde s_{x_t,j}^{k-1} - v_{x_t,j}} + \frac{\eta}{2\lambda_j}|\mE_{x_t,j}^k|^2 \notag \\
            \ge &(1-\theta)^k\inner{\tilde s_{x_t,j}^{1} - \tilde s_{x_t,j}^{0}}{\tilde s_{x_t,j}^{0} - v_{x_t,j}} + \frac{\eta}{2\lambda_j}\sum_{t=1}^k(1-\theta)^{k-t}|\mE_{x_t,j}^t|^2 \notag \\
            =   &(1-\theta)^k\eta\inner{- \nabla h_j(\tilde y_{x_t,j}^{0}) -  \mE_{x_t,j}^0}{\tilde y_{x_t,j}^{0} - v_{x_t,j}} + \frac{\eta}{2\lambda_j}\sum_{t=1}^k(1-\theta)^{k-t}|\mE_{x_t,j}^t|^2 \label{subeq.22a}\\
            =   &(1-\theta)^k\eta\inner{\lambda_j (v_{x_t,j} - \tilde y_{x_t,j}^{0}) - \mE_{x_t,j}^0}{\tilde y_{x_t,j}^{0} - v_{x_t,j}} + \frac{\eta}{2\lambda_j}\sum_{t=1}^k(1-\theta)^{k-t}|\mE_{x_t,j}^t|^2 \notag \\
            =   &-(1-\theta)^k\eta \lambda_j|v_{x_t,j} - \tilde y_{x_t,j}^{0}|^2 + (1-\theta)^k\eta \inner{\mE_{x_t,j}^0}{v_{x_t,j} - \tilde y_{x_t,j}^{0}} + \frac{\eta}{2\lambda_j}\sum_{t=1}^k(1-\theta)^{k-t}|\mE_{x_t,j}^t|^2 \notag \\
            \ge &(1-\theta)^k\underbrace{\eta \inner{\mE_{x_t,j}^0}{v_{x_t,j} - \tilde y_{x_t,j}^{0}}}_{\spadesuit} + \frac{\eta}{2\lambda_j}\sum_{t=1}^k(1-\theta)^{k-t}|\mE_{x_t,j}^t|^2, \label{subeq.22b}
        \end{align}
    \end{subequations}
    where \eqref{subeq.22a} follows from the update in \eqref{eq.agd for h}, and \eqref{subeq.22b} is implied by $\lambda_j < 0$. Recall $v_{x_t,j} - \tilde y_{x_t,j}^{0} = -\frac{1}{\lambda_j} \tilde \nabla_j \hat f_{x_t}(s_{x_t}^0)$, and thus the term $\spadesuit$ can be lower bounded as follows:
    \begin{subequations}
        \begin{align}
            &\eta \inner{\mE_{x_t,j}^0}{v_{x_t,j} - \tilde y_{x_t,j}^{0}} \notag \\
            =   &-\frac{\eta}{\lambda_j}\inner{\mE_{x_t,j}^0}{\tilde \nabla_j \hat f_{x_t}(s_{x_t}^0)} \notag\\
            =   &\frac{\eta}{\lambda_j}\inner{\mE_{x_t,j}^0}{\mE_{x_t,j}^0} - \frac{\eta}{\lambda_j}\inner{\mE_{x_t,j}^0}{\tilde g_{x_t,j}(y_{x_t}^0;\mu)} \label{subeq.23a}\\
            =   &\frac{\eta}{\lambda_j}\inner{\mE_{x_t,j}^0}{\mE_{x_t,j}^0} + \frac{1}{\lambda_j}\inner{\mE_{x_t,j}^0}{\tilde s_{x_t,j}^{1} -  \tilde y_{x_t,j}^{0}} \label{subeq.23b}\\
            \ge &\frac{\eta}{\lambda_j}|\mE_{x_t,j}^0|^2 + \frac{1}{2\lambda_jB^{\frac{3}{2}}}|\mE_{x_t,j}^0|^{2} + \frac{B^{\frac{3}{2}}}{2\lambda_j}|\tilde s_{x_t,j}^{1} - \tilde y_{x_t,j}^{0}|^{2}, \label{subeq.23c} 
        \end{align}
    \end{subequations}
    where \eqref{subeq.23a} is due to $\mE_{x_t, j}^0 = \tilde g_{x_t, j}(y_{x_t}^0; \mu) - \nabla h_j (\tilde y_{x_t, j}^0)$ and $ \nabla h_j (\tilde y_{x_t, j}^0) =  \tilde \nabla_j \hat f_{x_t}(s_{x_t}^0)$, \eqref{subeq.23b} follows from $\tilde s_{x_t,j}^{1} =  \tilde y_{x_t,j}^{0} - \tilde g_{x_t,j}(y_{x_t}^0;\mu)$, and \eqref{subeq.23c} holds because $\lambda_j < 0$ and Young's inequality. Putting all the above inequalities together gives
    \begin{align*}
        &h_j(\tilde s_{x_t,j}^{k+1}) - h_j(\tilde s_{x_t,j}^{k}) \\
        \le &-\frac{\theta}{2\eta}|\tilde s_{x_t,j}^{k+1} - \tilde s_{x_t,j}^{k}|^2 + (1-\theta)^k\left(\eta|\mE_{x_t,j}^0|^2 + \frac{1}{2B^{\frac{3}{2}}}|\mE_{x_t,j}^0|^{2} + \frac{B^{\frac{3}{2}}}{2}|\tilde s_{x_t,j}^{1} - \tilde y_{x_t,j}^{0}|^{2}\right) + \frac{\eta}{2}\sum_{t=1}^k(1-\theta)^{k-t}|\mE_{x_t,j}^t|^2 \\
        \le &-\frac{\theta}{2\eta}|\tilde s_{x_t,j}^{k+1} - \tilde s_{x_t,j}^{k}|^2 +\eta|\mE_{x_t,j}^0|^2 + \frac{1}{2B^{\frac{3}{2}}}|\mE_{x_t,j}^0|^{2} + \frac{B^{\frac{3}{2}}}{2}|\tilde s_{x_t,j}^{1} - \tilde y_{x_t,j}^{0}|^{2} + \frac{\eta}{2}\sum_{t=1}^k(1-\theta)^{k-t}|\mE_{x_t,j}^t|^2.
    \end{align*}
    Summing over $j \in \mathcal{S}_2$ implies that
    \begin{subequations}
        \begin{align}
            &\sum_{j \in \mathcal{S}_2} h_j(\tilde s_{x_t,j}^{k+1}) - h_j(\tilde s_{x_t,j}^{k}) \notag \\
        \le &-\frac{\theta}{2\eta}\sum_{j \in \mathcal{S}_2}|\tilde s_{x_t,j}^{k+1} - \tilde s_{x_t,j}^{k}|^2 +\eta\sum_{j \in \mathcal{S}_2}|\mE_{x_t,j}^0|^2 + \frac{1}{2B^{\frac{3}{2}}}\sum_{j \in \mathcal{S}_2}|\mE_{x_t,j}^0|^{2} \notag \\
        &+ \frac{B^{\frac{3}{2}}}{2}\sum_{j \in \mathcal{S}_2}|\tilde s_{x_t,j}^{1} - \tilde y_{x_t,j}^{0}|^{2} + \frac{\eta}{2}\sum_{t=1}^k(1-\theta)^{k-t}\sum_{j \in \mathcal{S}_2}|\mE_{x_t,j}^t|^2 \notag \\
        \le &-\frac{\theta}{2\eta}\sum_{j \in \mathcal{S}_2}|\tilde s_{x_t,j}^{k+1} - \tilde s_{x_t,j}^{k}|^2 + \eta \mathbf{E}(\mu)^2 + \frac{\mathbf{E}(\mu)^{2}}{2B^{\frac{3}{2}}}\notag \\
        &+ \frac{B^{\frac{7}{2}}}{2} + \frac{\eta}{2}\sum_{t=1}^k(1-\theta)^{k-t}\left(2\mathbf{E}(\mu)^2 + 8\rho^2B^4\right), \label{subeq.27aa} \\
        \le&-\frac{\theta}{2\eta}\sum_{j \in \mathcal{S}_2}|\tilde s_{x_t,j}^{k+1} - \tilde s_{x_t,j}^{k}|^2 + \eta \mathbf{E}(\mu)^2 + \frac{\mathbf{E}(\mu)^{2}}{2B^{\frac{3}{2}}} + \frac{B^{\frac{7}{2}}}{2}  + \frac{\eta \mathbf{E}(\mu)^2}{\theta} + \frac{4\eta \rho^2 B^4}{\theta}, \notag
        \end{align}
    \end{subequations}
    where we apply $\sum_{j \in \mathcal{S}_2}|\tilde s_{x_t,j}^{1} - \tilde y_{x_t,j}^{0}|^{2} \le \|s_{x_t}^1 - y_{x_t}^0\|^2 \le B^2$ and the result of Lemma \ref{lemma.upper bound of mE} in \eqref{subeq.27aa}.
    Finally, we conclude
    \begin{subequations}
        \begin{align}
            &\sum_{j \in \mathcal{S}_2} h_j(\tilde s_{x_t,j}^\mK) \notag \\
            =   &\sum_{j \in \mathcal{S}_2} \sum_{k=0}^{\mKb} h_j(\tilde s_{x_t,j}^{k+1}) - h_j(\tilde s_{x_t,j}^{k}) + h_j(\tilde s_{x_t,j}^{0}) \notag \\
            \le &-\frac{\theta}{2\eta} \sum_{j \in \mathcal{S}_2} \sum_{k=0}^{\mKb}|\tilde s_{x_t,j}^{k+1} - \tilde s_{x_t,j}^{k}|^2 + \eta\mK \mathbf{E}(\mu)^2 + \frac{\mK}{2B^{\frac{3}{2}}}\mathbf{E}(\mu)^{2} + \frac{\mK B^{\frac{7}{2}}}{2} + \frac{\eta\mK \mathbf{E}(\mu)^2}{\theta} + \frac{4\eta \mK\rho^2 B^4}{\theta} \notag.
        \end{align}
    \end{subequations}
\end{proof}
\begin{corollary}\label{coro.decrease tss}
     Suppose that Assumption \ref{assum.lipschitz gradient} and \ref{assum.lipschitz hessian} hold. Under the parameter setting \eqref{eq.parameter setting}, choose a reasonably small $\mu$ such that $\mathbf{E}(\mu) \le \frac{lB}{2}$ in Algorithm \ref{alg.AZO}. Then, for the tangent space step at iterate $x_t$, when the ``if condition" triggers, we have:
    \begin{equation*}
         \hat f_{x_t}(s_{x_t}^\mK) \le \hat f_{x_t}(s_{x_t}^0) -\frac{3\theta B^2}{8\eta K} + \eta K \mathbf{E}(\mu)^2 + \frac{K}{2B^{\frac{3}{2}}}\mathbf{E}(\mu)^{2} + \frac{K B^{\frac{7}{2}}}{2} + \frac{5\eta K \mathbf{E}(\mu)^2}{\theta} + \frac{20\eta K\rho^2 B^4}{\theta} + \frac{32\rho B^3}{3}.
    \end{equation*}
\end{corollary}
\begin{proof}
    Putting Lemma \ref{lemma.tss s1} and Lemma \ref{lemma.tss s2} together, we obtain
    \begin{align*}
         &\sum_{j=1}^d h_j(\tilde s_{x_t,j}^\mK) \\
         \le &-\frac{3\theta}{8\eta}\sum_{k=0}^{\mKb}\|\tilde s_{x_t}^{k+1} - \tilde s_{x_t}^{k}\|^2 + \eta\mK \mathbf{E}(\mu)^2 + \frac{\mK}{2B^{\frac{3}{2}}}\mathbf{E}(\mu)^{2} + \frac{\mK B^{\frac{7}{2}}}{2} + \frac{5\eta\mK \mathbf{E}(\mu)^2}{\theta} + \frac{20\eta \mK\rho^2 B^4}{\theta} \\
         \le &-\frac{3\theta B^2}{8\eta\mK} + \eta\mK \mathbf{E}(\mu)^2 + \frac{\mK}{2B^{\frac{3}{2}}}\mathbf{E}(\mu)^{2} + \frac{\mK B^{\frac{7}{2}}}{2} + \frac{5\eta\mK \mathbf{E}(\mu)^2}{\theta} + \frac{20\eta \mK\rho^2 B^4}{\theta} \\
         \le &-\frac{3\theta B^2}{8\eta K} + \eta K \mathbf{E}(\mu)^2 + \frac{K}{2B^{\frac{3}{2}}}\mathbf{E}(\mu)^{2} + \frac{K B^{\frac{7}{2}}}{2} + \frac{5\eta K \mathbf{E}(\mu)^2}{\theta} + \frac{20\eta K\rho^2 B^4}{\theta},
    \end{align*}
    combining with Lemma \ref{lemma.tss function value with h} completes the proof.
\end{proof}

\subsection{Tangent space step: small Riemannian gradient}
For the scenario that the ``if condition" does not trigger in the tangent space step, we establish that the tangent space step outputs a satisfactory point with a small Riemannian gradient.
\begin{lemma}\label{lemma.small gradient}
     Suppose that Assumption \ref{assum.lipschitz gradient}, \ref{assum.lipschitz hessian} and \ref{assum.retraction} hold. Under the parameter setting \eqref{eq.parameter setting}, choose a reasonably small $\mu$ such that $\mathbf{E}(\mu) \le \frac{lB}{2}$ in Algorithm \ref{alg.AZO}.Then, for the tangent space step at iterate $x_t$, when the ``if condition" does not trigger, we have:
    \begin{equation*}
        \|\grad f(x_{t+1})\| \le \frac{1}{\sigma_{\min}}\cdot (2\rho B^2 + \frac{2\sqrt{2}B}{K^2\eta} + \frac{2\theta B}{K\eta} + \left(2\mathbf{E}(\mu)^2 + 8\rho^2B^4\right)^{\frac{1}{2}}).
    \end{equation*}
\end{lemma}
\begin{proof}
    According to Subroutine \ref{sub.TSS}, when the "if condition" does not trigger, the tangent space step outputs $x_{t+1} = \retr_{x_t}(y_{x_t}^*)$, and thus, it holds that:
    \begin{equation*}
        \grad f(x_{t+1}) =  \grad f(\retr_{x_t}(y_{x_t}^*)) = (T_{x_t, y_{x_t}^*}^*)^{-1} \nabla \hat f_{x_t}(y_{x_t}^*).
    \end{equation*}
    Therefore, it is sufficient to upper bound the term $\|\nabla \hat f_{x_t}(y_{x_t}^*)\|$. By applying Minkowski's inequality, we have 
    \begin{equation*}
        \|\nabla \hat f_{x_t}(y_{x_t}^*)\| = \left(\sum_{j=1}^d |\tilde \nabla_j \hat f_{x_t}(y_{x_t}^*)|^2\right)^{\frac{1}{2}} \le \left(\sum_{j=1}^d |\tilde \nabla_j \hat f_{x_t}(y_{x_t}^*) - \nabla h_j(\tilde y_{x_t,j}^*)|^2\right)^{\frac{1}{2}} + \left(\sum_{j=1}^d |\nabla h_j(\tilde y_{x_t,j}^*)|^2\right)^{\frac{1}{2}}.
    \end{equation*}
    Note that $\nabla h_j(\tilde y_{x_t,j}^*) = \tilde \nabla_j \hat f_{x_t}(s_{x_t}^0) + \lambda_j(\tilde y_{x_t,j}^* - \tilde s_{x_t,j})$, using the same argument in the proof of Lemma \ref{lemma.upper bound of mE} gives
    \begin{subequations}
        \begin{align}
        &\left(\sum_{j=1}^d |\tilde \nabla_j \hat f_{x_t}(y_{x_t}^*) - \nabla h_j(\tilde y_{x_t,j}^*)|^2\right)^{\frac{1}{2}} \notag\\
        = &\left(\sum_{j=1}^d |\tilde \nabla_j \hat f_{x_t}(y_{x_t}^*) - \tilde \nabla_j \hat f_{x_t}(s_{x_t}^0) - \lambda_j(\tilde y_{x_t,j}^* - \tilde s_{x_t,j})|^2\right)^{\frac{1}{2}} \notag \\
        = &\|\nabla \hat f_{x_t}(y_{x_t}^*) - \nabla \hat f_{x_t}(s_{x_t}^0) - \nabla^2 \hat f_{x_t}(s_{x_t}^0)(y_{x_t}^* - s_{x_t}^0)\| \notag \\
        \le &\frac{\rho}{2}\|y_{x_t}^* - s_{x_t}^0\|^2 \label{subeq.27b} \\
        \le &2\rho B^2, \label{subeq.distance}
        \end{align}
    \end{subequations}
    where \eqref{subeq.27b} is due to the Lipschitz continuity of $\nabla^2 \hat f_{x_t}(\cdot)$ and the \eqref{subeq.distance} comes from the following result
    \begin{equation*}
        \|y_{x_t}^* - s_{x_t}^0\| \le \frac{1}{K_0 + 1}\sum_{k=0}^{K_0} \|y_{x_t}^k - s_{x_t}^0\| \le 2B.
    \end{equation*}
    For the term $\left(\sum_{j=1}^d |\nabla h_j(\tilde y_{x_t,j}^*)|^2\right)^{\frac{1}{2}}$, note that $\tilde y_{x_t,j}^* = \frac{1}{K_0 + 1}\sum_{k=0}^{K_0} \tilde y_{x_t,j}^k$ and $\nabla h_j(\cdot)$ is a one-dimensional linear function, it equivalents as
    \begin{equation*}
       \left(\sum_{j=1}^d |\nabla h_j(\tilde y_{x_t,j}^*)|^2\right)^{\frac{1}{2}} =  \left(\sum_{j=1}^d \left|\nabla h_j\left(\frac{1}{K_0 + 1}\sum_{k=0}^{K_0} \tilde y_{x_t,j}^k\right)\right|^2\right)^{\frac{1}{2}} = \frac{1}{K_0 + 1}\left(\sum_{j=1}^d \left|\sum_{k=0}^{K_0}\nabla h_j(\tilde y_{x_t,j}^k)\right|^2\right)^{\frac{1}{2}}.
    \end{equation*}
    Recall the update \eqref{eq.agd for h}, we have
    \begin{align*}
        \eta\nabla h_j(\tilde y_{x_t,j}^k)  &= \tilde s_{x_t,j}^{k+1} - \tilde y_{x_t,j}^k + \eta \mE_{x_t,j}^k \\
        &= \tilde s_{x_t,j}^{k+1} - \tilde s_{x_t,j}^k  - (1-\theta)(\tilde s_{x_t,j}^{k} - \tilde s_{x_t,j}^{k-1})+ \eta \mE_{x_t,j}^k.
    \end{align*}
    Consequently, it holds that
    \begin{subequations}
        \begin{align}
           &\left(\sum_{j=1}^d |\nabla h_j(\tilde y_{x_t,j}^*)|^2\right)^{\frac{1}{2}} \notag \\
           =   &\frac{1}{(K_0 + 1)\eta}\left(\sum_{j=1}^d \left|s_{x_t,j}^{K_0} - \tilde s_{x_t,j}^{0} + \theta(\tilde s_{x_t,j}^{K_0} - \tilde s_{x_t,j}^{0}) + \eta\sum_{k=0}^{K_0} \mE_{x_t,j}^k\right|^2\right)^{\frac{1}{2}} \notag \\
           \le &\frac{1}{(K_0 + 1)\eta}\left(\sum_{j=1}^d \left|\tilde s_{x_t,j}^{K_0+1} - \tilde s_{x_t,j}^{K_0}\right|^2\right)^{\frac{1}{2}} + \frac{\theta}{(K_0 + 1)\eta}\left(\sum_{j=1}^d \left|\tilde s_{x_t,j}^{K_0} - \tilde s_{x_t,j}^{0}\right|^2\right)^{\frac{1}{2}} + \frac{1}{K_0 + 1} \left(\sum_{j=1}^d\left|\sum_{k=0}^{K_0}\mE_{x_t,j}^k\right|^2\right)^{\frac{1}{2}} \notag \\
           \le &\frac{1}{(K_0 + 1)\eta}\left\|s_{x_t}^{K_0+1} - s_{x_t}^{K_0}\right\| + \frac{\theta}{(K_0 + 1)\eta}\left\|s_{x_t}^{K_0} - s_{x_t}^{0}\right\| + \frac{1}{\sqrt{K_0+1}}\left(\sum_{k=0}^{K_0}\sum_{j=1}^d|\mE_{x_t,j}^k|^2 \right)^{\frac{1}{2}} \label{subeq.28b} \\
           \le &\frac{2}{K\eta}\left\|s_{x_t}^{K_0+1} - s_{x_t}^{K_0}\right\| + \frac{2\theta B}{K\eta} + \left(2\mathbf{E}(\mu)^2 + 8\rho^2B^4\right)^{\frac{1}{2}} \label{subeq.28c}
        \end{align}
    \end{subequations}
    where \eqref{subeq.28b} holds because $\left|\sum_{k=0}^{K_0}\mE_{x_t,j}^k\right|^2 \le (K_0 + 1)\sum_{k=0}^{K_0}|\mE_{x_t,j}^k|^2$, and \eqref{subeq.28c} follows from the $K_0 + 1 \ge \frac{K}{2}$, $\theta \le 1$, Lemma \ref{fact} and \ref{lemma.upper bound of mE}. Recall the definition of $K_0$, we obtain
    \begin{align*}
            &\left\|s_{x_t}^{K_0+1} - s_{x_t}^{K_0}\right\|^2 \\
        \le &\frac{1}{K-\lfloor K/2 \rfloor} \sum_{k=\lfloor K/2 \rfloor}^{K-1}\|s_{x_t}^{k+1} - s_{x_t}^{k}\|^2 \\
        \le &\frac{1}{K-\lfloor K/2 \rfloor} \sum_{k=0}^{K-1}\|s_{x_t}^{k+1} - s_{x_t}^{k}\|^2 \\
        \le &\frac{2B^2}{K^2}.
    \end{align*}
    Therefore, combining all the above inequalities gives that
    \begin{align*}
        &\|\grad f(x_{t+1})\| \\
        \le &\|(T_{x_t, y_{x_t}^*}^*)^{-1}\| \cdot (2\rho B^2 + \frac{2\sqrt{2}B}{K^2\eta} + \frac{2 B}{K\eta} + \left(2\mathbf{E}(\mu)^2 + 8\rho^2B^4\right)^{\frac{1}{2}}) \\
        \le &\frac{1}{\sigma_{\min}} \cdot (2\rho B^2 + \frac{2\sqrt{2}B}{K^2\eta} + \frac{2\theta B}{K\eta} + \left(2\mathbf{E}(\mu)^2 + 8\rho^2B^4\right)^{\frac{1}{2}}),
    \end{align*}
    where the last inequality comes from Assumption \ref{assum.retraction}.
\end{proof}

\subsection{Proof of Theorem \ref{thm.fosp}}
\begin{theorem}[Theorem \ref{thm.fosp} restated]
    Suppose that Assumption \ref{assum.lipschitz gradient}, \ref{assum.lipschitz hessian} and \ref{assum.retraction} hold. Set the parameters in Algorithm \ref{alg.AZO} as follows
    \begin{equation}\label{eq.parameter setting fosp}
       \eta = \frac{1}{4l}, ~B=\frac{1}{8}\sqrt{\frac{\epsilon}{\rho}}, ~\theta = \frac{\rho^{\frac{7}{4}}\epsilon^{\frac{1}{4}}}{l}, ~r=0, ~K = \frac{\rho^{\frac{5}{4}}}{4\epsilon^{\frac{1}{4}}}.
    \end{equation}
     For any $x_0 \in \M$ and sufficiently small $\epsilon > 0$, 
     choose $\mu = \mathcal{O}\left(\frac{\epsilon^{1/4}}{d^{1/4}}\right)$ in Lines 3 of Algorithm \ref{alg.AZO}, and $\mu = \mathcal{O}\left(\frac{\epsilon^{5/8}}{d^{1/4}}\right)$ in Line 5 of Subroutine \ref{sub.TSS}. Then Algorithm \ref{alg.AZO} with Option I outputs an $\epsilon$-approximate first-order stationary point. The total number of function value evaluations is no more than
    \begin{equation}
        \mathcal{O}\left(\frac{(f(x_0) - f_{\operatorname{low}})d}{\epsilon^{\frac{7}{4}}}\right). \notag
    \end{equation}
\end{theorem}
\begin{proof}
    Given an iterate $x_t$, in the large estimator scenario where $\|g_{x_t}(0;\mu)\| \ge lB$, we choose $\mu = \mathcal{O}\left(\frac{\epsilon^{1/4}}{d^{1/4}}\right)$; combining this with Lemma \ref{lemma.zero order estimator error} results in $\mathbf{E}(\mu) \le \mathcal{O}\left(\sqrt{\epsilon}\right)$. Consequently, Lemma \ref{lemma.function value decrease large gradient} yields:
    \begin{equation}
        f(x_{t+1}) - f(x_t) \le - \min\left\{\frac{lB^2}{16}, lb^2\right\} = -\frac{l\epsilon}{1024\rho}. \notag
    \end{equation}
    For the small estimator scenario where $\|g_{x_t}(0;\mu)\| \le lB$, the Algorithm \ref{alg.AZO} switches to Subroutine \ref{sub.TSS}, i.e. the tangent space step. Note that $r = 0$, it implies
    \begin{equation*}
        f(x_{t}) = \hat f_{x_t}(s_{x_t}^0).
    \end{equation*}
    Moreover, we choose $\mu = \mathcal{O}\left(\frac{\epsilon^{5/8}}{d^{1/4}}\right)$ in Subroutine \ref{sub.TSS}, resulting in $\mathbf{E}(\mu) \le \mathcal{O}\left(\epsilon^{5/4}\right)$. Thus, when the ``if condition" triggers, from Corollary \ref{coro.decrease tss}, we have
    \begin{subequations}
        \begin{align}
        &f(x_{t+1}) - f(x_{t}) \notag \\
        = &\hat f_{x_t}(s_{x_t}^\mK) - \hat f_{x_t}(s_{x_t}^0) \notag \\
        \le &-\frac{3\theta B^2}{8\eta K} + \eta K \mathbf{E}(\mu)^2 + \frac{K}{2B^{\frac{3}{2}}}\mathbf{E}(\mu)^{2} + \frac{K B^{\frac{7}{2}}}{2} + \frac{5\eta K \mathbf{E}(\mu)^2}{\theta} + \frac{20\eta K\rho^2 B^4}{\theta} + \frac{32\rho B^3}{3} \notag \\
        \le &-\frac{\epsilon^{\frac{3}{2}}}{24\sqrt{\rho}} + \eta K \mathbf{E}(\mu)^2 + \frac{K}{2B^{\frac{3}{2}}}\mathbf{E}(\mu)^{2} + \frac{5\eta K \mathbf{E}(\mu)^2}{\theta} \notag \\ 
        \le &-\frac{\epsilon^{\frac{3}{2}}}{24\sqrt{\rho}} + \mathcal{O}\left(\epsilon^{\frac{9}{4}}\right) + \mathcal{O}\left(\epsilon^{\frac{3}{2}}\right) + \mathcal{O}\left(\epsilon^{2}\right) \notag \\
        \le &-\frac{\epsilon^{\frac{3}{2}}}{32\sqrt{\rho}}. \notag
        \end{align}
    \end{subequations}
    When the ``if condition" does not trigger, Lemma \ref{lemma.small gradient} tells 
    \begin{subequations}
        \begin{align}
        & \|\grad f(x_{t+1})\| \notag \\
       \le &\frac{1}{\sigma_{\min}}\cdot (2\rho B^2 + \frac{2\sqrt{2}B}{K^2\eta} + \frac{2\theta B}{K\eta} + \left(2\mathbf{E}(\mu)^2 + 8\rho^2B^4\right)^{\frac{1}{2}}) \notag \\
       \le &\frac{1}{\sigma_{\min}}\cdot (2\rho B^2 + \frac{2\sqrt{2}B}{K^2\eta} + \frac{2\theta B}{K\eta} + 4\rho B^2) \label{subeq.fosp tss small gradient 1}\\
       \le &\mathcal{O}(\epsilon), \label{subeq.fosp tss small gradient 2}
        \end{align}
    \end{subequations}
    where \eqref{subeq.fosp tss small gradient 1} holds because $\mathbf{E}(\mu)^2 \le \mathcal{O}\left(\epsilon^{5/2}\right)$ and $8\rho^2 B^4 = \mathcal{O}\left(\epsilon^2\right)$, and \eqref{subeq.fosp tss small gradient 2} comes from the parameter setting \eqref{eq.parameter setting fosp}. Therefore, at each iteration $t$, once $\norm{g_{x_t}(0;\mu)} \ge lB$ holds or the ``if condition" does not trigger in the tangent space step, we observe the following function value decrease:
    \begin{equation*}
        f(x_{t+1}) - f(x_t) \le -\min \left\{\frac{l\epsilon}{1024\rho}, \frac{\epsilon^{\frac{3}{2}}}{32\sqrt{\rho}}\right\} = -\frac{\epsilon^{\frac{3}{2}}}{32\sqrt{\rho}}.
    \end{equation*}
    Otherwise, if the 'if condition' does not trigger, $x_{t+1}$ is already an $\epsilon$-approximate first-order stationary point. As the tangent space step requires at most $K = \mathcal{O}\left(\epsilon^{-1/4}\right)$ iterations, and each iterate needs $2d$ function value evaluations to construct the zeroth-order estimator, the total number of function value evaluations must be less than
    \begin{equation*}
        O\left(\frac{(f(x_0) - f_{\operatorname{low}})d}{\epsilon^{\frac{7}{4}}}\right).\notag
    \end{equation*}
\end{proof}

\subsection{Proof of Theorem \ref{thm.sosp}}
\begin{theorem}[Theorem \ref{thm.sosp} restated]
    Suppose that Assumption \ref{assum.lipschitz gradient}, \ref{assum.lipschitz hessian} and \ref{assum.retraction} hold. Set the parameters in Algorithm \ref{alg.AZO} as follows
    \begin{equation}\label{eq.parameter setting sosp}
       \eta = \frac{1}{4l}, ~\chi = \mathcal{O}\left(\log \frac{d}{\delta \epsilon} \right) \ge 1, ~B=\frac{1}{8\chi^2}\sqrt{\frac{\epsilon}{\rho}}, ~\theta = \frac{\rho^{\frac{7}{4}}\epsilon^{\frac{1}{4}}}{l}, ~r=\frac{\theta B}{6 K}, ~K = \frac{\chi\rho^{\frac{5}{4}}}{4\epsilon^{\frac{1}{4}}}.
    \end{equation}
     For any $x_0 \in \M$ and sufficiently small $\epsilon > 0$, 
     choose $\mu = \mathcal{O}\left(\frac{\epsilon^{1/4}}{d^{1/4}\chi}\right) = \tilde{\mathcal{O}}\left(\frac{\epsilon^{1/4}}{d^{1/4}}\right)$ in Lines 3 of Algorithm \ref{alg.AZO}, and $\mu = \min\left\{\mathcal{O}\left(\frac{\epsilon^{5/8}}{d^{1/4}\chi^2}\right),\mathcal{O}\left(\frac{\epsilon^{7/8}}{\chi^3\sqrt{d}}\right)\right\} =\tilde{\mathcal{O}}\left(\frac{\epsilon^{7/8}}{\sqrt{d}}\right)$ in Line 5 of Subroutine \ref{sub.TSS}. Then perturbed Algorithm \ref{alg.AZO} with Option I outputs an $\epsilon$-approximate second-order stationary point with a probability of at least $1-\delta$. The total number of function value evaluations is no more than
    \begin{equation}
    O\left(\frac{(f(x_0) - f_{\operatorname{low}})d}{\epsilon^{\frac{7}{4}}}\log^6\left(\frac{d}{\delta\epsilon}\right)\right). \notag
    \end{equation}
\end{theorem}
\begin{proof}
    By a similar argument, for the scenario $\|g_{x_t}(0;\mu)\| \ge lB$ at iterate $x_t$, we have
    \begin{equation}
        f(x_{t+1}) - f(x_t) \le - \min\left\{\frac{lB^2}{16}, lb^2\right\} = -\frac{l\epsilon}{1024\chi^4\rho}. \notag
    \end{equation}
    For the tangent space step at iterate $x_t$, we start with a perturbed point in $\T_{x_t}\M$, that is $s_{x_t}^0 = \xi_t \sim \Uniform(\mathbb{B}_{x_t,r}(0))$, it follows
    \begin{equation*}
       f(x_t) - \hat f_{x_t}(s_{x_t}^0) = \hat f_{x_t}(0) - \hat f_{x_t}(\xi_t) \le \inner{\nabla \hat f_{x_t}(\xi_t)}{-\xi_t} + \frac{l}{2}\|\xi_t\|^2 \le r \cdot \|\nabla \hat f_{x_t}(\xi_t)\| + \frac{lr^2}{2}.
    \end{equation*}
     Recall we choose $\mu = \mathcal{O}\left(\frac{\epsilon^{1/4}}{d^{1/4}\chi}\right)$ in Line 3 of Algorithm \ref{alg.AZO}, and thus $\mathbf{E}(\mu) \le \frac{lB}{2}$ holds. Consequently, the term $\|\nabla \hat f_{x_t}(\xi_t)\|$ can be upper bounded as
     \begin{equation*}
         \|\nabla \hat f_{x_t}(\xi_t)\| \le \|\nabla \hat f_{x_t}(\xi_t) - \nabla \hat f_{x_t}(0)\| + \|\nabla \hat f_{x_t}(0) - g_{x_t}(0;\mu)\| + \|g_{x_t}(0;\mu)\| \le l \cdot \|\xi_t\| + \mathbf{E}(\mu) + lB \le lr +\frac{3l B}{2}.
     \end{equation*}
    Substituting $r = \frac{\theta B}{6 K}  = \tilde{\mathcal{O}}\left(\epsilon\right)$ gives
    \begin{equation*}
        f(x_t) - \hat f_{x_t}(s_{x_t}^0) \le \frac{3lr^2}{2} + \frac{3lrB}{2} \le \tilde{\mathcal{O}}\left(\epsilon^2\right) + \frac{\theta B^2}{16\eta K} \le \frac{\theta B^2}{8\eta K},
    \end{equation*}
    combining with the choice of $\mu \le \mathcal{O}\left(\frac{\epsilon^{5/8}}{d^{1/4}\chi^2}\right)$ in Line 5 of Subroutine \ref{sub.TSS} leads to
    \begin{align*}
        &f(x_{t+1}) - f(x_t) \\
        = &\hat f_{x_t}(s_{x_t}^\mK) - \hat f_{x_t}(s_{x_t}^0) + \hat f_{x_t}(\xi_t) - f(x_t) \\
        \le &-\frac{\theta B^2}{4\eta K} + \eta K \mathbf{E}(\mu)^2 + \frac{K}{2B^{\frac{3}{2}}}\mathbf{E}(\mu)^{2} + \frac{K B^{\frac{7}{2}}}{2} + \frac{5\eta K \mathbf{E}(\mu)^2}{\theta} + \frac{20\eta K\rho^2 B^4}{\theta} + \frac{32\rho B^3}{3} \notag \\
        \le &-\frac{\epsilon^{\frac{3}{2}}}{96\chi^5\sqrt{\rho}} + \eta K \mathbf{E}(\mu)^2 + \frac{K}{2B^{\frac{3}{2}}}\mathbf{E}(\mu)^{2} + \frac{5\eta K \mathbf{E}(\mu)^2}{\theta} \notag \\ 
        \le &-\frac{\epsilon^{\frac{3}{2}}}{96\chi^5\sqrt{\rho}} + \mathcal{O}\left(\frac{\epsilon^{\frac{9}{4}}}{\chi^7}\right) + \mathcal{O}\left(\frac{\epsilon^{\frac{3}{2}}}{\chi^9}\right) + \mathcal{O}\left(\frac{\epsilon^{2}}{\chi^7}\right) \notag \\
        \le &-\frac{\epsilon^{\frac{3}{2}}}{192\chi^5\sqrt{\rho}}. \notag
    \end{align*}
    when the ``if condition" triggers. Therefore, in the case of the function value decrease, we have
    \begin{equation}\label{eq.sosp function value decrease}
        f(x_{t+1}) - f(x_t) \le -\min\left\{\frac{l\epsilon}{1024\chi^4\rho}, \frac{\epsilon^{\frac{3}{2}}}{192\chi^5\sqrt{\rho}}\right\} = -\frac{\epsilon^{\frac{3}{2}}}{192\chi^5\sqrt{\rho}}.
    \end{equation}
    Similarly, since the tangent space step requires at most $K = \mathcal{O}(\frac{\chi}{\epsilon^{1/4}})$ iterations and each iterate needs $2d$ function value evaluations to construct the
    zeroth-order estimator, the total number of function value evaluations does not exceed
    \begin{equation*}
        O\left(\frac{(f(x_0) - f_{\operatorname{low}})d}{\epsilon^{\frac{7}{4}}}\log^6\left(\frac{d}{\delta\epsilon}\right)\right).
    \end{equation*}
    For the scenario that the ``if condition" does not trigger in the tangent space step at iterate $x_t$, from the proof of Theorem \ref{thm.fosp}, it holds that
    \begin{equation*}
        \|\grad f(x_{t+1})\| \le \mathcal{O}(\epsilon).
    \end{equation*}
    To achieve the $\epsilon$-approximate second-order stationary points, it remains to analyze the value of $\lambda_{\min}(\Hess f(x_{t+1}))$. Suppose $\lambda_{\min}(\nabla^2 \hat f_{x_t}(0)) \ge -\sqrt{\rho\epsilon}$, then it holds that
    \begin{subequations}
        \begin{align}
            &\lambda_{\min}(\nabla^2 \hat f_{x_t}(y_{x_t}^*)) \notag \\
            \ge &\lambda_{\min}(\nabla^2 \hat f_{x_t}(0)) - |\lambda_{\min}(\nabla^2 \hat f_{x_t}(y_{x_t}^*)) - \lambda_{\min}(\nabla^2 \hat f_{x_t}(0))| \notag \\
            \ge &\lambda_{\min}(\nabla^2 \hat f_{x_t}(0)) - \|\nabla^2 \hat f_{x_t}(y_{x_t}^*) - \nabla^2 \hat f_{x_t}(0)\| \notag \\
            \ge &\lambda_{\min}(\nabla^2 \hat f_{x_t}(0)) - \rho\|y_{x_t}^* - s_{x_t}^0\| - \rho\|s_{x_t}^0\| \notag\\
            \ge &\lambda_{\min}(\nabla^2 \hat f_{x_t}(0)) - 2\rho B- \rho r \label{subeq.33a} \\
            \ge &-2\sqrt{\rho\epsilon}, \notag
        \end{align}
    \end{subequations}
    where \eqref{subeq.33a} follows from $\|y_{x_t}^* - s_{x_t}^0\| \le 2B$ and $\|s_{x_t}^0\| = \|\xi_t\| \le r$. From Lemma \ref{lemma.pullback gradient hessian}, we have
    \begin{equation*}
        \nabla^2 \hat f_{x_t}(y_{x_t}^*) = T_{x_t, y_{x_t}^*}^* \Hess f(x_{t+1}) T_{x_t, y_{x_t}^*} + W_{y_{x_t}^*},
    \end{equation*}
    and it implies that
    \begin{subequations}
        \begin{align}
            &\lambda_{\min}(\Hess f(x_{t+1})) \notag \\
            \ge &\frac{\lambda_{\min}\left(T_{x_t, y_{x_t}^*}^* \Hess f(x_{t+1}) T_{x_t, y_{x_t}^*}\right)}{\lambda_{\max}(T_{x_t, y_{x_t}^*}^* T_{x_t, y_{x_t}^*})} \label{subeq.34a} \\
            \ge &\frac{\lambda_{\min}\left(\nabla^2 \hat f_{x_t}(y_{x_t}^*) -  W_{y_{x_t}^*}\right)}{\sigma_{\max}^2} \label{subeq.34b} \\
            \ge &\frac{\lambda_{\min}\left(\nabla^2 \hat f_{x_t}(y_{x_t}^*)\right) + \lambda_{\min}\left(-W_{y_{x_t}^*}\right)}{\sigma_{\max}^2} \label{subeq.wely} \\
            \ge &-\frac{2\sqrt{\rho\epsilon}}{\sigma_{\max}^2} - \frac{25\tau}{\sigma_{\max}^2\sigma_{\min}}\epsilon \label{subeq.34c} \\
            \ge &-\frac{4\sqrt{\rho\epsilon}}{\sigma_{\max}^2} \notag
        \end{align}
    \end{subequations}
    where \eqref{subeq.34a} is due to the Ostrowski’s Theorem \cite{ostrowski1961some}, \eqref{subeq.34b} comes from Assumption \ref{assum.retraction}, \eqref{subeq.wely} uses Wely's inequality \cite{horn2012matrix}, and \eqref{subeq.34c} comes from the following inequality
    \begin{equation*}
        \|W_{y_{x_t}^*}\| = \max_{\dot{s}_{x_t} \in \T_{x_t}\M, \|\dot{s}_{x_t}\| = 1} \inner{W_{y_{x_t}^*}\dot{s}_{x_t}}{\dot{s}_{x_t}} \le \|\gamma_{x_t, \dot{s}_{x_t}}^{\prime \prime}(0)\|\|\grad f(\retr_{x_t}(y_{x_t}^*))\| \le \frac{25\tau}{\sigma_{\min}}\epsilon.
    \end{equation*}
    Consider the case $\lambda_{\min}(\nabla^2 \hat f_{x_t}(0)) < -\sqrt{\rho\epsilon}$, we define the following stuck region in the tangent space step at iterate $x_t$: 
    \begin{equation*}
        \mathcal{X}_t^{\operatorname{stuck}} = \left\{\begin{array}{l}
            \left\{s_{x_t} \in \mathbb{B}_{x_t, r}(0): \{s_{x_t}^k\}_{k=1}^K \text { satisfies } s_{x_t}^0 = s_{x_t} \text{ and } K\sum_{k=0}^{K-1} \|s_{x_t}^{k+1} - s_{x_t}^k\|\leq B^2 \right\},
            \text { if } \lambda_{\min}(\nabla^2 \hat f_{x_t}(0)) < -\sqrt{\rho\epsilon}, \\
            \emptyset, \text { otherwise. }
            \end{array}\right.
    \end{equation*}
    From Lemma \ref{lemma.pertubation escape}, we know that the probability of $s_{x_t}^0 = \xi_t \in \mathcal{X}_t^{\operatorname{stuck}}$ satisfies $\operatorname{Pr}\left\{\xi_t \in \mathcal{X}_t^{\operatorname{stuck}}\right\} \le \delta$. Therefore, once the 'if condition' does not trigger in the tangent space step, with a probability of at least $1-\delta$, $x_{t+1}$ is an $\epsilon$-approximate second-order stationary point. 
\end{proof}

\begin{lemma}\label{lemma.pertubation escape}
Suppose that Assumption \ref{assum.lipschitz gradient}, \ref{assum.lipschitz hessian} and \ref{assum.retraction} hold. Under the parameter settings in Theorem \ref{thm.sosp}, let $r_0 = \frac{\delta r}{\sqrt{d}}$. In cases where $\lambda_{\min}(\nabla^2 \hat f_{x_t}(0)) < -\sqrt{\rho\epsilon}$, given $s_{x_t}^{\prime 0}, s_{x_t}^{\prime \prime 0} \in \mathbb{B}{x_t,r}(0)$ with $s{x_t}^{\prime 0}-s_{x_t}^{\prime \prime 0}=r_0 v_1$, where $v_1$ is the minimum eigen-direction of $\nabla^2 \hat f_{x_t}(0)$, choose $\mu = \mathcal{O}\left(\frac{\epsilon^{7/8}}{\chi^3\sqrt{d}}\right) = \tilde{\mathcal{O}}\left(\frac{\epsilon^{7/8}}{\sqrt{d}}\right)$ such that $\mathbf{E}(\mu) \le \frac{\rho B \theta r_0}{2}$. By running the tangent space step starting at $s_{x_t}^{\prime 0}$ and $s_{x_t}^{\prime \prime 0}$ respectively we have
$$
\max \left\{K \sum_{k=0}^{K-1}\left\|s_{x_t}^{\prime k+1}-s_{x_t}^{\prime k}\right\|^2, K \sum_{k=0}^{K-1}\left\|s_{x_t}^{\prime \prime k+1}-s_{x_t}^{\prime \prime k}\right\|^2\right\}>B^2 .
$$
that is, at least one of the iterates triggers the "if condition".
\end{lemma}
The proof of this lemma follows from Lemma 18 in \cite{jin2018accelerated} and Lemma B.2 in \cite{li2022restarted}, and thus we list the sketch. The details can be found in \cite{li2022restarted} and \cite{jin2018accelerated}
\begin{proof}
    For any point $s_{x_t} \in \T_{x_t}\M$, we introduce the notation $e_{x_t}(s_{x_t};\mu): = \nabla \hat f_{x_t}(s_{x_t}) - g_{x_t}(s_{x_t};\mu)$, and thus, the update in tangent space step at iterate $x_t$ can be rewritten as
    \begin{align*}
        &y_{x_t}^k = s_{x_t}^k + (1-\theta)(s_{x_t}^k - s_{x_t}^{k-1}) \\
        &s_{x_t}^{k+1} = y_{x_t}^k - \eta \nabla \hat f_{x_t}(y_{x_t}^k) + \eta e_{x_t}(y_{x_t}^k;\mu).
    \end{align*}
    Denoting $\mathbf{w}_{x_t}^k := s_{x_t}^{\prime k} - s_{x_t}^{\prime \prime k}$, from the above update, we obtain
    $$
    \begin{aligned}
    \begin{bmatrix}
    \mathbf{w}_{x_t}^{k+1} \\
    \mathbf{w}_{x_t}^{k}
    \end{bmatrix}= & \begin{bmatrix}
    (2-\theta)(I-\eta \nabla^2 \hat f_{x_t}(0)) & -(1-\theta)(I-\eta \nabla^2 \hat f_{x_t}(0)) \\
    I & 0
    \end{bmatrix} \begin{bmatrix}
    \mathbf{w}_{x_t}^{k} \\
    \mathbf{w}_{x_t}^{k-1}
    \end{bmatrix} \\
    & -\eta\begin{bmatrix}
    (2-\theta) \Delta_{x_t}^k \mathbf{w}_{x_t}^{k} -(1-\theta) \Delta_{x_t}^k \mathbf{w}_{x_t}^{k-1} + e_{x_t}(y_{x_t}^{\prime \prime k};\mu) - e_{x_t}(y_{x_t}^{\prime k};\mu) \\
    0
    \end{bmatrix},
    \end{aligned}
    $$
    where $\Delta_{x_t}^k = \int_0^1\left(\nabla^2 \hat f\left(\tau y_{x_t}^{\prime k} + (1-\tau) y_{x_t}^{\prime \prime k}\right)-\nabla^2 \hat f_{x_t}(0)\right) \mathrm{d} \tau$. For simplicity, let
    \begin{equation*}
        A_{x_t} = \begin{bmatrix}
    (2-\theta)(I-\eta \nabla^2 \hat f_{x_t}(0)) & -(1-\theta)(I-\eta \nabla^2 \hat f_{x_t}(0)) \\
    I & 0
    \end{bmatrix}
    \end{equation*}
    and $\phi_{x_t}^k = (2-\theta) \Delta_{x_t}^k \mathbf{w}_{x_t}^{k} -(1-\theta) \Delta_{x_t}^k \mathbf{w}_{x_t}^{k-1} + e_{x_t}(y_{x_t}^{\prime \prime k};\mu) - e_{x_t}(y_{x_t}^{\prime k};\mu)$, we further have
    \begin{equation}\label{eq.w_k update}
       \begin{bmatrix}
    \mathbf{w}_{x_t}^{k+1} \\
    \mathbf{w}_{x_t}^{k}
    \end{bmatrix} = A_{x_t} \begin{bmatrix}
    \mathbf{w}_{x_t}^{k} \\
    \mathbf{w}_{x_t}^{k-1}
    \end{bmatrix} - \eta \begin{bmatrix}
        \phi_{x_t}^k \\ 0
    \end{bmatrix} = A_{x_t}^{k+1} \begin{bmatrix}
    \mathbf{w}_{x_t}^{0} \\
    \mathbf{w}_{x_t}^{0}
    \end{bmatrix} - \eta \sum_{r=0}^k A_{x_t}^{k-r} \begin{bmatrix}
        \phi_{x_t}^r \\ 0
    \end{bmatrix},
    \end{equation}
    
    To proceed, we prove this lemma by contradiction. Assume that none of the iterates $s_{x_t}^{\prime 0}, s_{x_t}^{\prime 1}, \ldots, s_{x_t}^{\prime K}$ and $s_{x_t}^{\prime\prime 0}, s_{x_t}^{\prime\prime 1}, \ldots, s_{x_t}^{\prime\prime K}$ trigger the “if condition”, which implies that
    \begin{align*}
        &\|s_{x_t}^{\prime k} - s_{x_t}^{\prime 0}\| \le B, ~\|y_{x_t}^{\prime k} - s_{x_t}^{\prime 0}\| \le 2B, ~ k = 1,\ldots,K, \\
        &\|s_{x_t}^{\prime\prime k} - s_{x_t}^{\prime\prime 0}\| \le B, ~\|y_{x_t}^{\prime\prime k} - s_{x_t}^{\prime\prime 0}\| \le 2B, ~ k = 1,\ldots,K.
    \end{align*}
    Combining with the fact that $\|s_{x_t}^{\prime 0}\| \le r$, $\|s_{x_t}^{\prime \prime 0}\| \le r$ and $r \le B$, we have
    \begin{equation*}
        \|\Delta_{x_t}^k\| \le \rho \max\left\{\|y_{x_t}^{\prime k}\|, \|y_{x_t}^{\prime \prime k}\|\right\} \le \rho \max\left\{\|y_{x_t}^{\prime k} - s_{x_t}^{\prime 0}\| + \|s_{x_t}^{\prime 0}\|, \|y_{x_t}^{\prime \prime k} - s_{x_t}^{\prime\prime 0}\| + \|s_{x_t}^{\prime\prime 0}\|\right\} \le 3\rho B.
    \end{equation*}
    Consequently, the term $\|\phi_{x_t}^k\|$ can be upper bounded as
    \begin{equation*}
        \|\phi_{x_t}^k\| \le 2\|\Delta_{x_t}^k\|\|\mathbf{w}_{x_t}^k\| + \|\Delta_{x_t}^k\|\|\mathbf{w}_{x_t}^{k-1}\| + 2\mathbf{E}(\mu) \le 6\rho B(\|\mathbf{w}_{x_t}^k\| + \|\mathbf{w}_{x_t}^{k-1}\|) + 2\mathbf{E}(\mu).
    \end{equation*}
    From the update \eqref{eq.w_k update}, we see
    \begin{equation*}
        \mathbf{w}_{x_t}^k = [I \quad 0]A_{x_t}^k \begin{bmatrix}
    \mathbf{w}_{x_t}^{0} \\
    \mathbf{w}_{x_t}^{0}
    \end{bmatrix} - \eta [I \quad 0] \sum_{r=0}^{k-1} A_{x_t}^{k-1-r} \begin{bmatrix}
        \phi_{x_t}^r \\ 0
    \end{bmatrix}.
    \end{equation*}
    Next, we set up an induction on $k$ to show:
    \begin{equation*}
        \left\|\eta [I \quad 0] \sum_{r=0}^{k-1} A_{x_t}^{k-1-r} \begin{bmatrix}
        \phi_{x_t}^r \\ 0
    \end{bmatrix}\right\| \le \frac{1}{2}\left\|[I \quad 0]A_{x_t}^k \begin{bmatrix}
    \mathbf{w}_{x_t}^{0} \\
    \mathbf{w}_{x_t}^{0}
    \end{bmatrix}\right\|.
    \end{equation*}
    It is easy to check the base case holds for $k = 0$ since $\mathbf{E}(\mu) = \frac{\rho B \theta r_0}{2} \le 2\eta \rho B r_0$. Then, assume that for all iterations less than or equal to $k$, the induction assumption holds. We have
    \begin{equation*}
        \left\|\mathbf{w}_{x_t}^k\right\| = \left\|[I \quad 0]A_{x_t}^k \begin{bmatrix}
    \mathbf{w}_{x_t}^{0} \\
    \mathbf{w}_{x_t}^{0}
    \end{bmatrix} - \eta [I \quad 0] \sum_{r=0}^{k-1} A_{x_t}^{k-1-r} \begin{bmatrix}
        \phi_{x_t}^r \\ 0
    \end{bmatrix}\right\| \le 2\left\|[I \quad 0]A_{x_t}^k \begin{bmatrix}
    \mathbf{w}_{x_t}^{0} \\
    \mathbf{w}_{x_t}^{0}
    \end{bmatrix}\right\|,
    \end{equation*}
    and it further implies that
    \begin{equation*}
        \left\|\phi_{x_t}^k\right\| \le 12\rho B\left(\left\|[I \quad 0]A_{x_t}^k \begin{bmatrix}
    \mathbf{w}_{x_t}^{0} \\
    \mathbf{w}_{x_t}^{0}
    \end{bmatrix}\right\| + \left\|[I \quad 0]A_{x_t}^{k-1} \begin{bmatrix}
    \mathbf{w}_{x_t}^{0} \\
    \mathbf{w}_{x_t}^{0}
    \end{bmatrix}\right\|\right) + 2\mathbf{E}(\mu) \le 24\rho B\left\|[I \quad 0]A_{x_t}^k \begin{bmatrix}
    \mathbf{w}_{x_t}^{0} \\
    \mathbf{w}_{x_t}^{0}
    \end{bmatrix}\right\| + 2\mathbf{E}(\mu),
    \end{equation*}
    where the last inequality is due to the monotonicity of $\left\|[I \quad 0]A_{x_t}^k \begin{bmatrix}
    \mathbf{w}_{x_t}^{0} \\
    \mathbf{w}_{x_t}^{0}
    \end{bmatrix}\right\|$ in $k$ (Lemma 33 in \cite{jin2018accelerated}). For the case $k+1$, we have
    \begin{subequations}
        \begin{align}
            &\left\|\eta [I \quad 0] \sum_{r=0}^{k} A_{x_t}^{k-r} \begin{bmatrix}
        \phi_{x_t}^r \\ 0
    \end{bmatrix}\right\| \notag \\
    \le &\eta \sum_{r=0}^{k} \left\|[I \quad 0] A_{x_t}^{k-r} \begin{bmatrix}
        I \\ 0
    \end{bmatrix}\right\|\left\|\phi_{x_t}^r\right\| \notag \\
    \le &\eta \sum_{r=0}^{k} \left\|[I \quad 0] A_{x_t}^{k-r} \begin{bmatrix}
        I \\ 0
    \end{bmatrix}\right\|\left(24\rho B\left\|[I \quad 0]A_{x_t}^r \begin{bmatrix}
    \mathbf{w}_{x_t}^{0} \\
    \mathbf{w}_{x_t}^{0}
    \end{bmatrix}\right\| + 2\mathbf{E}(\mu)\right) \notag \\
    = &\eta \sum_{r=0}^k |a_{x_t}^{k-r}|\left(24\rho B|a_{x_t}^r - b_{x_t}^r|r_0 + 2\mathbf{E}(\mu)\right) \label{subeq.40a} \\
    \le &26\eta\rho B\sum_{r=0}^k |a_{x_t}^{k-r}||a_{x_t}^r - b_{x_t}^r|r_0 \label{subeq.40b} \\
    \le &26\eta\rho B\sum_{r=0}^k \left(\frac{2}{\theta} + k + 1\right)|a_{x_t}^{k+1} - b_{x_t}^{k+1}|r_0 \label{subeq.40c} \\
    \le &26\eta\rho B K\left(\frac{2}{\theta} + K\right)\left\|[I \quad 0]A_{x_t}^{k+1} \begin{bmatrix}
    \mathbf{w}_{x_t}^{0} \\
    \mathbf{w}_{x_t}^{0}
    \end{bmatrix}\right\| \notag \\
    \le &\frac{1}{2}\left\|[I \quad 0]A_{x_t}^{k+1} \begin{bmatrix}
    \mathbf{w}_{x_t}^{0} \\
    \mathbf{w}_{x_t}^{0}
    \end{bmatrix}\right\|, \label{subeq.40d}
    \end{align}
    \end{subequations}

    where we define $[a_{x_t}^k \quad -b_{x_t}^k] = [1 \quad 0] A_{x_t, \min}^k$ and 
    $$A_{x_t, \min} = \begin{bmatrix}
    (2-\theta)(1-\eta \lambda_{\min}(\nabla^2 \hat f_{x_t}(0))) & -(1-\theta)(I-\eta \lambda_{\min}(\nabla^2 \hat f_{x_t}(0))) \\
    I & 0
    \end{bmatrix}.$$ 
    Then, we apply the same argument in the proof of Lemma B.2 in \cite{li2022restarted} to the inequality \eqref{subeq.40a}, \eqref{subeq.40b} comes from $|a_{x_t}^r - b_{x_t}^r| \ge \frac{\theta}{2}$ (Lemma 38 in \cite{jin2018accelerated}) and $\mathbf{E}(\mu) \le \frac{\rho B \theta r_0}{2}$, and \eqref{subeq.40c} uses Lemma 31 in \cite{jin2018accelerated}. From the parameter settings, we have $26\eta\rho B K\left(\frac{2}{\theta} + K\right) \le \frac{1}{2}$ in \eqref{subeq.40d}. Therefore, the introduction is established, which yields
    \begin{equation*}
        \|\mathbf{w}_{x_t}^K\| \ge \frac{1}{2}\left\|[I \quad 0]A_{x_t}^{K} \begin{bmatrix}
    \mathbf{w}_{x_t}^{0} \\
    \mathbf{w}_{x_t}^{0}
    \end{bmatrix}\right\| = \frac{r_0}{2}|a_{x_t}^K - b_{x_t}^K| \ge \frac{\theta r_0}{4}\left(1+\frac{\theta}{2}\right)^K \ge 5B,
    \end{equation*}
    where we use Lemma 33 in \cite{jin2018accelerated}, $\eta \lambda_{\min}(\nabla^2 \hat f_{x_t}(0))) \le -\theta^2$ and $K \ge \frac{2}{\theta} \log \frac{20 B}{\theta r_0}$. However, for the term $\|\mathbf{w}_{x_t}^K\|$, it also holds that
    \begin{equation*}
        \|\mathbf{w}_{x_t}^K\| \le \|s_{x_t}^{\prime K} - s_{x_t}^{\prime 0}\| + \|s_{x_t}^{\prime 0} - s_{x_t}^{\prime\prime 0}\| + \|s_{x_t}^{\prime \prime K} - s_{x_t}^{\prime \prime 0}\| \le 4B,
    \end{equation*}
    which leads to a contradiction. Therefore, at least one of the iterates $s_{x_t}^{\prime 0}, s_{x_t}^{\prime 1}, \ldots, s_{x_t}^{\prime K}$ and $s_{x_t}^{\prime\prime 0}, s_{x_t}^{\prime\prime 1}, \ldots, s_{x_t}^{\prime\prime K}$ trigger the “if condition”.
    
\end{proof}
\section{Proofs of Non-asymptotic Convergence Analysis}
In this section, we prove that non-perturbed RAZGD with Option II converges to second-order stationary points asymptotically. It follows from that the tangent space step TSS locally avoids saddle points. To prove the local saddle avoidance, it is helpful to use the augmentation method to extend the update rule in the tangent space to a dynamical system of $s^{k+1}$ and $w^{k+1}$ that only depend on $s^k$ and $w^k$, i.e., regard $y^k$ as an intermediate variable. Despite we are interested in the zeroth-order method, the stability analysis of the zeroth-order algorithm heavily depends on the structure of its first-order counterpart. Therefore, we start with the analysis of the first-order tangent space step, which provides the second-order convergence immediately. 
\subsection{First-order tangent space step}
We use the augmentation method to re-write the tangent space step in the following way,
\begin{align}
    y^k&=s^k+(1-\theta)(s^k-w^k)
\\
s^{k+1}&=y^k-\eta g(y^k;\mu)
\\
w^{k+1}&=s^k
\end{align}
where $g(y;\mu)$ is the zeroth order approximation of the gradient $\nabla f(y)$ with smoothing parameter $\mu$. We will not emphasize that $g$ is performed at the point $x$ at this stage, just to reduce the complexity of notations. The three steps of the updating rule can be denoted by three mappings that consist of the mapping of the algorithm that updates $s^k,w^k$ to $s^{k+1},w^{k+1}$. We denote
\begin{align*}
    F_1(s,w)&=s+(1-\theta)(s-w)
    \\
    F_2(y)&=y-\eta g(y;\mu)
    \\
    F_3(s)&=s,
\end{align*}
and then the algorithm can be written compactly as
\[
\psi(s,w)=\left(F_2\circ F_1(s,w),F_3(s)\right)
\]
which is a mapping from $T_xM\times T_xM$ onto itself. The fixed point of the first order accelerated $(s^*,w^*)$ is necessarily a point such that $s^*=w^*$ and the gradient
\[
\nabla f(y^*)=\nabla f(s^*+(1-\theta)(s^*-w^*))=0.
\]
We will investigate the local structure of the zeroth order variant at the point $(s^*,w^*)$. The differential $D\psi(s^*,w^*)$ equals to
\begin{align}
D\psi(s^*,w^*)&=\left[
\begin{array}{c}
DF_2\circ DF_1(s^*,w^*)
\\
DF_3(s^*)
\end{array}
\right]
\end{align}
As an immediate result and important argument bridging first-order and zeroth-order accelerated gradient descent in the tangent space, we first prove that the first-order tangent space step avoids saddle points.
The following classic result of the stable manifold theorem will be used to complete the proof for the first-order method.
\begin{theorem}[\cite{shub}]\label{SMT}
    Let $p$ be a fixed point for the $C^r$ local diffeomorphism $h:U\rightarrow\mathbb{R}^n$ where $U\subset\mathbb{R}^n$ is an open neighborhood of $p$ in $\mathbb{R}^n$ and $r\ge 1$. Let $E^s\oplus E^c\oplus E^u$ be the invariant splitting of $\mathbb{R}^n$ into generalized eigenspaces of $Dh(p)$ corresponding to eigenvalues of absolute value less than one, equal to one, and greater than one. To the $Dh(p)$ invariant subspace $E^s\oplus E^c$ there is an associated local $h$ invariant embedded disc $W^{loc}_{sc}$ which is the graph of a $C^r$ function $r:E^s\oplus E^c\rightarrow E^u$, and ball $B$ around $p$ such that:
    $h(W^{loc}_{sc})\cap B\subset W^{loc}_{sc}$. If $h^n(x)\in B$ for all $n\ge 0$, then $x\in W^{loc}_{sc}$.   
\end{theorem}

\begin{lemma}
    Suppose that $0$ is a strict saddle point of the pullback function in the tangent space, then the measure of the local initial points that converge to $0$ is zero.
\end{lemma}
\begin{proof}
    The structure of $\psi$ gives the expression of its differential. Since 
\begin{align}
D F_2\circ D F_1&=(I-\eta\nabla^2f(y^*))((2-\theta)I,-(1-\theta)I)
\\
&=\left((2-\theta)(I-\eta\nabla^2f(y^*)),-(1-\theta)(I-\eta\nabla^2f(y^*))\right)
\end{align}
and 
\[
DF_3=\left(I,0\right),
\]
we have that
\[
D\psi(s^*,w^*)=\left[
\begin{array}{cc}
(2-\theta)(I-\eta\nabla^2f(y^*))&-(1-\theta)(I-\eta\nabla^2f(y^*))
\\
I&0
\end{array}
\right].
\]
Note that $D\psi$ is similar to 
\[
\left[
\begin{array}{cc}
(2-\theta)(I-\eta H)&-(1-\theta)(I-\eta H)
\\
I&0
\end{array}
\right]
\]
provided $\nabla^2f(y^*)$ is diagonalizable where $H$ is the diagonal matrix consisting of eigenvalues of $\nabla^2f(y^*)$. We can abuse the notation by
\begin{align}
\det(D\psi-\lambda I)&=\det\left(\left[
\begin{array}{cc}
(2-\theta)(I-\eta H)-\lambda I&-(1-\theta)(I-\eta H)
\\
I&-\lambda I
\end{array}
\right]\right)
\\
&=\det\left(((2-\theta)(I-\eta H)-\lambda I)+(1-\theta)(I-\eta H)\left(-\frac{1}{\lambda}I\right)\right)(-\lambda)^n
\\
&=\det\left(-\lambda((2-\theta)(I-\eta H)-\lambda I)+(1-\theta)(I-\eta H)\right)
\\
&=\det\left(\lambda^2I-\lambda(2-\theta)(I-\eta H)+(1-\theta)(I-\eta H)\right)
\end{align}
Since all matrices involved above are all diagonal, the determinant is nothing but the product of the following polynomials for $i\in[n]$:
\[
\lambda^2-(2-\theta)(1-\eta \lambda_i)\lambda+(1-\theta)(1-\eta\lambda_i).
\]
Suppose $\lambda_i$ is a negative eigenvalue (existence is guaranteed by assuming $y^*$ is a saddle point), the eigenvalue of $D\psi$ must contain the following one
\[
\lambda=\frac{(2-\theta)(1-\eta\lambda_i)+\sqrt{(2-\theta)^2(1-\eta\lambda_i)^2-4(1-\theta)(1-\eta\lambda_i)}}{2}.
\]
Since we can choose $\theta$ and $\eta$ so that
\[
\eta>\frac{\frac{2}{2-\theta}-1}{-\lambda_i}
\]
which guarantees that 
\[
(2-\theta)(1-\eta\lambda_i)>2,
\]
thus $\lambda>1$ (unstable fixed point). The step $\eta$ can be arbitrarily small (so that $\psi$ is a diffeomorphism) by taking $\theta$ as small as possible. Applying the center-stable manifold theorem \ref{SMT}, we complete the proof.
    
\end{proof}

\subsection{Zeroth-order tangent space step with constant contraction}
In this subsection, we show the asymptotic convergence for the zeroth-order tangent space step with constant contracting parameter $\beta$. 
The zeroth order tangent space step can be extended with the smoothing parameter to a new mapping $\tilde{\psi}(s,w,\mu)$ with a contraction factor $\beta$ as follows,
\begin{equation}\label{psi}
\tilde{\psi}(s,w,\mu)=\left(\psi(s,w,\mu),\beta\mu\right)
\end{equation}
where $\psi(s,w,\mu)$ considers the smoothing parameter as a proper variable such that $\psi$ is a mapping defined on $T_xM\times T_xM\times\mathbb{R}\rightarrow T_xM\times T_xM$. Note that the zeroth order approximation $g(y;\mu)$ may not provide a fixed point of the gradient descent, in order to asymptotically output a fixed point the gradient descent, it is necessary to contract the smoothing parameter so that the zeroth order approximation algorithm has the same set of fixed points as the gradient descent. Motivated by the zeroth order approximation scheme of \cite{GP2019}. The tangent space mapping requires a contracting smoothing parameter $\beta\mu$ for the whole tangent space step $\TSSA$. Another observation on the tangent space step $\TSSA$ from the asymptotic perspective, is the condition in the \textbf{while} loop. Since the asymptotic convergence empirically works well and is more convenient in the parameter settings, there is no need to use finite step $K$ in $\TSS$ mapping, but the condition $k\sum_{j=0}^k\norm{s_k^{j+1}-s_x^{j}}^2>B^2$ suffices to control the process of the \textbf{while} loop. The next lemma shows that the $\TSSA$ step is almost impossible to converge to a saddle point.

\begin{lemma}
    Consider mapping $\tilde{\psi}$ is defined as \eqref{psi}. The set of initial condition in the tangent space that converges to saddle point, i.e., $0$ in this setting, has measure zero.
\end{lemma}

\begin{proof}
    The differential of $\tilde{\psi}$ can be computed in the following way,
    \[
    D\tilde{\psi}=\left[
    \begin{array}{ccc}
       D_s\psi  & D_w\psi & D_{\mu}\psi 
       \\
       0  & 0 & \beta  
    \end{array}
    \right].
    \]
Recall that in the zeroth order approximation, $\psi$ is a mapping consisting of the approximated gradient $g(y;\mu)$, which is different from the first order method. The differential of $g(y;\mu)$ gives the differential of $\tilde{\psi}$ and $\psi$, so we compute $Dg(y;\mu)$ concretely, since $0$ is the only fixed point for the $\mu$ component, we need to compute the Taylor expansion at $(y^*,0)$ where $y^*$ is the fixed point of the first order counterpart of the algorithm. Thus, we have $D_{s,w}\psi(y^*,0)$ coincide with the differential computed in the first order method, and $D_{\mu}\psi$ is $(-\eta D_{\mu}g(y;\mu),0)^{\top}$, where
\begin{align}
    D_{\mu}g(y;\mu)=\left[
    \begin{array}{c}
    \frac{\partial g_1(y;\mu)}{\partial\mu}
    \\
    \vdots
    \\
    \frac{\partial g_d(y;\mu)}{\partial\mu}
    \end{array}
    \right]=\left[
    \begin{array}{c}
         \frac{\partial}{\partial\mu}\left(\frac{f(y+\mu e_1)-f(y)}{\mu}\right)
         \\
         \vdots
         \\
         \frac{\partial}{\partial\mu}\left(\frac{f(y+\mu e_d)-f(y)}{\mu}\right)
    \end{array}
    \right].
\end{align}
Since the block matrix $[D_s\psi,D_w\psi]$ computed at $(y^*,0)$ is the same as that has been computed in the first order method, and we have shown that the determinant of the block matrix is not zero, therefore, we are ready to obtain the determinant of $D\tilde{\psi}$ at the fixed point $(y^*,0)$. It is obvious that 
\[
\det\left(D\tilde{\psi}(s^*,w^*,0)\right)=\det\left(D\psi(s^*,w^*,0)\right)\cdot\beta
\]
and 
\[
\det\left(D\tilde{\psi}(s^*,w^*,0)-\lambda I\right)=\det\left(D\psi(s^*,w^*,0)-\lambda I\right)(\beta-\lambda).
\]
Based on the spectral analysis of the first order tangent step, we conclude that the escaping direction the zeroth order approximation tangent space step is provided by the unstable direction of the first order method. In the end, applying the stable manifold theorem \cite{shub}, we conclude that the set of initial points that converge to saddle point in the asymptotic variant of tangent space step is of measure zero since these initial points belong to a lower dimensional manifold.
\end{proof}

\subsection{Zeroth-order tangent space step with time-varying contraction}
We next prove the asymptotic saddle avoidance of the tangent space step $\TSSA$ when the smoothing parameter reduces in a slower rate, which is more practical from a zeroth-order perspective.
\begin{lemma}\label{slower}
Suppose $\TSSA$ is executed with the update rule on the smoothing parameter $\mu$ given by 
\[
\mu_{k+1}=\left(1-\frac{1}{k+2}\right)\mu_k
\],
then the probability of $\TSSA$ converging to a saddle point is zero.
\end{lemma}
\begin{proof}
The dynamical system augmented by $\mu$ is the following,
\begin{align}
y^k&=s^k+(1-\theta)(s^k-w^k)
\\
s^{k+1}&=y^k-\eta g(y^k;\mu_k)
\\
w^{k+1}&=s^k
\\
\mu_{k+1}&=(1-\frac{1}{k+2})\mu_k
\end{align}
which is an augmentation with the smoothing parameter $\mu$
Following the previous arguments, we can write the mapping on the parameters $(s,w,\mu)$ as follows,
\[
\tilde{\psi}_k(s,w,\mu)=\left(\psi(s,w,\mu),\left(1-\frac{1}{k+2}\right)\mu\right)
\]
and the differential of $\tilde{\psi}$ is
\[
D\tilde{\psi}_k=\left[
\begin{array}{ccc}
D_s\psi&D_w\psi&D_{\mu}\psi
\\
0&0&1-\frac{1}{k+2}
\end{array}
\right]
\]
Since the zeroth order method with contraction factor converges to stationary points of the corresponding first-order method, we can investigate the same Taylor expansion and differential of the algorithm expanded at the stationary point, especially at saddle point. The eigenvalues of the operator $D\tilde{\psi}(s^*,w^*,0)$ can be analyzed by the characteristic polynomial

\[
\det\left(D\tilde{\psi}_k(s^*,w^*,0)-\lambda I\right)=\det\left(D\psi(s^*,w^*,0)-\lambda I\right)\left(1-\frac{1}{k+2}-\lambda\right).
\]
Note that except for the eigenvalue $1-\frac{1}{k+2}$, all the other eigenvalues are the same as the case whose the contraction parameter is a constant $\beta$. Therefore, there is an $(2d+1)\times(2d+1)$ invertible matrix $C_k$ for each $k$ such that
\[
A_k=C_k^{-1}D\tilde{\psi}_k(s^*,w^*,0)C_k=\left[
\begin{array}{cc}
P_k&
\\
&Q_k
\end{array}
\right]
\]
where the eigenvalues $\lambda_1,...,\lambda_s$ of $P_k$ have magnitude less than 1, and the eigenvalues $\lambda_{s+1},...,\lambda_{2d+1}$ of the matrix $Q_k$ have magnitude greater than 1 (guaranteed by the property of a saddle point). Since the algorithm now is time dependent, i.e., the update rule $\tilde{\psi}_k$ contains a time dependent term $1-\frac{1}{k+2}$ and thus the Jordan block is also time dependent, the time independent argument that directly follows the stable manifold theorem is not valid in this time dependent setting. To show the same result as what holds for constant contraction case, we need to investigate the structure of the dynamical system in detail. Denote $A(m,n)$ the successive product of the $n$th till the $m$th matrices, i.e., $A(m,n)=A_m\cdot...\cdot A_n$. With the help of this notation, we can express the product 
\[
A(m,n)=\left[
\begin{array}{cc}
P_m...P_n&
\\
&Q_m...Q_n
\end{array}
\right]=\left[
\begin{array}{cc}
P(m,n)&
\\
&Q(m,n)
\end{array}
\right].
\]
Recall that the dynamical system induced by the tangent space step is
\[
\left[
\begin{array}{c}
s^{k+1}
\\
w^{k+1}
\\
\mu_{k+1}
\end{array}
\right]=\tilde{\psi}(s^k,w^k,\mu_k).
\]
Assuming that the saddle point is $(s^*,w^*,0)=(0,0,0)$, and the above dynamical system has the following expression obtained from the Taylor expansion around $(0,0,0)$,
\[
\left[
\begin{array}{c}
s^{k+1}
\\
w^{k+1}
\\
\mu_{k+1}
\end{array}
\right]=D\tilde{\psi}_k(0,0,0)\left[
\begin{array}{c}
s^k\\w^k\\ \mu_k
\end{array}
\right]+\xi_k(s^k,w^k,\mu_k)
\]
where $\xi_k(\cdot,\cdot,\cdot)$ is the remainder of $\tilde{\psi}_k$. Starting from the initial condition $(s^0,w^0,\mu_0)$, the dynamical system can be represented by 
\[
z_{k+1}=\left[
\begin{array}{cc}
P(k,0)&
\\
&Q(k,0)
\end{array}
\right]
z_0
+\sum_{i=0}^k\left[
\begin{array}{cc}
P(k,i+1)&
\\
&Q(k,i+1)
\end{array}
\right]\xi_i(z_i),
\]
where $z_k$ is the dynamical system topologically conjugated to $(s^k,w^k,\mu_k)$. Splitting $z_k$ and $\xi_i(z_i)$ into contracting and expanding components according to $P(k,0)$ and $Q(k,0)$, i.e., this decomposition is actually based on the magnitudes of the eigenvalues of $D\tilde{\psi}_k$ which is determined by the Hessian of the objective function $f$ at saddle points. Further information of the Jordan matrix $P_k$ and $Q_k$ can be inferred. The expanding matrix $Q_k$ contains only constant eigenvalues with magnitude greater than 1. The contracting matrix $P_k$ contains constant eigenvalues and one eigenvalue that is exactly $1-\frac{1}{k+2}$. The stable-unstable decomposition of $z_k$ can be further refined into stable with constant eigenvalues less than 1, stable with eigenvalue $1-\frac{1}{k+2}$, and unstable with constant eigenvalues greater than 1. Specifically, we decompose $z_k$ into
\[
z_k=\left[
\begin{array}{c}
z_k^+\\ z_k^{\mu} \\ z_k^-
\end{array}
\right],
\]
and $\xi_i(z_i)$ into
\[
\xi_i(z_i)=\left[
\begin{array}{c}
\xi_i^+(z_i) \\ 0\\ \xi_i^-(z_i)
\end{array}
\right]
\]
where the remainder with respect to $z_k^{\mu}$ is zero because the update rule of $\mu$ is a linear function. Based on this decomposition, we can refine the formulation of the dynamical system of $z_k$ in the following way,
\begin{align*}
z_{k+1}^+&=P(k,0)z_0^++\sum_{i=0}^kP(k,i+1)\xi_i^+(z_i)
\\
z_{k+1}^{\mu}&=\left(1-\frac{1}{k+2}\right)z_{k}^{\mu}
\\
z_{k+1}^-&=Q(k,0)z_0^-+\sum_{i=0}^kQ(k,i+1)\xi_i^-(z_i)
\end{align*}
where we still use $P$ as the Jordan block of stable component without distinguishing from the one containing $1-\frac{1}{k+2}$. Letting $k\rightarrow\infty$, we have formally the unstable component $z_0^-$ of the initial condition $z_0$ satisfying
\[
z_0^-=-\sum_{i=1}^{\infty}Q(i-1,0)^{-1}\xi_{i-1}^-(z_{i-1}),
\] 
and then the updated term $z_{k+1}$ can be written as
\begin{align*}
z_{k+1}&=z_{k+1}^+\oplus z_{k+1}^{\mu}\oplus z_{k+1}^-
\\
&=\left(P(k,0)z_0^++\sum_{i=0}^kP(k,i+1)\xi_i^+(z_i)\right)\oplus \left(1-\frac{1}{k+2}\right)z_k^{\mu}\oplus\left(Q(k,0)z_0^-+\sum_{i=0}^kQ(k,i+1)\xi_i^-(z_i)\right)
\end{align*}
where the last summand can be further written as
\[
-\sum_{i=0}^{\infty}Q(k+1+i,k+1)^{-1}\xi_{k+1+i}^-(z_{k+1+i}).
\]
The update rule can be understood as an operator acting on the space of bounded sequences converging to zero. Since $P(k,0)$ and $Q(k,0)$ are matrices only involving constant eigenvalues, there exists constants $K_1, K_2<1$ such that
\begin{align}
\norm{P(m,n)}_2&\le K_1^{m-n+1}
\\
\norm{Q(m,n)^{-1}}_2&\le K_2^{m-n+1}.
\end{align}
The Lyapunov-Perron argument \cite{PPW19} asserts that there exists a small neighborhood around the saddle point, such that $T$ is an contraction map on the space of sequences converging to zero, and consequently, the initial point that can be carried to the saddle point (the zero) by the algorithm must lie on a lower dimensional manifold. To make this point precise, we investigate the norm of the difference of two sequences $T$ acting on. Let $u=\{u_n\}_{n\in\mathbb{N}}$ and $v=\{v_n\}_{n\in \mathbb{N}}$,

\begin{align}
(Tu-Tv)_{k+1}&=(Tu)_{k+1}-(Tv)_{k+1}
\\
&=\left(Q(k,0)(u_0^+-v_0^+)+\sum_{i=0}^kP(t,i+1)(\xi_i^+(u_i)-\xi_i^+(v_i))\right)
\\
&\oplus \frac{1}{3(k+2)}\left(u_0^{\mu}-v_0^{\mu}\right)
\\
&\oplus\left(-\sum_{i=0}^{\infty}Q(k+1+i,k+1)^{-1}(\xi_{k+1+i}^{-}(u_{k+1+i})-\xi_{k+1+i}^-(v_{k+1+i}))\right)
\end{align}
where the coefficient of the middle component comes from the product
\[
\frac{1}{3(k+2)}=\prod_{i=0}^k\left(1-\frac{1}{i+2}\right).
\]
Let $d(u,v)$ be the metric defined by the supremum norm of the sequence $\{u_i-v_i\}_{i\in\mathbb{N}}$. Since it has been proven by \cite{feng2022accelerated} that $T$ is a contracting map without the component of $\frac{1}{3(k+2)}(u_0^{\mu}-v_0^{\mu})$, i.e., there exists a constant $K<1$ such that
\[
d(Tu,Tv)\le Kd(u,v),
\]
and $\frac{1}{3(k+2)}\le\frac{1}{6}<1$, it guarantees a new constant $K'<1$, so that $T$ acting on the space of the considered sequence with $\mu$-component is an contracting map. Thus, the existence and uniqueness of the stable manifold in a neighborhood of the saddle point follow from the existence and uniqueness of the fixed point of $T$. So the probability of the initial condition lying on such lower dimensional manifold so that the iterates converge to saddle point is zeor.
\end{proof}

Now we are able to finalize the proof of Theorem \ref{thm.asymptotic}.
\begin{proof}[proof of Theorem \ref{thm.asymptotic}] It has been established that the probability of $\TSSA$ staying in a neighborhood of a saddle point is zero, for any $\TSSA$ stage. 
\[
\text{Pr}\left\{\lim_{k\rightarrow\infty}k\sum_{j=0}^{k-1}\norm{s_x^{j+1}-s_x^j}^2\le B^2\right\}=0
\]
and then the probability for the iterations to stay in a neighborhood of a second-order stationary point is 1. Since the zeroth-order acceleration with contracting parameter $\beta<1$ converges to stationary point, it follows that the probability for $\TSSA$ output a second-order stationary point is 1. Together with the above Lemma \ref{slower} for the case when the contracting parameter decreases in a slower manner (which slows the decreasing of smoothing parameter $\mu$ in the $\TSSA$ stage), we complete the proof of the theorem.

\end{proof}
\section{Implementation of RZGD and PZGD}\label{sec.RZOGD}
For Riemannian zeroth-order gradient descent (RZGD), it iteratively utilizes the Riemannian zeroth-order gradient descent step (Subroutine \ref{sub.RZGD}) until convergence. In the case of Euclidean projected zeroth-order gradient descent (PZGD), we first compute the Euclidean zeroth-order estimator (denoted by $g_{\operatorname{E}}(\cdot)$), take a Euclidean zeroth-order gradient descent step, and then project onto the Riemannian manifold. The pseudocodes of both algorithms are presented below.

\floatname{algorithm}{Algorithm}
\setcounter{algorithm}{1}
\begin{multicols}{2}
\begin{algorithm}[H]
   \caption{\textbf{R}iemannian \textbf{Z}eroth-order \textbf{G}radient \textbf{D}escent Algorithm (\textbf{RZGD})} \label{alg.RZGDA}
\begin{onehalfspace} %
\begin{algorithmic}[1]
  \STATE \textbf{input:} parameters $\eta$, and $B$
      \STATE\textbf{initialize:} $x_0 \in \M$, $t = 0$
      \FOR{$t = 0, 1, \cdots, \infty$}
        \STATE Compute estimator $g_{x_t}(0;\mu)$
        \IF{$\|g_{x_t}(0; \mu)\| \ge lB$}
            \STATE $x_{t+1} = \textbf{RZGDS}(x_t,\eta,g_{x_t}(0;\mu))$
        \ELSE
            \STATE Terminate with $x_t$
        \ENDIF
      \ENDFOR
\end{algorithmic}
\end{onehalfspace}
\end{algorithm}

\begin{algorithm}[H]
   \caption{Euclidean \textbf{P}rojected \textbf{Z}eroth-order \textbf{G}radient \textbf{D}escent Algorithm (\textbf{PZGD})}
\begin{onehalfspace} %
\begin{algorithmic}[1]
  \STATE \textbf{input:} parameters $\eta_t$
      \STATE\textbf{initialize:} $x_0 \in \M$, $t = 0$
      \FOR{$t = 0, 1, \cdots, \infty$}
        \STATE Compute Euclidean estimator $g_{\operatorname{E}}(x_t)$
        \STATE $x_{t+1} = \operatorname{proj}_{\M}\left(x_t - \eta_t g_{\operatorname{E}}(x_t)\right)$
      \ENDFOR
\end{algorithmic}
\end{onehalfspace}
\end{algorithm}
\end{multicols}

For completeness, we establish the function query complexity of RZGD, which serves as a benchmark for demonstrating the acceleration achieved by our RAZGD.
\begin{theorem}
    Suppose that Assumptions \ref{assum.lipschitz gradient}, \ref{assum.lipschitz hessian} and \ref{assum.retraction} hold. Set parameters in Algorithm \ref{alg.RZGDA} as follows
    \begin{equation}
        \eta = \frac{1}{4l}, \ B=\frac{\epsilon}{2l}. \notag
    \end{equation}
     For any $x_0 \in \M$ and sufficiently small $\epsilon > 0$, choose $\mu = \mathcal{O}\left(\frac{\sqrt{\epsilon}}{d^{1/4}}\right)$ in Line 4 of Algorithm \ref{alg.RZGDA}. Then Algorithm \ref{alg.RZGDA} outputs an $\epsilon$-approximate first-order stationary point. The total number of function value evaluations is no more than
    \begin{equation}
        O\left(\frac{(f(x_0) - f_{\operatorname{low}})d}{\epsilon^{2}}\right). \notag
    \end{equation}  
\end{theorem}
\begin{proof}
    Recall the approximation error of the Riemannian coordinate-wise zeroth-order estimator (Lemma \ref{lemma.zero order estimator error}), it holds that
    \begin{equation*}
        \|g_{x_t}(0;\mu) - \nabla \hat f_{x_t}(0)\| \le \frac{\epsilon}{4} = \frac{lB}{2}
    \end{equation*}
    by setting $\mu = \mathcal{O}\left(\frac{\sqrt{\epsilon}}{d^{1/4}}\right)$. For the scenario where $\|g_{x_t}(0; \mu)\| \ge lB$ holds, Lemma \ref{lemma.function value decrease large gradient} gives
    \begin{equation*}
         f(x_{t+1}) - f(x_t) \le - \min\left\{\frac{lB^2}{16}, lb^2\right\} = -\frac{\epsilon^2}{64l}.
    \end{equation*}
    Otherwise, we have
    \begin{equation*}
        \|\grad f(x_t)\| = \|\nabla \hat f_{x_t}(0)\| \le \|g_{x_t}(0;\mu) - \nabla \hat f_{x_t}(0)\| + \|g_{x_t}(0;\mu)\| \le \frac{3}{4}\epsilon,
    \end{equation*}
    where the first equality holds as $T_{x_t, 0}$ is identity. Therefore, as computing the zeroth-order estimator once requires $2d$ function value evaluations, the total number of function value evaluations must be less than
    \begin{equation*}
        O\left(\frac{(f(x_0) - f_{\operatorname{low}})d}{\epsilon^2}\right).
    \end{equation*}
\end{proof}
\section{Riemannian Geometry of the Simplex}\label{GometryofSimplex}
The Riemannian geometry of the positive orthant $\mathbb{R}_+^d=\{x:x_i>0 \text{ for all }i\in[d]\}$ was studied by researchers from mathematical biology and evolutionary game theory \cite{Shahshahani,MertiRiemannGame}. For completeness, this section provides missing details of calculation based on the Riemannian geometry of positive orthant and simplex in the experiment. 
Formally the positive orthant is $\mathbb{R}_+^d$ is endowed with a Riemannian metric whose metric matrix $\{g_{ij}(x)\}$ is diagonal with $g_{ii}(x)=\frac{\abs{x}}{x_i}$ where $\abs{x}=\sum_{j=1}^dx_j$, i.e.,
\[
g(x)=\left[
\begin{array}{ccc}
\frac{\abs{x}}{x_1}&&0
\\
&\ddots&
\\
0&&\frac{\abs{x}}{x_d}
\end{array}
\right]
\]
$\mathbb{R}_+^d$ is a single chart manifold with a non-Euclidean structure. 
To compute the pullback function $\hat f_x = f \circ \retr_x$ on the unit simplex, we introduce the exponential map on the Shahshahani manifold as the retraction. Given a point $x \in\Delta^{d-1}$ and a vector $s \in T_x\Delta^{d-1}$, the exponential map is
\[
\Exp_x(s)=\left(\frac{x_1e^{s_1}}{\sum_jx_je^{s_j}},...,\frac{x_de^{s_d}}{\sum_jx_je^{s_j}}\right)^{\top} \in \mathbb{R}^d.
\]


\end{document}